\documentclass{article}
\usepackage{amsmath,amsthm,amsfonts,amssymb,amscd,latexsym, fullpage,textcomp,amscd}
\begin{document}


\theoremstyle{plain}
\newtheorem{thm}{\bf Theorem}[section]
\newtheorem{lem}[thm]{\bf Lemma}
\newtheorem{cor}[thm]{\bf Corollary}
\newtheorem{prop}[thm]{\bf Proposition}
\newtheorem{notation}[thm]{\bf Notation}
\theoremstyle{remark}
\newtheorem{Case}{\bf Case}
\newtheorem{rem}[thm]{\bf Remark}
\newtheorem{claim}[thm]{\bf Claim}

\theoremstyle{definition}
\newtheorem{Def}{\bf Definition}
\newtheorem*{pf}{\bf Proof}
\newtheorem{Conj}{\bf Conjecture}

\newcommand{\nc}{\newcommand}
\newcommand{\rc}{\renewcommand}

\newcommand{\ca}{{\mathcal A}}
\newcommand{\BB}{{\mathcal B}} 
\newcommand{\CC}{{\mathcal C}}
\newcommand{\DD}{{\mathcal D}}
\newcommand{\EE}{{\mathcal E}}
\newcommand{\FF}{{\mathcal F}}
\newcommand{\GG}{{\mathcal G}}
\newcommand{\HH}{{\mathcal H}}
\newcommand{\II}{{\mathcal I}}
\newcommand{\JJ}{{\mathcal J}}
\newcommand{\KK}{{\mathcal K}}
\newcommand{\LL}{{\mathcal L}}
\newcommand{\MM}{{\mathcal M}}
\newcommand{\NN}{{\mathcal N}}
\newcommand{\OO}{{\mathcal O}}
\newcommand{\PP}{{\mathcal P}}
\newcommand{\QQ}{{\mathcal Q}}
\newcommand{\RR}{{\mathcal R}}
\newcommand{\TT}{{\mathcal T}}
\newcommand{\UU}{{\mathcal U}}
\newcommand{\VV}{{\mathcal V}}
\newcommand{\WW}{{\mathcal W}}
\newcommand{\ZZ}{{\mathcal Z}}
\newcommand{\XX}{{\mathcal X}}
\newcommand{\YY}{{\mathcal Y}}
\nc{\bba}{{\mathbb A}}
\nc{\bbb}{{\mathbb B}}
\nc{\bbc}{{\mathbb C}}
\nc{\bbd}{{\mathbb D}}
\nc{\bbe}{{\mathbb E}}
\nc{\bbf}{{\mathbb F}}
\nc{\bbg}{{\mathbb G}}
\nc{\bbh}{{\mathbb H}}
\nc{\bbi}{{\mathbb I}}
\nc{\bbj}{{\mathbb J}}
\nc{\bbk}{{\mathbb K}}
\nc{\bbl}{{\mathbb L}}
\nc{\bbm}{{\mathbb M}}
\newcommand{\N}{{\mathbb N}}
\nc{\bbo}{{\mathbb O}}
\nc{\bbp}{{\mathbb P}}
\nc{\bbq}{{\mathbb Q}}
\nc{\bbr}{{\mathbb R}}
\nc{\bbs}{{\mathbb S}}
\newcommand{\bb}{{\mathbb T}}
\nc{\bbu}{{\mathbb U}}
\nc{\bbv}{{\mathbb V}}
\nc{\bbw}{{\mathbb W}}
\nc{\bbx}{{\mathbb X}}
\nc{\bby}{{\mathbb Y}}
\nc{\bbz}{{\mathbb Z}}
\nc{\fA}{{\mathfrak A}}
\nc{\fB}{{\mathfrak B}}
\nc{\fC}{{\mathfrak C}}
\nc{\fD}{{\mathfrak D}}
\nc{\fE}{{\mathfrak E}}
\nc{\fF}{{\mathfrak F}}
\nc{\fG}{{\mathfrak G}}
\nc{\fH}{{\mathfrak H}}
\nc{\fI}{{\mathfrak I}}
\nc{\fJ}{{\mathfrak J}}
\nc{\fK}{{\mathfrak K}}
\nc{\fL}{{\mathfrak L}}
\nc{\fM}{{\mathfrak M}}
\nc{\fN}{{\mathfrak N}}
\nc{\fO}{{\mathfrak O}}
\nc{\fP}{{\mathfrak P}}
\nc{\fQ}{{\mathfrak Q}}
\nc{\fR}{{\mathfrak R}}
\nc{\fS}{{\mathfrak S}}
\nc{\fT}{{\mathfrak T}}
\nc{\fU}{{\mathfrak U}}
\nc{\fV}{{\mathfrak V}}
\nc{\fW}{{\mathfrak W}}
\nc{\fZ}{{\mathfrak Z}}
\nc{\fX}{{\mathfrak X}}
\nc{\fY}{{\mathfrak Y}}
\nc{\fa}{{\mathfrak a}}
\nc{\fb}{{\mathfrak b}}
\nc{\fc}{{\mathfrak c}}
\nc{\fd}{{\mathfrak d}}
\nc{\fe}{{\mathfrak e}}
\nc{\ff}{{\mathfrak f}}
\nc{\fh}{{\mathfrak h}}
\nc{\fj}{{\mathfrak j}}
\nc{\fk}{{\mathfrak k}}
\nc{\fl}{{\mathfrak{l}}}
\nc{\fm}{{\mathfrak m}}
\nc{\fn}{{\mathfrak n}}
\nc{\fo}{{\mathfrak o}}
\nc{\fp}{{\mathfrak p}}
\nc{\fq}{{\mathfrak q}}
\nc{\fr}{{\mathfrak r}}
\nc{\fs}{{\mathfrak s}}
\nc{\ft}{{\mathfrak t}}
\nc{\fu}{{\mathfrak u}}
\nc{\fv}{{\mathfrak v}}
\nc{\fw}{{\mathfrak w}}
\nc{\fz}{{\mathfrak z}}
\nc{\fx}{{\mathfrak x}}
\nc{\fy}{{\mathfrak y}}

\nc{\al}{{\alpha }}
\nc{\be}{{\beta }}
\nc{\ga}{{\gamma }}
\nc{\de}{{\delta }}
\nc{\ep}{{\varepsilon }}
\nc{\vap}{{\tepsilon }}

\nc{\ze}{{\zeta }}
\nc{\et}{{\eta }}
\nc{\vth}{{\vartheta }}

\nc{\io}{{\iota }}
\nc{\ka}{{\kappa }}
\nc{\la}{{\lambda }}
\nc{\vpi}{{     \varpi          }}
\nc{\vrho}{{    \varrho         }}
\nc{\si}{{      \sigma          }}
\nc{\ups}{{     \upsilon        }}
\nc{\vphi}{{    \varphi         }}
\nc{\om}{{      \omega          }}

\nc{\Ga}{{\Gamma }}
\nc{\De}{{\Delta }}
\nc{\nab}{{\nabla}}
\nc{\Th}{{\Theta }}
\nc{\La}{{\Lambda }}
\nc{\Si}{{\Sigma }}
\nc{\Ups}{{\Upsilon }}
\nc{\Om}{{\Omega }}

\newcommand{\inv}{^{-1}}
\newcommand{\noi}{\noindent}
\newcommand{\lra}{{\longrightarrow}}
\nc{\st}{{\; | \; }}
\nc{\trm}{\textreferencemark}
\nc{\imo}{{i-1}}
\nc{\imob}{ _\imo}
\nc{\imot}{ ^\imo }
\nc{\ben}{	\begin{enumerate}\item		}
\nc{\een}{	\end{enumerate}			}
\nc{\bi}{    \begin{itemize}\item    }
\nc{\ei}{    \end{itemize}   }
\rc{\i}{  \item }
\nc{\notat}{\noi {\bf Notation. }}

\nc{\ps }{Pre_{sSet}}
\nc{\pa}{Pre _{sAb}}
\nc{\tps}{\tilde Pre_{sSet}}
\nc{\tpa}{\tilde Pre_{sAb}}
\nc{\csha}{\check {Sh}_{sAb}}
\nc{\cshs}{\check {Sh}_{sSet}}
\nc{\shs}{Sh_{sSet}}
\nc{\sha}{Sh_{sAb}}
\nc{\psc}{{\check {Sh}_{sAb}}}
\nc{\psl}{\ps ^{loc}}
\nc{\cna}{\check{N\ca}}
\nc{\h}{hom(\YY , \cna)}
\nc{\cnao}{\check {N \ca [1] ^{a,b}}}

\nc{\tih}{{\tilde \HH}}
\nc{\vb}{\text{vector bundle}}
\nc{\vbm}{{\VV ect \BB und _M}}
\nc{\vs}{{Vect_{fd}}}
\nc{\cxn}{\text{connection}}
\nc{\bla}{{\mathfrak g ^* \rtimes \mathfrak g}}
\nc{\gc}{\text{generalized complex }}
\nc{\gcstr}{\text{generalized complex structure }}
\nc{\gcstrs}{\text{generalized complex structures }}
\nc{\spic}{\text{Picard $\omega$-category }}
\nc{\spics}{ { Pic_{\om} }}
\nc{\cab}{{Ch^{+}( A b)}}
\nc{\qis}{\text{quasi-isomorphism }}
\nc{\uu}{\underline u}
\nc{\vv}{\underline v}
\nc{\ww}{\underline w}
\nc{\td}{\tilde \De}
\nc{\lla}{{\longleftarrow}}

\begin{title}{Abelian groups in $\omega$-categories}
\end{title}
\author{Brett Milburn\footnote{milburn@math.utexas.edu}}
\date{}
\maketitle

\begin{abstract}  We study abelian group objects in $\om$-categories and discuss the well-known Dold-Kan correspondence from the perspective of $\om$-categories as a model for strict $\infty$-categories.  The first part of the paper is intended to compile results from the existing literature and to fill some gaps therein.  We go on to consider a parameterized Dold-Kan correspondence, i.e. a Dold-Kan correspondence for presheaves of $\om$-categories.  The main result is to describe the descent or sheaf condition in terms of a glueing condition that is familiar for 1 and 2-stacks.  
\end{abstract}

\tableofcontents 

\section{Introduction}

Our goal is to investigate abelian group objects in $\infty$-categories.  There are many notions of $\infty$-category.  Lurie \cite{lur} and Leinster \cite{lei} provide good--though not exhaustive--surveys of various definitions.  Other models such as complete Segal spaces \cite{lur2} and crossed complexes \cite{brh} also appear in the literature.  The approach taken in this paper is to consider $\om$-categories as strict $\infty$-categories.  \\

Roughly, an $\om$-category coincides with the intuitive description of an $\infty$-category.  It has objects and n-morphisms for $n\geq 1$.  One can compose n-morphisms and take the k-th source or target of an n-morphism to get a k-morphism.   In contrast, quasicategories (simplicial sets which admit fillers for inner horns) are also a model for $\infty$-categories.  While simplicial sets are useful from the perspective of homotopy theory, they are not endowed with all of the desired structure that one would like for an $\infty$-category.  Namely, there is no natural choice for identity morphisms or composition.  From this perspective, it is useful to consider $\om$-categories, which have the advantage of not having the same deficits.  Furthermore, $\om$-categories enjoy many nice properties.  For example, n-categories are easily defined.  However,  $\om$-categories are strict $\infty$-categories in the sense that composition is associative on the nose, and they do not contain the coherence data that one might desire for weak $\infty$-categories (although weak $\om$-categories have also been studied \cite{str4, lei}).  Presently our interest is in abelian group objects in $\infty$-categories.  We will see that abelian group objects in simplicial sets are in fact strict $\infty$-categories, and we may therefore interpret them as $\om$-categories.   \\

Sections 1-4 are primarily expository.  We begin by defining $\om$-categories in \S \ref{omcatssection} and introducing the notion of equivalence of $\om$-categories.  In \S \ref{piccatssection}, we formulate and prove the well-known statement that abelian groups objects in $\infty$-categories are the same as chain complexes of abelian groups in non-negative degrees.  Furthermore, we show that this equivalence induces a derived equivalence.  Two generalizations are pursued.  Firstly, we introduce the notion of an $I$-category for any partially ordered set $I$.  Of particular interest are the $\bbz $-categories.  Abelian group objects in $\bbz$-categories are equivalent to chain complexes of abelian groups.  The analogue of the correspondence between $\om$-categories and simplicial sets is now between $\bbz$-categories and combinatorial spectra (cf. \cite{jarch}), though this is not made precise here.  Similarly, the relationship between chain complexes and $\bbz$-categories is also analogous to the Quillen equivalence between chain complexes of abelian groups and $H\bbz$-module spectra described by Schwede and Shipley \cite{scs,ss2}.  Secondly, we observe that since an $\om$-category is determined by a set equipped with some structure maps, we may think of an $\om$-category as an $\om$-category in $Sets$.  This can be extended to define an $\om$-category in an arbitrary category with fibered products.  We generalize the equivalence between Chain complexes $Ch^+(Ab)$ of abelian groups in non-negative degree and $Pic_\om$, abelian group objects in $\om$-categories.  For any abelian category $\CC$ with countable direct sums, we show that $Ch^+(\CC)$ is equivalent to $\CC _\om$, $\om$-categories in $\CC$.  \\

Simplicial abelian groups, denoted $sAb$, can also be thought of as abelian group objects in $\infty$-categories.  The Dold-Kan correspondence \cite{dk} states that there is an equivalence $Ch^+(Ab) \simeq sAb$.  In section \S \ref{dksection} we cite a recent result of Brown, Higgins, and Sivera \cite{bro} that relates the Dold-Kan correspondence to the equivalence $Ch^+ (Ab) \simeq Pic _\om$.  To put it succinctly, there is a nerve functor, due to Street \cite{str}, $N: \om Cat \lra sSet$ from $\om$-categories to simplicial sets, which when restricted to $Pic _\om$ gives $N : Pic _\om \lra sAb$.  The Dold-Kan equivalence $Ch^+(Ab) \lra sAb$ is, up to isomorphism, the composition of $Ch^+(Ab) \lra Pic _\om $ with the nerve functor.  Each of $Ch^+(Ab)$, $Pic _\om $, and $sAb$ are naturally equivalent not just as categories but also as model categories.  \\

The core of the paper is in \S \ref{ommdescentsection}, where we consider descent for presheaves of $\om$-categories.  Descent for $\om$-categories has been considered by Street in \cite{str2,str4}, and Verity showed \cite{ver2} that Street's definition of descent is equivalent to the standard notion of descent for presheaves of simplicial sets, where the two are related by the nerve functor.  The Dold-Kan correspondence extends to presheaves with values in $Ch^+(Ab)$, $Pic _\om $, or $sAb$.  We consider several model structures on the presheaf categories, in particular one where the fibrant objects are precisely the sheaves (i.e. those satisfying descent with respect to all hypercovers) and one where the fibrant objects are those satisfying \v Cech descent (i.e. descent with respect to open covers).  We show that the homotopy category of simplicial sheaves of abelian groups on a space $X$ satisfying descent is equivalent to the derived category in non-negative degrees $D^{\geq 0} (\ca b)$ of sheaves of abelian groups on $X$.  \\ 

The key result is Theorem \ref{equivdescent} which states that a presheaf of simplicial abelian groups on site $\mathcal S$ satisfies \v Cech descent if and only if it satisfies a more concrete glueing condition, which can be explained roughly as being able to glue objects and n-morphism from local sections.  In more detail, a presheaf $\ca$ of simplicial abelian groups, $\ca $ satisfies \v Cech descent if and only if for every $X\in \mathcal S$ and open cover $\UU = \{ U_i\}_{i \in I}$ of $X$, 
\ben given local objects $x_i \in \ca (U_i)_0$ which are glued together by 1-morphisms and higher degree morphisms in a coherent way, there exists a global object $x\in \ca (X)$, unique up to isomorphism, which glues the $x_i$, and 
\i for any n-morphisms $x,y \in \ca (X)_n$, the presheaf $Hom_\ca(x,y)$ whose objects are the $(n+1)$-morphisms from $x$ to $y$ satisfies the above glueing condition for objects.  
\een 

This can be interpreted as providing a computational tool for determining whether a presheaf satisfies \v Cech descent or a way of constructing a sheafification of a given presheaf.  We hope that this has applications in the study of n-gerbes.  Let $G$ be an abelian group and $X$ a topological space.  With the appropriate notion of torsor, one may view an n-gerbe for $G$ on $X$ as a torsor for a presheaf of simplicial abelian groups, namely it is generated by $G$ in degree $n$ and $0$ elsewhere.  Given the computational descent condition, it should become apparent that isomorphism classes of $n$-gerbes for $G$ are given by $\check H^n (X,G)$.  In this way, this paper is a step towards viewing sheaves of $\infty$-categories as geometric realizations of cohomology classes.  This point of view is further explained in the final section, where we draw on the insights of Fiorenza, Sati, Schreiber, and Stasheff \cite{fss,sss,sch2} to describe torsors for sheaves valued in $\infty$-groups.

\section{$\om$-categories}\label{omcatssection}

We begin by defining $\om$-categories, which are a model for strict $\infty$-categories. 

\begin{Def}~\label{16nov1} The data for an \emph{$\om$-category} is a set ${A}$ with maps $s_i , t_i : {A} \lra {A} $ for $i \in \N$ and maps $*_i : {A} \times _{A} {A} \lra {A}$, where  ${A} \times _{A} {A} \lra {A}$ is the fibered product, given maps $s_i : {A} \lra {A} $ and $t_i : {A} \lra {A}$.  Let $\rho _i ,\si _i \in \{s_i , t_i \}$ denote any source or target map. \\

\noi $({A} , s_i , t_i , *_i )_{i \in \N}$ is said to be an $\om$-category if the following 3 conditions are satisfied:
\begin{enumerate} 
\item For all $i \in \N$, $({A} , s_i , t_i , *_i)$ is a category.  In other words, 
\begin{enumerate}
\item $\rho _i \sigma _i = \sigma _i$ 
\item $a *_i s_i(a) = t_i (a) *_i a = a$
\item $(a*_i b)*_i c = a *_i (b *_i c)$
\item $s_i (a*_ib) = s_i b$, and $t_i (a*_i b) = t_i a$
\end{enumerate}
\item For all $i < j$,$( A _i
, A_j)$ is a strict 2-category.  That is,
\begin{enumerate}
\item $\rho _j \sigma _i = \sigma _i$
\item $\sigma _i \rho _j = \sigma _i$
\item $\rho _j (a *_i b) = \rho _j a *_i \rho _j b$
\item $(a*_jb)*_i(\alpha *_j \beta) = (a*_i\alpha) *_j (b*_i \beta)$ whenever
both sides are defined.  
\end{enumerate}
\item For all $a \in {A}$, there is some $i \in \N$ such that $s_i a = t_i a = a$. 
\end{enumerate}

\end{Def}

\begin{Def}\label{homkdef} 
\ben For an $\om$-category ${A} $ and $i \in \N$, \emph{i-objects} in ${A}$ are ${A} _i : = s_i {A}$, and \emph{strict i-objects} are ${A} _i \setminus {A} _{i-1}$.  In conforming to convention, we also refer to i-objects as \emph{i-morphisms}.  
\i Let $\om$Cat denote the category whose objects are $\om$-categories and morphisms are functors between $\om$-categories, meaning maps of sets which preserve all structures $s_i$, $t_i$, $*_i$ for all $i$.  
\i We write $Ob : \om Cat \lra  S ets $ for the forgetful functor which sends an $\om$-category to its underlying set. 
\i For ${A} \in \om Cat$ and $a,b \in {A} _i$, let $Hom ^{i} _{A} (a,b) := \{x \in {A} \st  s_i x = a, \; \textrm{and}\; t_i x = b \}$.  
\een
\end{Def}

\begin{rem}\label{productslemma}  $\om$Cat is a symmetric monoidal category. For ${A} $, ${B} \in \infty$-cat, The product ${A} \times {B}$ is just the cartesian product as sets, and source, target, and composition maps are defined componentwise.  
\end{rem}

\begin{Def} We say that $A \in \om Cat$ is a \emph{groupoid} if every n-morphism $n\geq 1$ is an isomorphism, meaning that if $x \in A_n$, there exists for every $j<n$ a $y \in A$ such that $x*_{j}y = t_{j}x$ and $y*_{j} x = s_{j}x$.  
\end{Def}

\begin{rem} It is a well known result of Brown and Higgins \cite{brh2} that there is an equivalence between $\om$-groupoids and crossed complexes.
\end{rem}

\subsection{ Equivalences of  $\om$-categories}

\begin{Def}\label{1octa} 

\begin{enumerate}
\item For any $\om$-category $ A$, two i-objects $a , a' \in {A} _i$ are said to be isomorphic if there exists $u \in Hom ^{i+1}(a , a')$ and $v \in Hom^{i+1}(a' ,a)$ such that $u*_i v = a'$ and $v*_i u = a$.  (In the language of Street \cite{str}, $a$ and $a'$ are 1-equivalent.)

\item Let $F : {A} \lra {B}$ be a functor of $ $ $\om$-categories.  We say that $F$ is an equivalence of $ $ $\om$-categories if
\begin{enumerate}
\item\label{1octa2i} Any 0-object, $b \in {B}$ is isomorphic to $Fa$ for some 0-object $a $ in ${A}$,  
\item\label{1octa2ii}  for any $i \geq 0$ and $a , a ^\prime  \in A_i$ such that $s _{i-1} a = s _{i-1} a^\prime$, $t_{i-1 }a = t_{i-1} a ^\prime$ and $\psi \in Hom ^{i} _{B} (Fa , Fa')$, there exists $\phi \in Hom ^{i} _{A} (a , a')$ and an isomorphism $\be \in Hom ^{i+1} _{B}(F\phi , \psi)$, and 
\item\label{1octa2iii} for i-objects $a,a^\prime \in A$, if $Fa$ is isomorphic to $Fa'$ in ${B}$, then $a$ is isomorphic to $a'$ in ${A}$
\end{enumerate}

\noi Conditions \ref{1octa2i}, \ref{1octa2ii}, and \ref{1octa2iii} are the higher-categorical analogues of being essentially surjective, full, and faithful respectively.   The meaning of this definition is roughly that $F{A} _i$ should be the same as ${B} _i$, up to (i+1)-isomorphism.  In the notation of definition \ref{16nov1}, the first two conditions can be restated as:
\begin{enumerate}
\item For all $y \in {B} _0$, there exists $x \in {A} _0$ and isomorphism $f \in {B} _1$ such that $s_0 f = Fx$ and $t_0f =  y$.

\item If $a , a' \in {A} _i$ such that $s_{i-1}a  = s_{i-1}a ^\prime$, $t_{i-1} a = t_{i-1} a^\prime$, and $\psi \in {B} _{i+1}$ such that $s_i \psi = Fa$ and $t_i \psi = Fa'$, then there exists $\phi \in {A} _{i+1}$ and an isomorphism $\be \in {B} _{i+2}$ such that $s_i \phi = a $, $t_i \phi = a'$, $s_{i+1}\be = F \phi$, and $t_{i+1} \be = \psi$.

\end{enumerate}
\end{enumerate}
\end{Def}

\begin{rem} Note that when ${A}$ and ${B}$ are groupoids, then $F : {A} \lra {B}$ automatically satisfies condition \ref{1octa2iii} of Definition \ref{1octa} if it satisfies \ref{1octa2i} and \ref{1octa2ii}.  Thus, the third condition is superfluous when we are dealing with groupoids.  Furthermore, when $A$ and $B$ are groupoids, condition \ref{1octa2i} can be viewed as a special case of condition \ref{1octa2ii} if we add a point $\{*\} = A_{-1} = B_{-1}$ in degree $-1$.  
\end{rem}

Equivalences of $\om$-categories are in fact weak equivalences in a cofibrantly generated model structure on $\om Cat$ \cite{lmw}.  Ara and M\'etayer showed in \cite{arm} that this model structures restricts to one on $\om$-groupoids in a way that is compatible with Brown and Golansi\'nski's model structure on crossed complexes \cite{brg}. 

\section{Picard $\om$-categories}\label{piccatssection}
 
 \begin{Def} A \emph{\spic } is an abelian group object in $\om$Cat. We let $Pic_\om$ denote the category of Picard $\om$-categories, where $Hom_{Pic_\om}({A} , {B}) = \{ F \in Hom_{\om cat}({A}, {B}) \st F\circ + = + \circ (F\times F) \}$.  
\end{Def}

\begin{rem}For a $ $ Picard $\om$-category ${A}$, the fact that $+ : {A} \times {A} \lra {A}$ is a functor implies that each ${A} _i$ is a subgroup.  Also, we observe that if $A $ is a Picard $\om$-category, then $Ob(A)$ is an abelian group.  
\end{rem}

\begin{prop}\label{groupobjects} 
\begin{enumerate} 
\item An $\omega$-category ${A}$ such that $Ob({A})$ is endowed with the structure of an abelian group is a Picard $\om$-category if and only if
\begin{enumerate}
 \item $+ : {A} \times {A} \lra {A}$ is a functor of $\om$-categories \\
 and 
 \item $x * _i y = x +y - s_i (x) $ whenever the left hand side is defined.
  \end{enumerate}
\item $Pic _\om$ is an abelian category. 
\end{enumerate}   
\end{prop}
 
\begin{proof}
\begin{enumerate}
\item First suppose that ${A} \in Pic_\om$.  Then $+$ is a functor.  Furthermore, it is clear that $Ob(A)$ must be an abelian group object in $Set$.  We only need to verify that $x * _i y = x +y - s_i (x) $ whenever the left-hand side is defined.  Since $+$ is a functor, $(x+y)*_n (x^\prime + y^\prime) = (x*_n x^\prime )+(y*_ny^\prime)$ if the right side is defined.  Suppose that $s_ix = t_iy$. Then $x*_iy = (x+0)*_i (s_i x +(y-t_iy)) = x*_is_ix + 0*_i (y - t_i y) = x + (t_i y - t_iy )*_i (y-t_iy) = x + t_iy *_i y +(-t_iy*_i-t_iy) = x + y +-t_iy = x+y-s_ix $. \\ 

Now suppose that ${A} \in \om Cat$ such that $Ob(A)$ is an abelian group and conditions 1a and 1b are satisfied.  We wish to show that $A \in Pic _\om$.  Since $Ob$ is faithful and $Ob A$ is an abelian group object in $Set$, $A$ is an abelian group object in $\om Cat $ provided that addition $+$ and inverse $\io : A \lra A$ are functors of $\om$-categories.  By condition 1a, $+$ is a functor, so it only remains to see that $\io $ is a functor. 
 Since $0= \rho_n 0 = \rho _n (x+ x\inv)= \rho _n x + \rho _n x \inv$, $\rho _n x\inv = (\rho _n x )\inv$ for $\rho _n \in \{s_n, t_n \}$ so that $\io$
 respects source and target maps. 
 Also, since $(x*_n y)\inv = (x+y  - s_nx)\inv = x\inv + y \inv - s_nx\inv = x\inv *_n y\inv$, $\io$
 respects compositions.  We conclude that ${A}$ is a group object in $\om$Cat.  

\item For ${A} $, ${B} \in Pic_\om$, $Hom({A} , {B})$ is an abelian group.  The sum $\phi + \psi$ of two functors preserves all source and target maps and also preserves composition because $+$ is a functor: $(\phi + \psi)x*_ny = \phi x *_n \phi y + \psi x *_n \psi y = (\phi x + \psi x) *_n (\phi y + \psi y) = (\phi + \psi )x *_n (\phi + \psi )y$.  Direct sums, kernels, and cokernels are gotten by taking each on the level of abelian groups, e.g. $Ob (Ker \phi) = Ker Ob(\phi)$.  
It is clear how to define source, targets and compositions on direct sums and kernels.  For cokernels, since functors respect source and target maps, there is no difficulty in defining source and target maps on a cokernel.  Composition in a cokernel is defined by letting $x*_n y = x+y-s_nx$.    
\end{enumerate}
\end{proof}

\begin{rem} \label{forgetfuncts}The forgetful functors $\FF _{Ab}$, $\FF _{Set}= Ob$, $\FF _\om$ taking values in ${A} b$, $Set$, and $\om$Cat respectively are all faithful functors.  It follows from Proposition \ref{groupobjects}b that if $A$, $B \in Pic _\om$ and $g \in Hom_{Ab}(\FF_{Ab} A , \FF_{Ab} B)$ respects all source and target maps, then $g = \FF_{Ab} f$ for some $f \in Hom_{Pic _\om} (A,B)$.   
\end{rem}

\subsection{Picard $\om$-Categories and Chain Complexes}

\noi {\bf Notation.} Let $Ch^+ ({A}b)$ denote the category of complexes of abelian groups in non-negative degrees. \\

Let $Pic$ denote the category of Picard categories in the sense of Deligne \cite{del}.  That is, a Picard category $\CC$ is a quadruple $(\CC , + , \sigma, \tau)$, where $+: \CC \times \CC \lra \CC$ is a functor such that for all objects $x \in \CC$, $x+: \CC \lra \CC$ is an equivalence, and addition is commutative and associative up to isomorphisms $\tau$ and $\sigma$.  Deligne made the observation that one can assign to a complex $A^1 \lra A^0$ of abelian groups a Picard category $PA$, whose objects are $A_0$ and $Hom_{PA}(a,b)= \{f \in A_1 \st df = b-a \}$.  He goes on to show that $P : Ch^{0,1}(Ab) \lra Pic$ is an equivalence and also induces an equivalence between $D^{0,1}(Ab)$ and $Pic$ modulo natural isomorphism. \\ 

Define $Pic_\om ^1 = \{ {A} \in Pic_\om \st {A}  = {A} _1 \}$.  The relationship between $Pic _\om ^1$ and $Pic$ is explained in Proposition \ref{deligne}, the proof of which is found in section \ref{delignedetails} of the appendix.  Furthermore, when restricted to short complexes in degrees $1$ and $0$ only, Theorem \ref{26oct1} is a strictification theorem which states that $Pic _\om ^1$ and $Pic$ are equivalent.  

\begin{prop}\label{deligne} $Pic ^1 _{strict} $ consists of all small Picard categories in $Pic$ such that $+$ is strictly associative and commutative (i.e. $\tau$ and $\si$ are identities) and for each $x \in ob(\CC)$, $x+ : \CC \lra \CC$ is an isomorphism, not just an equivalence.
\end{prop}

Deligne's correspondence extends to longer complexes.  The correspondence in Proposition \ref{functlemma1} was considered by Bourn, Steiner, Brown, Higgins \cite{bourn, ste, brh}, et al.   
  
\begin{prop}\label{functlemma1} A complex $A$ of abelian groups defines a Picard
  $\om$-category $ P(A) $.  This assignment ${P} :Ch^+( {A}b) \lra \spics $ is a functor.  
\end{prop}

\begin{proof} Let ${P} = {P}(A) $ consist of sequences $ x =
  ( (x_0 ^- , x_0 ^+ ) , (x_1 ^- , x_1 ^+ ), ...) $ with $x_i ^\alpha
  \in A^i$ such that $dx_i ^\alpha = x^+ _{i-1} -  x^- _{i-1} $ for
  all $i\geq 1$, $\alpha \in \{+,-\}$. 
We define source and target maps by $s_i x =  ( (x_0 ^- , x_0 ^+ ) , (x_{i-1} ^- , x_{i-1} ^+ ), (x_i ^- , x_i ^- ), (0,0), ...) $, and $t_i x = ( (x_0 ^- , x_0 ^+ ) , (x_{i-1} ^- , x_{i-1} ^+ ), (x_i ^+ , x_i ^+ ) , (0,0), ...) $.  If $s_i x = t_i y$, define
\[ x *_i y =  ( (x_0 ^- , x_0 ^+ ) ,...,(x_{i-1} ^- , x_{i-1} ^+ ),  (y_i ^- , x_i ^+ ),(x_{i+1} ^-
  + y_{i+1} ^- ,x_{i+1} ^+
  + y_{i+1} ^+ ), 
...)
\]
We need to check that $x*_iy$ is an element of $PA$.  Firstly, $dx_i ^+ = dx_i ^- = dy_i ^+ = dy _i ^- $ since $s_ix = t_iy$.  Secondly, $d( x_{i+1} ^+
  + y_{i+1} ^+)= d(x_{i+1} ^-
  + y_{i+1} ^-) = (x_i ^+  - x_i ^-) + (y_i ^+ - y_i ^-) = x_i ^+ - y_i
  ^-$ since $x_i ^- = y_i ^+$.  Finally, for $j>i+1$, it is obvious
  that $d (x*_iy)_j ^\alpha = d((x*_iy)_{j-1} ^+ -(x*_iy)_{j-1} ^-)$.
  Hence, $(x*_iy) \in  {P} $.
It is easily checked that ${P}A$ is an $\om$-category. \\

\noi We now define an operation ${P} \times
{P} \lra {P} $  which
makes ${P}$ into a \spic. 
This is the obvious operation 
\[ x+ y =  ( (x_0 ^- +y_0 ^- , x_0 ^+ +y_0 ^+) , (x_1 ^- + y_1 ^-, x_1
^+ +   y_1 ^+ ), ...), 
\]
which obviously satisfies $x+y - s_n x = x*_ny$ when composition is defined.  To see that $+$ is a functor, let $x ,y ,a,b \in {P}$ such that $s_n x = t_na$ and $s_ny = t_n b$.  Then
\begin{eqnarray*} + ((x,y)*_i (a,b) )& =& (x*_i a) + (y*_ib)  \\ &=&  
( (x_0 ^-  , x_0 ^+ ) ,...,
(a_i ^- , x_i ^+), (a_{i+1} ^- + x_{i+1} ^-, a_{i+1}
^+ +   x_{i+1} ^+ ),...) \\
&& \qquad + 
( (y_0 ^-  , y_0 ^+ ) ,...,
(b_i ^- , y_i ^+), (b_{i+1} ^- + y_{i+1} ^-, b_{i+1}
^+ +   y_{i+1} ^+ ),...) \\ & =& 
( (x_0 ^- +y_0 ^- , x_0 ^+ +y_0 ^+ ) ,...,
(a_i ^-  +b_i ^- , x_i ^+ + y_i ^+), (a_{i+1} ^- + x_{i+1} ^- + b_{i+1} ^- + y_{i+1} ^-, a_{i+1}
^+ +   x_{i+1} ^+ + b_{i+1}
^+ +   y_{i+1} ^+ ),...) \\ & =& (x+y)*_i (a+b),
\end{eqnarray*}
so by Proposition \ref{groupobjects}, ${P}$ a \spic. 
The assignment $ {P} : C^{\leq 0} ({A} b) \lra \spics $ is obviously a functor; for a morphism $f: A \lra B$ in $Ch^+(Ab)$, $Pf$ is given by $({P} f x)^ \al_i = f(x_i ^\al)$ for $\al \in \{+, -\}$. 
\end{proof} 
\noi Henceforth, we shall denote a sequence $x = ( (x_0 ^- , x_0 ^+ ) , (x_1 ^- , x_1 ^+ ), ...) $ by $( x_i )_{i \in \N}$ or simply $(x_i)$, where $x_i = (x_i ^- , x_i ^+)$. 

\begin{lem}\label{7nov2} A Picard $\om$-category $A$ defines a chain complex ${Q}(A) \in Ch^+(Ab)$ in such a way that ${Q} : \spics \lra Ch^+ ({A}b)$ is a functor.\end{lem}

\begin{proof} Since $+ {A} \times {A} \lra {A}$ is a functor, it respects all operators $*_i $, $s_i$, $t_i$.  Hence, $+ {A} _i \times {A} _i \lra {A} _i$, so ${A} _i$ is a subgroup of ${A}$.  Therefore it makes sense to define ${Q}^{i}:= ({Q}A)^i: = {A} _i / {A} _{i-1}$ for $i >0$ and $Q^0 = A_0$. We define, for each $i > 0$ an operation $d : {Q}^i \lra {Q} ^{i - 1} $ by first defining a homomorphism $d _0 : {A} _i \lra {A} _{i-1} $.  Define $d _0 = t_{i-1} - s_{i-1}$.  We must check that this is a homomorphism.  If $x,y \in {A} _i$,
\begin{eqnarray*}
d _0 (x +y ) &=& t\imob (x+y) - s\imob (x+y ) \\
&=& t\imob x + t\imob y - (s\imob x + s\imob y) \\
&=& t\imob x - s\imob x + (t\imob y - s\imob y)\\
&=& d_0  x + d _0 y.
\end{eqnarray*}

\noi If $x \in {A} \imob$, then $t \imob x = x= s\imob x$, so $d _0 ({A} \imob ) = 0$.  Therefore, $d _0$ is a group homomorphism such that 
${A} \imob \subset Ker d _0$.
  This determines a homomorphism $d : {Q}^i \lra {Q}^{i-1}$.  To see that $d ^2 = 0$, for $x \in {Q}^i$,
 $d ^2 x = d (t\imob x - s\imob x) 
= t_{i-2}(t\imob x - s\imob x) - s_{i-2}(t\imob x - s\imob x) 
= t_{i-2} x
- t_{i-2} x - (s_{i-2} x - s_{i-2} x)
= 0$.
Hence, ${Q}$ is a complex of abelian groups. \\

 We have constructed ${Q}$ from ${A}$, which gives a map ${Q} : \spics \lra C^{\leq 0} ({A} b) $.  If $F: {A} \lra {B}$ is a map of $\om$-categories, $F : {A} _i \lra {B} _i$ for each $i \geq0$, so $F$ descends to a map ${Q}(F) : {A} _i / {A} \imob \lra {B} _i / {B} \imob$, which is easily seen to be a map of complexes.  This makes ${Q}$ is a functor of 1-categories.
\end{proof}  

\begin{thm}\label{26oct1}(\cite{bourn}) ${Q} \circ {P} \simeq id$  and ${P} \circ {Q} \simeq id$.  Therefore, ${Q}$ and ${P}$ are equivalences of categories.  
\end{thm}

\begin{proof} 
From a complex $A \in Ch^+ ({A}b)$, we get a complex $ {Q} = {Q} ({P} A)$, where $ {Q}^i =  ({P}A) _i /  ({P}A )\imob$.  First define a map $A^i \lra ({P} A) _i$, $x \mapsto \hat x$, in the following way: 
 
\begin{displaymath} 
(\hat x)_j = 
\left\{ 
\begin{array}{ll}
(0, 0)     & \textrm{if $ j< \imo $}\\
(0 , d x) & \textrm{ if $j = \imo$}\\
(x,x)   & \textrm{ if $j=i$ }.
\end{array}
\right.
\end{displaymath}

\noi Now define a map $h : A^i \lra {Q} ^i$ by $h(x) = [\hat x] \in {P}{A} _i / {P}{A} \imob$.  Observe that for $y \in ({P} A) _i$, $[y] = [y - s \imob y ] = [\hat {y _i ^+}]= h(y_i ^+)$.    Therefore, $h$ is surjective.  It is clear that $h$ is a homomorphism and that it is injective. But $h$ must also be a map of complexes.  If $x \in A^i$, $dh(x) = d [\hat x] = [ t \imob \hat x - s\imob \hat x ]$, where

\begin{displaymath} 
(t \imob \hat x - s\imob \hat x)_j= 
\left\{ 
\begin{array}{ll}
(0, 0)     & \textrm{if $ j \neq \imo $}\\
(d x , d x) &  \textrm{if $ j=i-1$}
\end{array}
\right.
\end{displaymath}

Since $d ^2 = 0$, this is obviously the class of $\hat{ d x} $.  Hence, $h$ is a morphism of chain complexes, and we conclude by injectivity and surjectivity that $h_A : A \lra {Q} {P} (A)$ is an isomorphism.  However, to be an isomorphism of functors, ${Q} {P} \tilde \rightarrow 1$, these maps must satisfy $h_B \circ f = (({Q} {P} (f)) \circ h_A$ for any map of complexes $f: A \lra B$.  For $x \in A^i$, $h_B (f(x))= [\hat {(f(x))} ] \in ({P}B) _i / ({P} B) \imob $.  Applying the right-hand side to $x$, we get $(({Q} {P} (f)) \circ h_A x ={Q}({P} (f)) [\hat x] = [{P} (f) \hat x] = [\hat{ (f(x))}]$.  So $h$ does in fact define an isomorphism between endofunctors ${Q} {P} $ and $1$ of $Ch^+ ({A}b)$.  \\

Now we wish to show that for $A \in Pic _\om$, there is an isomorphism $\vphi : {A} \lra PQ A$ in $\spics$.  For the rest of the proof, for $x \in {A} _i$, let $[x]$ be its image in ${Q}^{i} = {A} _i / {A} \imob$, and for any $x \in {A}$, let $\mu (x) :=  min\{m \in \N \st s_m x = x \} $.  Now, for $x \in {A} $, define $\vphi x = \{(\vphi x)_i\}_{i \in \N}$ by:

\begin{displaymath} 
(\vphi x)_i = 
\left\{ 
\begin{array}{ll}
([s_i x ] , [t_i x])    & \textrm{ if $i \leq \mu (x)$}\\
(0, 0)     & \textrm{if $i > \mu (x)$}.
\end{array}
\right.
\end{displaymath}

\noi It is clear that $\vphi x \in {P}{Q} A$, but we still must check that $\vphi$ is a functor.  For $\rho _i \in \{ s_i , t_i \}$, it must be shown that $\vphi (\rho _i x) = \rho _i \vphi x$.  If $i > \mu (x)$, this is obvious.  If $i < \mu (x)$, 

\begin{displaymath} 
(\vphi (\rho _i x))_j = 
\left\{ 
\begin{array}{ll}
([s_i x ] , [t_i x])    & \textrm{ if $i <j$}\\
([\rho _i x], [\rho _i x])     & \textrm{if $i=j$}\\
(0,0)   & \textrm{ if $j>i$},
\end{array}
\right.
\end{displaymath}

\noi which is exactly the same formula for $(\rho _i (\vphi x))_j$.  Additionally, $\vphi$ must be a homomorphism of abelian groups, but this follows easily.  For simplicity, assume $\mu (x) \leq \mu (y)$.

\begin{align*} 
(\vphi (x+y))_i & =
\left\{ 
\begin{array}{ll}
([s_i (x+y) ] , [t_i (x+y)])    & \textrm{ if $i \leq \mu (x +y) = \mu (y)$}\\
([x+y] , [x +y])    & \textrm{ if $i = \mu (y)$}\\
(0, 0)     & \textrm{if $i > \mu (y)$}
\end{array} 
\right. \\ 
 & =  \left\{
\begin{array}{ll}
([s_i x] +[s_i y ] , [t_i x] +[t_i y])    & \textrm{ if $i \leq \mu (x +y) = \mu (y)$}\\
([x] +[y] , [x] + [y])    & \textrm{ if $i = \mu (y)$}\\
(0, 0)     & \textrm{if $i > \mu (y)$}
\end{array}
\right. \\
&=  (\vphi (x) + \vphi (y) ) _i.
\end{align*}

\noi By Remark \ref{forgetfuncts}, $\vphi$ is a morphism of Picard $\om$-categories. \\

\noi To show that $\vphi$ is an isomorphism of $\om$-categories, we simply show that $\vphi$ is bijection of sets.  Let ${P} = {P}{Q} (A)$.  First we show that $\vphi$ is surjective.  We prove by induction that $( {P} )_n$ is in the image of $\vphi$ for each each $n \in \N$.  If $n=0$, let $a = ((a_0 , a_0), (0,0),...) \in {P}  _0$, so $a_0 \in A^0 = {{P} } _0$, and $\vphi (a_0) = a$.  Now suppose that ${P}  _k  \subset \vphi ({A})$.  Let $x =(x_i)_{i \in \N} = (([a_i ^- ] , [a_i ^+]))_{i \in \N} \in {P} _{k+1}$, $a_i ^\pm \in {A} _i$.  Then $[a_{k+1}^- ] = [a_{k+1}^+ ]$, and $\vphi (a_{k+1}^+ ) - x \in ({P} )_k  \subset \vphi ({A})$ by induction hypothesis, so  $\vphi (a_{k+1}^+ ) - x = \vphi (z) $ for some $z \in {A} _k$.  Hence, $x = \vphi (a_{k+1}^+ -z) \in \vphi ({A})$.  Thus, ${P} _{k+1} \subset \vphi ({A})$ and therefore $\vphi$ is surjective.\\ 

Now we demonstrate that $\vphi$ is injective.  Let $x \in Ker\phi$ and $\mu = \mu (x)$ as above.  Then $s_n x = x $ for all $n \geq \mu$ and $s_k x \neq x $ for all $k < \mu$.  If $\phi x = 0$, $[s_kx] = 0 $ for all $k$, whence $s_k x \in {A} _{k-1}$.  In particular, $s_\mu x \in {A} _{\mu -1}$, which implies that $s_{\mu -1}x = x$, which contradicts the minimality of $\mu$.  Therefore to avoid a contradiction, $x$ must be $0$ and $Ker \phi = 0$.  Therefore, $\vphi$ is an isomorphism.\\

\noi For each ${A} \in \spics$, we've produced an $\vphi _A :  {A}  \tilde \lra {P} {Q} ({A}) $.  To complete the proof, we simply must see if this satisfies compatibility with morphisms in $\spics$.  For $x \in {A}$, we require that $ \vphi _B \circ f (x) = ({P} {Q} (f)) \circ \vphi _A$.  We consider the i-th entry in each sequence and see that the are the same.  Since $(\vphi _A (x) )_i = ([s _i x] , [t_i x])$, with $[s_ix], [t_i x] \in {A} _i / {A} \imob$, we have that $({P} ({Q} f)(\vphi _A (x))_i =  ({Q} f)(\vphi _A (x))_i = (({Q} (f))[s_i x] , ({Q} (f))[t_i x] )= ([f (s_ix)] , [f (t_ix)] ) = ( [s_i f(x)] , [t_i f(x)] ) = (\vphi _B \circ f (x))_i$. 
\end{proof}

\subsection{Equivalences of Picard $\om$-categories}
\begin{prop}\label{qisprop}For complexes $A,B \in C^{+} ({A} b)$ and map of complexes $f: A \lra B$, $f$ is a quasi-isomorphism if and only if ${P} f : {P} A \lra {P} B$ is an equivalence of $ $ $\om$-categories.
\end{prop}

\begin{proof} This proof will use the description in the first part of Definition~\ref{1octa}, as it simplifies the notation.  For an object $a \in A^n$ satisfying $da =0$, we denote its image in $H^n (A)$ by $[a]$, and for a map $f: A \lra B$, by abuse of notation, let $f$ also denote the induced map on cohomology. For ease of notation, let $F : {A} \lra {B}$ denote ${P} f : {P} A \lra {P} B$. \\

Let $A \stackrel{f}{\lra} B$ be a \qis of complexes $A$ $B \in Ch^+(Ab)$ . If $y \in PB_0 = B ^0$, there exists $x \in A^0$ such that $[fx] = [y]$, so there exists some $z \in B ^1 $ such that $dz = fx -y$.  Thus, $((y,fx),(z,z),0,...) \in Hom ^1 (fx ,y)$ is an isomorphism as required.  This proves condition (1) of Definition~\ref{1octa}.  To prove condition (2), let $\psi \in Hom_{{P}{B}}^{n}(Fx,Fy)$ with $x$, $y \in {P}{A} _n$ such that $s_{n-1}x = s_{n-1}y$ and $t_{n-1}x = t_{n-1}y$.  Then $\psi = ((\psi _0 ^- , \psi _0 ^+),...,(\psi _{n+1} ^- , \psi _{n+1} ^+),0,...)$ with $s_n \psi = Fx = ((fx _0 ^- , fx _0 ^+),...,(fx _n ^- , fx _n ^+),0,...)$ and  $t_n \psi = Fx = ((fy _0 ^- , fy _0 ^+),...,(fy _n ^- , fy _n ^+),0,...)$, so $\psi _n ^- = fx_n ^- = fx_n ^+$ and $\psi _n ^+ = fy_n ^- = fy_n ^+$.  This means that $d\psi_{n+1} ^\pm = fy_{n} ^\pm - fx_n ^\pm$ so that $[fx_n ^+ -fy_n ^+] = 0$ and therefore $[x_n ^+ - y_n ^+] = 0$ because $f$ is a \qis.  Hence, $d\phi = y_n ^+ - x_n ^+$ for some $\phi \in A ^{n+1}$. Since $\sigma _ix = \sigma _iy$ for $i < n$, $\sigma _i \in \{s_i , t_i \}$, $x_i ^\pm = y _i ^\pm $. We conclude that $((x _0 ^- , x _0 ^+),...,, (x_{n-1} ^- , x_{n-1} ^+),(x _n ^+ , y _n ^+),(\phi , \phi), 0,...)$ is an $(n+1)$-morphism from $x$ to $y$ as required.  Since ${P}{A}$ and ${P}{B}$ are groupoids, condition (2) entails condition (3).  Therefore, $F : {A} \lra {B}$ is an equivalence of $\om$-categories. \\

Now suppose that $F$ is an equivalence.  We will show that $f : H^n (A) \lra H^n(B)$ is an isomorphism for each $n$.  Let $x \in Ker(d: A^n \lra A^{n-1})$.  If $[fx] = 0$, then there is some $\psi $ such that $d\psi = f x$.  We see that $(0,...,(0,0),(0,fx),(\psi , \psi ),0...)$ is an $(n+1)$-isomorphism from $0$ to $\hat fx = (0,...,(0,0),(fx,fx),0...)$.  By condition (3) of Definition~\ref{1octa}, $\hat x $ is isomorphic to $ 0$. Such an isomorphism is of the form  $(0,...,(0,0),(0,x),(\phi , \phi ),0...)$ for some $\phi$ satisfying $d\phi = x$, so we see that $[x]=0$ and therefore $f$ is injective.  To see that $f$ is surjective, let $[y] \in H^n(B)$ for $n >0$ and consider $\hat y = (0,...,0,(0,0),(y,y),0,...)$, which is an $n$-isomorphism from $0 = F(0)$ to itself.  By condition (2), there exists an $n$-isomorphism $\hat x = (0,...,0,(x,x),0...)$ in ${A}$ together with an $(n+1)$-morphism $(0,...,0,(fx,y),(z,z),0...)$ from $F\hat x $ to $\hat y$.  Hence, $dz = y - fx $ so that $f[x] = [y]$ and $f$ is surjective.  For $n =0$, let $y \in H^0 (B)$, we use the same argument, except this time invoking condition (1) to ensure the existence of such an $x \in A^0$.  This shows that $f$ is a \qis. \end{proof}




Let $Ho(Pic_\om )$ denote $Pic_\om$ localized at the equivalences.  We now have the following corollary.

\begin{cor} The derived category $D ^{\leq 0} ({A} b)$ of abelian groups in degrees $\leq 0$ is equivalent to the homotopy category $Ho(Pic_\om)$.
\end{cor}

\subsection{Useful Facts for Picard $\om$-categories}\label{17nov2}

Lemma \ref{extensionlem} shows that if $A$ is a subcategory of $B$, and $B$ can be extended to an $\om$-category $\CC$, then $A$ can be extended to an $\omega $-sub-category of $\CC$.  

\begin{lem}\label{extensionlem} Let ${A}$ and ${B}$ be 1-categories with ${A}$ a subcategory of ${B}$.  If there is an $\om$-category $\CC$ such that $(\CC _1 , s_0 , t_0 , *_0 ) = {B}$, then there is an $\om$-subcategory $\CC '$ of $\CC$ such that $(\CC ' _1 , s_0 , t_0 , *_0 ) = {A}$.
\end{lem}

\begin{proof}  Define $Ob(\CC ') = \{ c \in Ob(\CC) \st s_1 c \; , \; t_1 c \in \CC ' _1 \}$.  We show that $\CC ^\prime$ is an $\om$-category by first showing that it is stable under all source and target maps and then showing that it is closed under composition.  If $x \in Ob (\CC ^\prime)$ and $j > 1$, $ \rho _1 \si _j  x = \rho _1 x \in \CC ' _1 $.  If $j = 1$, $\rho _1 \si _j x = \si _j x \in \CC ' _1$.  Finally, if $j = 0$, $\rho _1 \si _0 x = \si _0 x = \si _0 \rho _1 x$.  Since ${A}$ is a category, $\si _0 \rho _1 x \in {A} _0 = \CC ' _0 \subset \CC ' _1$.  This shows that $\CC '$ is stable under maps $\si _j \in \{ s_j , t_j \}$.  \\

\noi Now suppose $x, \, y \in \CC '$.  If $i \geq 1$, $s_1 (x *_i y ) = s_1 y \in \CC '$, and $t_1 (x*_i y ) = t_1 x \in \CC '$.  If $i=0$, $\rho _1 (x*_i y ) = \rho _1 x *_0 \rho _1 y \in {A} \subset \CC '$ since ${A}$ is a category.  This shows that $\CC '$ is stable under all compositions $*_i$.  Clearly $\CC '$ is an $\om$-category because all other necessary properties are inherited from $\CC$.  
\end{proof}

In \cite{str}, Street constructs an $\om$-category $Cat _\infty$ such that $((Cat _\infty)_1 , s_0, t_0 , *_0 ) = \om Cat$ so that $\om$-Cat is a category enriched over itself.  The construction is natural since it comes from an inner hom in $\om$Cat.  Lemma \ref{extensionlem} can be applied to $Pic_\om$: there is an $\om$-category $\CC$ such that $(\CC _1 , s_0 , t_0 , *_0 ) = \spics $.  In fact, we see in sections \ref{homotopy} and \ref{icats} that there is more than one $\om$-category which has $Pic _\om$ as its 0-objects and 1-morphisms.  \\

We now make the observation that for $ {A} \in \spics$, since there is a section of ${A} _i \lra {A} _i / {A} _{i-1}$ sending $[x]$ to $x - s_{i-1}x$.  It follows that ${A} _i \simeq  A_i / A_{i-1} \oplus {A} _{i-1}$ as abelian groups.  Moreover, 
\[ {A} \simeq A_0 \oplus \bigoplus _{i =1} ^\infty  A _i /A_{i-1}
\]
\noi With this identification, for $x \in  A_k/A_{k-1}$, 
\begin{displaymath} 
s_n x = 
\left\{ 
\begin{array}{ll}
x     & \textrm{if $ n \geq k $}\\
0  & \textrm{ if $n < k$}
\end{array}
\right.
\end{displaymath}
 
\begin{displaymath} 
t_n x = 
\left\{ 
\begin{array}{ll}
x     & \textrm{if $ n \geq k $}\\
dx  & \textrm{ if $n = k-1 $}\\
0 & \textrm{ if $n < k-1$}
\end{array}
\right.
\end{displaymath}
 
This can be written explicitly as $s_n : (x_0,...,x_m,...) \mapsto (x_0,...x_n,0,0,...)$ and $t_n : (x_0,...,x_m,...) \mapsto (x_0,...,x_{n-1}, x_n + dx_{n+1},0,0,...)$, where $d = t_n - s_n$.\\
There are several possible identifications of $A$ with $\bigoplus _{i =0} ^\infty  A _i/A_{i-1}$.  It is perhaps most transparent if $A = P(C) $ for $C \in Ch^+(Ab)$, so $C^i \simeq  A_i/A_{i-1}$, and $\bigoplus _{i=0} ^\infty C^i \lra A$ is given by $(x_0,x_1,x_2,...,x_n,0,0,...)\mapsto ((x_0 , x_0 +dx_1), (x_1, x_1 +dx_2), ...,(x_{n-1}, x_{n-1}+ dx_n), (x_n , x_n),(0,0),...) $.

 \begin{prop}\label{shiftfunctor} There is a functor $[-1] : \spics \lra \spics$ such that for $C \in Ch^+({A}b) $, ${P} (C [-1]) = ({P} (C))[-1]$, where $C[-1]$ denotes the complex $...\lra C^1 \lra C^0 \lra 0 \in Ch^+(Ab)$ with $C[-1]^n = C^{n-1}$ for $n>0$ and $C[-1]_0 = 0$.  For ${A} \in \spics$, we define $A[-1]$ by letting $Ob(A[-1]) = Ob(A)$, $( *[-1]_n= *_{n-1} , s[-1]_n = s_{n-1}, t[-1]_n = t_{n-1})$ for all $n \geq 1$, $s[-1]_0 = t[-1]_0 = 0$, and $x*_0 y = x+y$.
 \end{prop}
 \begin{proof}  That $A[-1] \in Pic _\om$ is evident. If $A = P(C)$, there is a bijection between $Ob (P{C}[-1] )$ and $Ob(({P}C)[-1])$.  For $ ((x_0 ^- , x_0 ^+ ), (x_1 ^- , x_1 ^+ ),...,(x_{n-1} ^- , x_{n-1}^+) ) \in (PC)[-1]_n$ maps to $ ((0,0), (x_0 ^- , x_0 ^+ ), (x_1 ^- , x_1 ^+ ),... )$. 
 \end{proof}
 
The following proposition is a generalization of a lemma found in \cite{brh}

\begin{prop}\label{16nov2} Let ${A}$ be any abelian group.  If ${A}$ admits $\bbz$-linear maps $s_n \, , t_n : {A} \lra {A}$ for $n \in \N$ satisfying conditions 1a, 2a, and 2b of definition~\ref{16nov1}, then there is a unique $\om$-category structure on ${A}$ such that ${A} \in \spics$.  
\end{prop}

\begin{proof} If $a,b \in {A}$ such that $s_na = t_nb$, then we define $a*_nb = a + b - s_n a$ and check that this makes ${A}$ into a $ $ Picard $\om$-category.  If compositions satisfy all $\om$Cat axioms, then it is easily seen that $+$ is a functor, i.e. $(a+b)*_n (a' + b') = (a *_n a')+ (b*_n b')$, whenever the right-hand side is defined.  We can easily check that $ (a+b)*_n (a' + b') = a+ b + a' + b' - s_n (a + b) = a+ b + a' + b' - s_n a  - s_n b =  (a *_n a')+ (b*_n b')$.  Thus, it suffices to check that ${A}$ satisfies all $\om$-category axioms of Definition~\ref{16nov1}.  \\

\noi (1b) $a*_n s_n (a) = a + s_n a - s_n a = a$, and $t_n a *_n a = t_n a + a - s_n t_n a = t_n a + a - t_n a = a$.\\

\noi (1d) $s_n (a*_nb ) = s_n (a + b - s_n a ) = s_n a + s_n b - s_n a= s_n b$, and similarly, $t_n (a*_n b) = t_n ( a+ b - s_n a ) = t_n a + t_n b - s_n a = t_n a$ since $t_n b = s_n a$.\\

\noi (1c) $(a*_n b)*_n c = (a + b - s_n a ) + c - s_n (a*_nb) = a + b + c - s_n a - s_n b$, whereas $a*_n (b*_n c ) = a + (b+ c - s_n b) = s_n a$.  \\

\noi (2c) $\rho _j (a*_i b) = \rho _j ( a + b - s_i a ) = \rho _j a + \rho _j b - \rho _j s_i a = \rho _j a + \rho _j b - s_i (\rho _j a ) = (\rho _j a ) *_i (\rho _j b)$  since $j>i$. \\

\noi (2d) $(a*_j b)*_i (\al *_j \be)= (a + b - s_j a) + (\al + \be - s_j \al ) - s_i (a+ b - s_ja) = a + b + \al + \be - s_j a - s_j \al - s_i a - s_i b + s_i a = a + b + \al + \be - s_j a - s_j \al  - s_i b $.  On the right-hand side we have $(a *_i \al ) *_j ( b *_i \be) = (a + \al - s_i a ) + ( b + \be - s_i b) - s_j ( a*_i \al ) = a + b + \al + \be - s_i a - s_i b - (s_ja *_i s_j \al) = a + b + \al + \be - s_i a - s_i b - (s_ja + s_j \al - s_i a) = a + b + \al + \be - s_j a - s_j \al - s_i b$.
\end{proof}

\noi For a set ${A}$, we let $\bbz [{A}]$ denote the free abelian group on the set ${A}$.  If ${A} \in \om Cat$, we can extend all source and target maps $\bbz$-linearly.  Proposition~\ref{16nov2} implies that $\bbz [{A}] \in \spics$.    

\begin{lem} The functor $\bbz : \om Cat \lra \spics$ sending $A$ to the free abelian group generated by $A$ is left-adjoint to the forgetful functor $\FF _\om : \spics \lra \om Cat$.  
\end{lem}

\begin{proof} Let ${A} \in \om Cat$ and ${B} \in \spics $. 
 Any $\vphi \in Hom_{\om Cat }({A} , \FF _\om ({B}))$
 can be extended $\bbz$-linearly to a map $\hat \vphi \in Hom_{{A} b} ( \bbz [{A}] , {B}) )$
 of abelian groups.  The fact that all $s_n$ and $t_n $ are $\bbz$-linear means that $\hat \vphi$ commutes with all source and target maps.  But since composition $*_n$ 
in a $ $ Picard $\om$-category is determined by all $+$,
 $s_n$, $t_n$, $\hat \vphi$ also respects compositions $*_n$.
  Therefore $\hat \vphi \in Hom _{\spics } (\bbz [ {A}] , {B})$.
  It is clear that the function $\vphi \mapsto \hat \vphi $ 
is injective.  To see that it is surjective, any $\psi \in  Hom _{\spics } (\bbz [ {A}] , {B})$ 
is also a map of abelian groups, so it comes from some $\vphi \in Hom _{Sets} ( Ob({A} ) , Ob(F({B})))$.
  Of course since $Ob ({A} ) \subset Ob (\bbz [{A}])$,
  $\vphi (a)= \psi (a)$ and so $\vphi$ 
is actually a map of $\om$-categories and 
$\hat \vphi = \psi$.  The inverse map is $\psi \mapsto \psi _{| A}$.\\

\noi To check that $G$ and $F$
 are adjoints, we must also see that we have a map of functors\\ 
$ Hom_{\spics} ( \bbz [ - ] , -) \tilde \lra  Hom_{\om Cat } (- , F - )$.
  In other words, for $f \in Hom _{\om Cat } (A, A')$ 
and $g \in Hom_{\spics} (B , B')$, 
then for $\psi \in Hom_{\spics} (\bbz[A'] , B)$,
 the maps $(g \circ \psi \circ \bbz [ - ])_{|A} =
 F(g) \circ (\psi _{|A'}\circ f \in Hom_{\om Cat} (A , F(B'))$. It is easy to see that these maps agree at the level of sets. 
\end{proof}  

In \cite{str} it is shown that $Ob : \om Cat \lra Set$ is represented by an object $2_\om \in \om Cat$.   

\begin{cor}\label{16nov3} $\bbz [2_\om] $ is a corepresentative for the functor $Ob : \spics \lra S ets$.  
\end{cor}

\begin{proof}
$Hom_{\spics} (\bbz [2 _\om] , B ) = Hom _{\om Cat} (2_\om , \FF _\om (B)) = Ob (\FF _\om (B)) = Ob (B)$.  
\end{proof}




\subsection{I-categories}\label{icats}

Since the structure maps for an $\om$-category are indexed by natural numbers, we may think of an $\om$-category as an $\mathbb N$-category.  The $\om$-category axioms depended only on $\mathbb N$ being a partially ordered set.  We may therefore extend the definition and define an $I$-category for any partially ordered set $I$.  In particular, we are interested in $\bbz$-categories and show that ``nice'' abelian group objects in $\bbz$-categories are the same as unbounded chain complexes of abelian groups.  

\begin{Def} Let $I$ be a linearly ordered set. 
\begin{enumerate} 
\item An \emph{I-category} is a set quadruple $({A} , s_i , t_i ,*_i)_{i \in I}$ as in Definition~\ref{16nov1} except that instead of $\mathbb N$, we have $I$. Let $I$-Cat denote the category of all I-categories.   
\item Let $Pic_I$ denote abelian group objects in $I$-cat, and let $Pic_I ^0$ denote the full subcategory, called Picard I-categories, the objects of which are ${A} \in Pic _I$ such that for all $x \in {A}$, there exists $n \in I$ such that $s_nx = 0$ and there exists $m\in I$ for which $s_mx = x$. 
\end{enumerate} 
\end{Def}

We can extend some of the results about $\om$-categories to $\bbz$-categories.  Theorem \ref{26oct1}, for instance, can be extended to $\bbz$-categories.

\begin{Def}\label{zcatsequiv} A functor $F : {A} \lra {B}$ of $\bbz$-categories is an equivalence if conditions \ref{1octa2ii} and \ref{1octa2iii} of Definition \ref{1octa} are met and for each $b \in {B}$, there exists $n$ such that $s_nb$ is isomorphic to some $F(a)$. 
\end{Def} 

It follows directly from the definition of $Pic_\bbz ^0$ that any map of Picard $\bbz$-categories which  satisfies conditions \ref{1octa2ii} and \ref{1octa2iii} of Definition \ref{1octa} is an equivalence.  In fact, since Picard $\bbz$-categories are groupoids, condition  \ref{1octa2ii} of Definition \ref{1octa} is sufficient. 

\begin{thm}\label{zcats} The category $Ch (\mathcal Ab)$ of chain complexes of abelian groups is equivalent to the category $Pic_\bbz ^0$ of Picard $\bbz$-categories.  
\end{thm} 
\begin{proof} The proofs of Proposition~\ref{groupobjects}, Lemma \ref{functlemma1}, \ref{7nov2}, and Theorem \ref{26oct1} extend naturally to $\bbz$-categories with only a few modifications.  First we define ${P} (A) $ to consist of sequences $(....(x_i ^- , x_i ^+),...)$ as before but require that only finitely many $x_i ^\pm$ are nonzero.  Secondly, in lemma \ref{7nov2}, to show the surjectivity of $\phi : {A} \lra {P} {Q} ({A})$, we choose $y \in {P} {Q} ({A})$ and such that $y_i ^\pm = 0$ for $i \leq n \in \bbz$.  To show that $y$ is in the image of $\phi $, we start the induction at $n$ instead of $0$.  Also, we note that $\mu (x) $ of Lemma~\ref{7nov2} is well defined except for when $x=0$.  
\end{proof}

Letting $Ho (Pic_\bbz ^0)$ denote Picard $\bbz$-categories localized at equivalences, we arrive at the following corollary, the proof of which is identical to the proof of Proposition \ref{qisprop}.

\begin{cor} The derived category $D {A} b$ of abelian groups is equivalent to the homotopy category of $Pic_\bbz ^0$.
\end{cor}

The following results about $\om$-cats also extend to $\bbz$-categories: Prop~\ref{groupobjects}, Proposition~\ref{qisprop}. Also, the equivalence ${P} : Ch ({A} b) \lra Pic_\bbz ^0$ is, just as for Picard $\om$-categories, isomorphic to the one that sends $A \in  Ch ({A} b) $ to $\bigoplus _{n \in \bbz} A ^n \in Pic _\bbz ^0$.  The following proposition is patent.

\begin{prop}\label{spectra} $Pic_\bbz ^0$ is a triangulated category with shift functor given as in Proposition~\ref{shiftfunctor}.  This gives $DPic_\bbz ^0$ the structure of a triangulated category.  The mapping cone of $f : {A} \lra {B}$ is $({B} [-1] \times {A} , s_n = s_{n-1} ^{B} \times s_n ^{A} , t_n = t_{n-1} ^{B} \times (f + t_n ^{A}))$.  The t-structure coming from the standard t-structure on $Ch ({A} b)$ is $D^{\geq 0 } = Ho( Pic _\om)$, $D^{\leq 0} = \{{A} \in Pic _\bbz ^0 \st s _n = id, t_n =id \; \textrm{for all} \; n > 0 \}$.  Its heart is $\{ {A} \in Pic_\om \st \; all \; \sigma _n = id \} \simeq {A} b$. 
\end{prop}

\begin{rem} The $\om$-category structure on $Ch({A}b)$ induced from Theorem~\ref{zcats} is given in the following way.  1-objects are maps of complexes.  Strict 2-objects are maps between maps of complexes $F, G : A \lra B$, i.e. $\phi : F \Longrightarrow G$ is a map $\phi \in Hom (A[-1],B)$ such that $d\phi = G - F$.  etc.
\end{rem}

\begin{rem} Proposition \ref{spectra} is just another way to say that the category of abelian group spectra is equivalent to $Ch(Ab)$.  The derived version also holds from this point of view, as was shown by Shipley \cite{shi}.
\end{rem}

\subsection{$\om$-categories in a category $\CC$}
Just as an $\om$-category can be viewed as an $\om$-category in sets and $Pic _\om$ consists of $\om$-categories in abelian groups, we can define $\om$-categories in any category $\CC$ with fibered products, and when $\CC$ is abelian, we generalize Theorem \ref{26oct1}.

\begin{Def} Let $\CC$ be any category with fibered products.  An $\om$-category in $\CC$ is an object $X$ of $\CC$ with maps $s_n $, $t_n : X \lra X$ for $n \in \N$ and compositions $X \times _X X \stackrel{*_n}{\lra} X$ satisfying the axioms in definition~\ref{16nov1} (where $X \times _X X$ is the fibered product with respect to $t_n$ and $s_n$).  Morphisms between $\om$-categories in $\CC$ are simply morphisms in $\CC$ commuting with all source, target, and composition maps.  We denote the category of $\om$-categories in $\CC$ by $\CC _\om $.
\end{Def}

\begin{lem} For any abelian category $\CC$ with infinite direct sums, $Ch^{+} (\CC)$ and $  \CC_\om $ are equivalent.  
\end{lem}
\begin{proof} We sketch a proof and leave the details to the reader.  First, we define a functor $P: Ch ^+ (\CC) \lra \CC _\om$ as follows.  Let $A$ be a complex in $Ch ^+ (\CC)$, and let $ P(A) = \bigoplus _{n=0} ^\infty A^i$.  



We must show that $B = P(A)$ is an $\om$-category in $\CC$.  We define source and target maps $s_n$, $t_n : \bigoplus _{n=0}^\infty A^i \lra \bigoplus _{n=0}^\infty A^i $ as follows.  First, let $s_n ^{i,j}$, $t_n ^{i,j} : A^i \lra A^j$ be given by  
\begin{displaymath} 
s_n ^{i,j} := 
\left\{
\begin{array}{ll}
1 & \textrm{ if $n \geq i =j$}\\
0     & \textrm{ otherwise}
\end{array} 
\right. 
\end{displaymath}

\begin{displaymath} 
t_n ^{i,j} := 
\left\{
\begin{array}{ll}
1  & \textrm{ if $n\geq i = j$}\\
d    & \textrm{ if $n =j = i-1$}\\
0     & \textrm{ otherwise}.
\end{array} 
\right. 
\end{displaymath}
Now, let $s_n ^i $ be the sum over $j$ of the compositions $A^i \stackrel{s_n ^{i,j}}\lra A^j \lra \bigoplus _{k=0}^\infty A^k$ and similarly for $t_n ^i$.  The morphisms $s_n ^i , \; t_n ^i: A^i \lra PA  $, $i \geq 0$, determine $s_n $, $t_n$.  Defining composition as $*_n  = \pi _2 + \pi _1 - s_n \pi _1$, one may verify that $(*_n , t_n , s_n) _{n\geq 0}$ satisfies conditions of Definition \ref{16nov1}.  For condition (2d) of Definition \ref{16nov1}, the statement for $\om$-categories in $\CC$ should read: $*_i (*_j \pi _1 \times *_j \pi _2) = *_j (*_i \pi _1 \times *_i \pi _2) ((\pi_1 \pi _1 \times \pi _1 \pi _2 )\times (\pi _2 \pi _1 \times \pi _2 \pi _2))$ when restricted to the appropriate subobject of $(B\times B) \times (B \times B)$.  Our definition of composition $*_n$ can be extended to $B \times B \lra B$, and one may check that $*_i (*_j \pi _1 \times *_j \pi _2)$ agrees with $ *_j (*_i \pi _1 \times *_i \pi _2) ((\pi_1 \pi _1 \times \pi _1 \pi _2 )\times (\pi _2 \pi _1 \times \pi _2 \pi _2))$ on $(B\times B) \times (B \times B)$.  Hence, they also agree on the appropriate fibered product.  Therefore, $B$ is an $\om$-category in $\CC$.  \\

Let $B = P(A)$, $D = P(C)$ for $A,C \in Ch ^+ (\CC)$.  The fact that $Hom_{\CC _\om }(B,D) \simeq Hom_{Ch^+(\CC)}(A,C)$ follows easily from a few observations.  Let $f $ be a morphism from $B $ to $D$.  First, since $fs_n = s_n f$, an inductive proof shows that $f(A^i) \lra D$ factors through $C^i \lra D$.  The fact that $f$ commutes with all $t_n $ shows that the induced maps $A^i \lra C^i$ commute with the differentials $A^i \stackrel{d}\lra A^{i-1} $ and  $C^i \stackrel{d}\lra C^{i-1} $.  We conclude that a morphism $f:B\lra D$ is equivalent to a morphism from $A$ to $C$ in $Ch^+(\CC)$.  Therefore, $P$ is fully faithful.  \\

We now show that $P$ is essentially surjective.  Define $Q : \CC _\om \lra Ch^+ (\CC) $ as follows. Given $B \in \CC _\om$, let $A = Q(B)$ be given by letting $A^n $ be the cokernel $B_n / B_{n-1}$ of the monomorphism $B_{n-1} \lra B _n$, where $B_n $ is the image of $s_n : B \lra B$.  The morphism $t_{n-1} - s_{n-1} : B_n \lra B_{n-1}$  induces a morphism from $B_n /B_{n-1} \lra B_{n-1}$, and we denote the composition with $B_{n-1} \lra B_{n-1} / B_{n-2}$ by $d = A^n \lra A^{n-1}$.  Let us see that $PQ B \simeq B$.  Let $f : B _n \lra B _n$ be $f = 1 - s_{n-1}$.  Then $f$ induces a morphism $\overline f : B_n / B_{n-1} \lra B_n$ such that $\pi \overline f = 1$.  Hence, $B_n \simeq B_n / B_{n-1} \oplus B_{n-1}$.  It is now clear that $PQ B \simeq B$.     
\end{proof}

\section{The Dold-Kan Correspondence}\label{dksection}

We begin by laying out basic definitions and notations which can be found in any standard text on the subject, such as \cite{goj}.  Let $\De $ denote the category of ordinals, with objects $[n]= \{ 0,1,...,n \}$ for $n \in \N$ and morphisms the (non-strictly) increasing set morphisms between them.  A \emph{simplicial set} is a functor $X: \De ^{op} \lra Set$.  We define \emph{r-simplices} in a simplicial set $X$ to be the set $X_r := X([r])$.  We let $\De ^n$ denote the simplicial set $Hom_{\De }(- , [n])$ and denote an $r$-simplex $\al : [r] \lra [n] $ by listing $(\al (0),\al(1),...,(\al(r))$.  For a simplicial set $X$, we let $d_i $, $s_i$ denote the face and degeneracy maps $X(\partial _i )$ and $X(\sigma _i)$ respectively, where $[n-1] \stackrel{\partial _i}\lra [n]$ is the morphism which skips only $i$, and $\sigma _i : [n]\lra [n-1]$ is the morphism which repeats only $i$.   For a category $\CC$, a \emph{simplicial object in $\CC$} is a functor from $\De ^{op} $ to $\CC$, and the category of simplicial objects in $\CC$ is denoted simply by $s\CC$.  A \emph{simplicial abelian group} is a simplicial object in the category $Ab$ of abelian groups.   \\

The Dold-Kan correspondence was discovered independently by Dold and Kan and can be found originally in \cite{dk} as well as a number of other references such as \cite{goj}, \cite{wei}.  

\begin{thm} If $\CC$ is an abelian category, there is an equivalence $K: s\CC \lra Ch^+ (\CC)$.  
\end{thm}  

$K: s\CC \lra Ch^+ (\CC)$ is given by $K(A)_n = \bigcap _{i=0} ^{n-1} Ker d_i$, and the differential $d: K(A)_n \lra K(A)_{n-1}$ is $d = (-1)^n d_n$.  We will be particularly interested in the case when $\CC = Ab$ or sheaves of abelian groups on some site.

\subsection{$\omega$-categories and quasicategories}  

To extend the idea of the nerve of an ordinary category, Street defines in \cite{str} the nerve of an $\om$-category, which defines a functor $N : \om Cat \lra sSet$.  We first review some background on parity complexes and the Street-Roberts conjecture.  The results presented in this section are a summary of some of the results in \cite{ver}.  The original nerve construction is can be found in \cite{str}, and the ideas were streamlined using the language of parity complexes in \cite{str2, str3}.

\subsubsection{Basics of Parity Complexes}\label{paritysection}

\begin{Def}  A \emph{pre-parity} is a graded set $C = \bigsqcup _{n=0} ^\infty C_n$ and a pair of operations sending $x \in C_n$ to $x^- \subset C_{n-1}$ and $x^+ \subset C_{n-1}$, called negative and positive faces of $x$ respectively.  If $x \in C_0$, we take $x ^- = x^+ = \emptyset$ by convention.  We also say that for $x \in C_n$, a face $a \in x^-$ has parity $1$ (odd) and $a\in x^+ $ has parity $0$ (even). Elements in $C_n$ are said to be n-dimensional.  
\end{Def} 

A \emph{Parity complex} is a pre-parity complex satisfying some additional axioms delineated in \cite{ver, str2}.  The additional technical assumptions do not conern us because the pre-parity complexes which we deal with here are all parity complexes.  \\

For a parity complex $C$ and $S \subset C$, let $|S|_n = \bigcup _{k=0} ^{n} S_k$, where $S_k = S \cap C_k$.  For $S \subset C$ and $\xi \in \{+,- \}$, let $S^\xi = \bigcup _{x\in S} x^\xi$, and let $S^\mp = S^- \setminus S^+$ and $S^\pm = S^+ \setminus S^-$. \\
 
 If $C$ is a graded set, we let $\NN (C)$ denote the $\om$-category with underlying set $\{(M,P) \st M, P \; \textrm{are finite subsets of} \; C\}$.  Source, target and compositions are given by 
\bi $s_n(M,P) = (|M|_n , M_n \cup |P|_{n-1})$
\i $t_n(M,P) = (|M|_{n-1} \cup P_n, |P|_n)$
\i $(N,Q)*_n(M,P) = (M \cup (N \setminus N_n), Q \cup (P\setminus P_n)$.  
\ei

There is another $\om$-category $\OO (C)$ attained from a parity complex $C$, which we will now describe.  For a parity complex $C$ and subsets $S , T \subset C$, we say that $S \perp T$ if $(S^+ \cap T^+) \cup (S^- \cap T ^-) = \emptyset$.  Another way to express this is to say that $S \perp T$ if $S$ and $T$ have no common faces of the same parity.  A subset $S $ of $C$ is called \emph{well-formed} if it has at most one 0-dimensional element and for distinct elements $x,y \in S$, $x\perp y$.  Define $\OO (C)$ to be the subcategory of $\NN (C)$ consisting of all $(M,P) \in \NN (C)$ such that $M$ and $P$ are both non-empty, well-formed subsets of $C$, $P = (M \cup M^+) \setminus M^- = (M \cup P^+) \setminus P^-$, and $M = (P \cup M^- )\setminus M^+ =  (P \cup P^- )\setminus P^+ $.  It is not immediately clear that $\OO (C)$ is an $\om$-category.  However, the work in \cite{str, str2} demonstrates that it is. \\

For a parity complex $C$, there are distinguished elements of $\OO (C)$.  Let $x \in C_n$.  We inductively define subsets $\pi (x) , \; \mu (x) \subset C$.  Let $\pi (x) _m = \mu (x) _m = \emptyset$ for $m >n$, let $\pi (x) _m = \mu (x) _m = \{ x \} $ for $m=n$, and let $\mu (x) _m = \mu (x) _{m+1} ^\mp$ and $\pi (x) _m = \pi (x) _{m+1} ^\pm$ for $0 \leq m <n$.  Then the element $<x> : = (\mu (x), \pi (x) ) \in \OO (C)$ is called an \emph{atom}.  Let $<C> = \{ <x> \st x \in C \}$.  Street proved \cite{str, str2} that $<C>$ freely generates $\OO (C)$ in the sense defined below. First we introduce some notation.  For $n \in \N$ and $B \in \om Cat$, let $|B|_n$ denote the $n$-category $(s_n B , *_i , s_i , t_i)_{0\leq i \leq n}$.

\begin{Def} Let $A$ be an $\om$-category and $G$ a subset of its elements, with grading $G_n = G \cap A_n$.  
\ben $A$ is \emph{freely generated} by $G$ if for all $\om$-categories $B$, all functors $f : |A|_n \lra B$ of $\om$-categories and maps of sets $g : G_{n+1} \lra B$ such that $s_n g(x) = f(s_n x)$ and $t_n g(x) = f(t_nx)$ for all $x\in G_{n+1}$, there exists a unique functor $\hat f: |A|_{n+1} \lra B$ of $\om $-categories such that $\hat f _{||A|_n } = f$ and $f_{|G_{n+1} } = g$.  
\i $A$ is \emph{generated} by $G$ if for each $n \geq 0$, $|A|_{n+1}$ is the smallest sub-$\om$-category of $A$ containing $|A|_n \cup G_{n+1}$.
\een
\end{Def}

If $A$ is freely generated by $G$, then $A$ is generated by $G$ (\cite{ver}).  

For Parity complexes $C$, $D$, a map of sets $f : C \lra D$ which respects the grading induces a morphism $\NN (f) : \NN (C) \lra \NN (D)$, sending $(M, P) $ to $(f(M) , f(P))$.  Let us consider only graded maps of sets $f: C \lra D$ such that 
\bi for all $x \in C_0$, $f(x) \subset D_0$ is a singleton set, and 
\i for all $n\geq 0$ and $x \in C_{n+1}$, $f(x)$ is well formed, $f(x^+ ) = (f(x^-) \cup f(x)^+ ) \setminus f(x)^-$, and  $f(x^- ) = (f(x^+) \cup f(x)^- ) \setminus f(x)^+$.
\ei 
Parity complexes together with graded set maps $f : C \lra D$ with these two properties form a category $Parity$ of parity complexes.  The two conditions are chosen so that the functor $\NN : Graded \; Sets \lra \om Cat$ restricts to a functor $\OO : Parity \lra \om Cat$.  \\

Of particular interest are the parity complexes $\tilde \De ^n$, which we now define.  r-dimensional elements of $\tilde \De ^n$ are subsets $\underline v = \{v_0 < v_1 <...<v_r \} $ of $[n] := \{ 0,1,...n\} \subset \N$ of size $r+1$.  We will often denote such a $\underline v \in \tilde \De ^n _r $ by $(v_0v_1...v_r)$.  The i-th face of $\underline v \in \tilde \De ^n _r$, denoted $\de _i \underline v = \{v_0 , ..v_{i-1}, \hat v_{i}, v_{i+1},...,v_r \} \in \tilde \De ^n _{r+1}$, where $\hat v_i $ denotes omission of $v_i$.  Now define the face operators $\underline v ^\xi = \{ \de _i \underline v \st i \in [r] \; \textrm{and $i$ is of parity} \; \xi \}$ for $\xi \in \{+,-\}$.  A morphism $[n] \stackrel{\al}\lra [n] $ in $\De$ induces a morphism $\tilde \De (\al) : \tilde \De ^m \lra \tilde \De ^m$ sending $\underline v \in \tilde \De ^m _r$ to $\emptyset$ if $\al v_i = \al v_{i+1} $ for some $i$ and to $\al \underline v = \{ \al v_0 , ..., \al v_r \} $ otherwise.  Thus, $\tilde \De $ is a functor from $\De$ to $Parity$, and we obtain the composition $\De \stackrel{\tilde \De }\lra Parity \stackrel{\OO}\lra \om Cat$.  The $\om$-category $\OO (\tilde \De ^n )$ is called the n-th oriental and has a unique non-identity n-morphism $\langle(01..n) \rangle $.  For a morphism $[m] \stackrel{\al}\lra [n]$ in $\De$, $\OO (\tilde \De (\al))$ maps $\langle\underline v \rangle $ to $\langle\al \underline v \rangle$.  \\

The product of parity complexes was shown in \cite{str2} to be a parity complex.  For parity complexes $C$, $D$, let $(C \times D)_n = \bigcup _{p+q = n} C_p \times D_q$, and for $\xi \in \{ + , -\}$, $(x,y)^\xi = x^\xi \times \{y\} \cup \{x\} \times y ^{\xi (p}$, where $\xi (p ) = \xi $ if $p$ is even and has the opposite parity of $p$ is odd.  

\subsubsection{The Nerve of an $\om$-Category and the Street-Roberts Conjecture}

In \cite{str}, Street defines the nerve functor $N: \om Cat \lra sSet$ with left adjoint $F_\om$. 
The nerve of an $\om$-category $A$ consists of composing $\OO \tilde \De $ with the Yoneda embedding $\om Cat \lra Set$.  More explicitly, The n-simplices of $NA$ are $Hom_{\om Cat} (\OO (\tilde \De ^n), A)$.  For an n-simplex $x : \OO (\tilde \De ^n) \lra A$ and morphism $\al : [m] \lra [n] $ in $\De$, $\al ^* x \in NA _m$ is the composition of $x$ with $\OO (\tilde \De (\al ))$.  The left-adjoint $F_\om : sSet \lra \om Cat$ is the left Kan extension of $\OO \circ \tilde \De : \De \lra \om Cat$ along the Yoneda embedding $Y : \De \lra sSet$.  For a simplicial set $X$, $F_\om (X)$ is characterized by the following property.  For each n-simplex $x \in X_n$, there is a a functor $\io _x : \OO (\tilde \De ^n) \lra F_\om (X)$ of $\om$-categories such that for any morphism $\al : [m] \lra [n]$ in $\De$, $\io _{\al ^* x} = \io _x \circ \OO (\tilde \De (\al))$, and for any other $\om$-category $A$ with such a family of maps ${j_x : \OO (\tilde \De ^n)\lra A }$, $n \in \N $, $ x \in X_n$, $j$ factors through $\io$.  For $X \in sSet$ and $x \in X_n$, let $[\![x ]\!] = \io _x (\langle01..n\rangle)$. \\

To get an idea of what the nerve of an $\om$-category looks like, an n-simplex of $NA$ looks like a drawing of an n-simplex in the $\om$ category $A$, meaning an n-simplex labeled with an n-morphism in $A$ and k-dimensional faces are labeled with k-morphisms in $A$.  It is an easy exercise to check that $NA _0 = A_0$, $NA_1 = A_1$.  A 2-simplex $x \in NA _2$ is a functor of $\om$-categories $x : \OO (\tilde \De ^2) \lra A$, which consists of a 0-objects $x(\langle0\rangle), \; x(\langle 1\rangle), \; x(\langle 2\rangle) \in A_0$, 1-morphisms $x(\langle i\rangle) \stackrel{x(\langle ij\rangle)}\lra x(\langle j\rangle)$ in $A_1$ for $i,j \in [2]$, and a 2-morphism $x(\langle 012\rangle) \in A_2$ such that $s_1 x(\langle 012\rangle) = x(\langle 02\rangle)$ and $t_1 x(\langle 012\rangle) = x(\langle 12\rangle)*_0 x(\langle 01\rangle)$.  \\ 

When we restrict to $Pic _\om$, Theorem \ref{dk1} guarantees that $N : Pic_\om \lra sAb$ is an equivalence.  In general, however, $N : \om Cat \lra sSet$ is not an equivalence.  The problem is that viewing an $\om$-category as a simplicial set by taking its nerve loses some information.  The simplicial set no longer remembers which n-simplices represent identity morphisms and so it forgets how to compose morphisms.  To remedy this situation, in \cite{str, rob}, Street and Roberts modify the modify the nerve construction to take values in the category $Cs$ of ``complicial sets.''  A complicial set is a simplicial set $X$ together with a collection of simplices $tX$ called thin simplices which satisfy certain axioms.  To name a few, 
\bi No 0-simplex of $X$ is in $tX$,
\i the only 1-simplices in $tX$ are degenerate 1-simplices,
\i the degenerate simplices of $X$ are in $tX$,
\i and for each $(n-1)$-dimensional k-horn for $n\geq 2$, $0<k<n$ has a unique thin filler.
\ei

The other properties can be found in \cite{ver}.  A morphism of complicial sets $f: (X,tX) \lra (Y, tY)$ is a morphism $f: X \lra Y$ of simplicial sets such that $f(tX) \subset tY$.  

\begin{rem}Complicial sets is a full subcategory of a larger category $Strat$ of \emph{stratified sets} whose objects are pairs $(X,tX)$ but which are not required to satisfy all of the axioms listed above for complicial sets.  Morhphisms, of course, are simply morphisms of simplicial sets which preserve thin simplices. There is a natural way of taking the product $\otimes$ of two stratified sets, where the underlying simplicial set of $X\otimes Y$ is $X\times Y$.  For instance, the thin r-simplices in $\De ^n \otimes \De ^1$ are the simplices $(x,y)\in \De ^n _r \times \De ^1 _r $ such that $x $ is degenerate at some $0\leq j <r$ and $y$ is degenerate at some $k\geq j$.
\end{rem}

The enhanced nerve construction $\fN : \om Cat \lra Cs$ sends $A$ to $(NA , tNA)$, where the thin n-simplices in $NA$ are the simplices $x : \OO (\tilde \De ^n) \lra A$ such that $x( \langle 01...n\rangle) $ is an $(n-1)$-morphism.  Composing $\fN$ with the forgetful functor $Cs \lra sSet$ ($(X, tX) \mapsto X$) gives the original nerve construction.  The nerve $\fN$ has a left adjoint $\mathfrak F _\om$ so that $\FF _\om ((X,tX))$ is attained from $F_\om (X)$ by ``collapsing'' morphisms corresponding to thin simplices, a process described in detail in \cite{ver}.  Theorem \ref{src}, known as the Street-Roberts conjecture, was proven by Verity in \cite{ver}.

\begin{thm}\label{src} $\fN : \om Cat \lra Cs$ is an equivalence of categories. 
\end{thm}

\begin{rem}  More recently, in \cite{nik}, Nikolaus defines a model category of \emph{algebraic Kan complexes} similar to the category $Cs$, which specifies a distinguished filler for each horn. He shows that algebraic Kan complexes is Quillen equivalent to simplicial sets.  
\end{rem}

\subsection{The Dold-Kan Triangle} The Dold-Kan correspondence \cite{dk} gives an equivalence between $Ch^+({A}b)$ and $sAb $, simplicial objects in abelian groups (or equivalently, abelian group objects in $sSet$).  Furthermore, $sAb$ and $Ch^+ (Ab)$ have model structures.  The model structure on $sAb$ is induced by the forgetful functor $U : sAb \lra sSet$.  Specifically, $sAb$ inherits the weak equivalences and fibrations from $sSet$; $f$ is a weak equivalence in $sAb$ if and only if $U(f)$ is a weak equivalence in $sSet$, and $f$ is a fibration in $sAb$ if and only if $Uf$ is a fibration in $sSet$.  The model structure on $Ch^+ (Ab)$ has quasi-isomorphisms as the weak equivalences, degree-wise epimorphisms (in positive degree) as the fibrations, and degree-wise monomorphisms with projective cokernels as cofibrations.  Additionally, $Pic_\om$ inherits a model structure from $Ch^+(Ab)$ via the equivalence $Ch^+(Ab) {\lra} Pic_\om$.  The weak-equivalences in $Pic _\om$ are morphisms which are equivalences of the underlying $\om$-categories.  The Dold-Kan correspondence is in fact an equivalence of model categories, as is explained in \cite{scs}.  We have seen that $Pic _\om \simeq Ch^+(Ab) \simeq sAb$ as model categories, but also the following theorem of Brown relates these two correspondences in the following way.   

\begin{thm}\label{dk1} (\cite{bro}, \cite{ncat}) The composition $N\circ {P} : Ch^+ ({A}b) \lra  sAb $ is the same as the Dold-Kan correspondence.  In other words, the following diagram commutes up to isomorphism.  
\[
\begin{CD}
 Ch^+ ({A}b)    @> {P}  >>       Pic_\om \\
@VV D V        @VV N V \\
sAb    @=   sAb
\end{CD}
\]
where $D$ denotes the Dold-Kan correspondence.  
\end{thm}

Since the Dold-Kan correspondence sends $X  \in  sAb$ to $A^\bullet$, where $A^n = \bigoplus _{i=0}^{n-1} Ker d_i $, it is now clear from the comments in section \ref{17nov2} that $N\inv : sAb \lra Pic_\om$ satisfies 
$N\inv (X ) \simeq \bigoplus _{n=0} ^\infty \cap _{i=0} ^{n-1} Ker d_i$. \\ 

In the Dold-Kan correspondence quasi-isomorphisms correspond to weak equivalences in $sSet$, and by Proposition~\ref{qisprop}, quasi-isomorphisms correspond to equivalences of Picard $\om$-categories.  Under the equivalence $N: Pic_\om \lra  sAb$, equivalences of $\om$-categories correspond to weak equivalences in $sAb$.  Moreover, upon taking the geometric realization : $|\, \cdot \, |:  sAb \lra Top$, the equivalences of $\om$-categories are identified with weak equivalences of topological spaces.  Localizing with respect to weak equivalences, we have an embedding of $Ho(Pic _\om)$ into the homotopy category of topological spaces.

\section{The Dold-Kan Correspondence for Sheaves}\label{sheafdksection} 
Throughout the next two sections, fix an essentially small site $\mathcal S$ with enough points equipped with a Grothenieck topology, such as the category of manifolds with the Etale topology or open sets on a fixed manifold $X$.  Henceforth, let ``prehseaf'' mean a presheaf on $\mathcal S$, i.e. a functor from $\mathcal S ^{op}$ into some category. 

\subsection{Definitions}
 
\begin{Def}\label{cechdescentdef} For a presheaf $F$ with values in model category $\MM$, an object $X $ of $\mathcal S$ and an open cover $\UU = \{U_i\} $ of $X$, let $\check F _\UU$ denote the cosimplicial diagram
\[
\prod F (U_i) \Rightarrow \prod F (U_{ij}) \Rrightarrow F (U_{ijk})...
\]
in $\MM$.  We write $\check F = \check F _\UU$ when the open cover is understood.  We say that $F$ satisfies \emph{\v Cech descent with respect to $\UU$} if the natural map $F(X)\lra holim \check F _\UU$ is a weak equivalence in $\MM$. We say that $F$ satisfies \emph{\v Cech descent} if $F$ satisfies \v Cech descent with respect to all objects $X \in \mathcal S$ and all open covers $\UU $ of $X$.  
\end{Def} 


Let $X$ be an object of $\mathcal S$, which we think of as a discrete presheaf of simplicial sets.  The concept of \emph{hypercover} $U \lra X$ is defined precisely in \cite{dhi}.  Informally, we may think of it as a resolution of a \v Cech cover of $X$.  Notice that for a hypercover $U\lra X$, and presheaf $F$ with values in model category $\MM$, $F(U)$ is a cosimplicial diagram in $\MM$ since $U$ is a simplicial diagram in $\mathcal S$, and we have a morphism $F(X) \lra F(U) $ in $\MM ^\De$, where $X$ is considered as a constant diagram.  
 
\begin{Def}\label{descentdef} Let $U \lra X$ be a hypercover and $F$ be a presheaf with values in model category $\MM$.  We say that $F$ satisfies \emph{descent with respect to $U \lra X$} if $F(X) \lra holim F(U)$ is a weak equivalence in $\MM$.  We say that $F$ satisfies \emph{descent} if it satisfies descent with respect to all hypercovers.
\end{Def}

By ``simplicial presheaf'' we mean a presheaf with values in $sSet$.  Let $\tps$ denote simplicial presheaves which are levelwise sheaves of sets (i.e. simplicial objects in sheaves of sets).  In general, for a category $\CC$, we denote presheaves on $\mathcal S $ with values in $\CC$ by $Pre _\CC$, and we let $\check Sh _\CC$ denote those presheaves which satisfy \v Cech descent and $Sh _\CC$ denote those satisfying descent, provided that $\CC$ is a model category.  For shorthand we write $Pre _\om$ for $Pre _{\om Cat}$ and $Pre_{\om Ab}$ for $Pre _{Pic _\om}$.  

\begin{rem}It was shown in \cite{dhi} that , $\cshs$ are the presheaves which satisfy descent for all bounded hypercovers and that there exist presheaves satisfying \v Cech descent but not descent for all hypercovers.    
\end{rem}


\begin{prop} For a presheaf $F$ of simplicial abelian groups, $F$ satisfies (\v Cech ) descent if and only if $UF: \mathcal S ^{op} \lra sSet$ satisfies (\v Cech) descent.
\end{prop}
\begin{proof} Let $D :\Delta ^{op}  \lra \mathcal S$ be a simplicial diagram associated to a \c Cech complex $\check C _\UU \lra X$ (i.e. the \v Cech nerve of an open cover of some $X \in \mathcal S$).  We know that $F(X) \lra holim(FD)$ is a weak equivalence if and only if $UF(X) \lra Uholim(FD)$ is a weak equivalence. We would like to show that $F(X) \lra holim FD$ is a weak equivalence in $sAb$ if and only if $UF(X) \lra holim(UFD)$ is a weak equivalence in $sSet$.  By the two out of three property of weak equivalences, this is true provided that $Uholim(FD) \lra holim (UFD)$ is a weak equivalence in $sSet$.  Thus, to complete the proof, it suffices to show that $U(holim FD) \lra holim (UFD)$ is a weak equivalence.   Since $U : sAb \lra sSet$ is the right adjoint in a Quillen pair $( \bbz [ - ], U)$, it naturally follows that for any diagram $G $ in $sAb$, $U(holim G) \lra holim (UG)$ is a weak equivalence.  
\end{proof}

\begin{rem} Since the category ${\ca } b$ of sheaves of abelian groups on $\mathcal S$ is abelian, the Dold-Kan correspondence provides an equivalences $Ch^+ (\ca b) \lra \tpa$ because simplicial objects in $\ca b$ are the same as $\tpa$.  
\end{rem}

\begin{prop}  Neither $\csha$ nor $\tpa$ is contained in the other.
\end{prop}

\begin{proof}
 To see this, consider the following example of a presheaf $F \in \ps$ which satisfies \v Cech descent but which is not a levelwise presheaf.  Let $\mathcal S = Op(X)$, open sets on a manifold $X$, and let $A_0$ be any non-zero abelian group.  Consider the presheaf of abelian groups $A$ such that $A(X) = A_0$ and $A (U) = 0$ if $U \neq X$ and the complex $A^* \in C ^+ (\ca b)$ which is $A$ in each degree and whose differential is the identity map on $A$.  The corresponding presheaf of simplicial abelian groups satisfies \v Cech descent but levelwise is not a sheaf of sets.\\

On the other hand, take any sheaf of abelian groups $A$ and consider the complex $A\lra 0$ in degrees $1$ and $0$.  The corresponding presheaf ${P} (A \lra 0)$ of $\om$-categories and in fact a presheaf of 1-categories.  However, it does not satisfy descent for stacks.  Hollander shows in \cite{hol} that a stack satisfies descent if and only if its nerve satisfies \v Cech descent as a simplicial presheaf.  Hence, $ N {P} (A \lra 0)$ is a presheaf of simplicial abelian groups which is levelwise a sheaf but does not satisfy \v Cech descent, showing that neither condition implies the other.    
\end{proof}

\subsection{Model Structures and Derived Dold-Kan}
There are several model structures on simplicial presheaves. There are, of course, the projective and injective model structures \cite{lur, hir}, which we denote by $Pre_{sSet}^{proj}$, $Pre_{sSet}^{inj}$ respectively.  Weak equivalences are the sectionwise weak equivalences.  In the projective model structure, fibrations are the sectionwise fibrations, and in the injective model structure, the cofibrations are the sectionwise cofibrations.  For each of these, one can take the left Bousfield localization $Pre_{sSet} ^{loc,inj}$ and $Pre_{sSet}^{loc,proj}$ at the hypercovers.  The existence of the localalization $Pre_{sSet} ^{loc,inj}$ follows from the work of Jardine \cite{jar}, and the construction of the local projective model structure is due to Blander \cite{bla}.  The weak equivalences in $Pre_{sSet}^{loc,proj}$ and $Pre_{sSet}^{loc,inj}$ are the stalkwise weak equivalences of simplicial sets since $\mathcal S$ has enough points \cite{jar2}.   The important feature of the local model structures is that in  $Pre_{sSet}^{loc,inj}$, the fibrant objects are the presheaves which are fibrant in $Pre_{sSet}^{inj}$ and satisfy descent for all hypercovers.  Fibrant objects in   $Pre_{sSet}^{loc,proj}$ are the ones which are sectionwise Kan complexes and satisfy descent for all hypercovers.  \\

Jardine shows the existence of a model structure on $Pre_{sAb}$ such that a morphism is a weak equivalence or fibration if and only if it is a weak equivalence or fibration in $Pre_{sSet} ^{loc,inj}$ \cite{jarch}.  From this, the next two results follow easily.  First we see that for any presheaf of simplicial abelian groups, there exists a sheafification (i.e. local fibrant replacement) which is also a presheaf of simplicial abelian groups.  The second result states that there is a derived Dold-Kan correspondence.     

\begin{lem}\label{brisu} Let $U : Pre _{sAb} \lra Pre_{sSet}$ denote the forgetful functor.  For every $ X\in \pa$, there is a map $X \stackrel{f}\lra Y$ in $\pa$ such that $Uf$ is a weak equivalence and $UY$ is a fibrant object in $Pre_{sSet}^{loc,inj}$ and $Pre_{sSet}^{loc,proj}$.
\end{lem}
\begin{proof} 
If $X \in \pa$, take a fibrant replacement $ X \lra Y$ for $X$ in $ Pre_{sAb}$ in the model structure of \cite{jarch} described above.  $UY \in  Pre_{sSet}^{loc,inj}$ is fibrant since $U :  Pre_{sAb} \lra  Pre_{sSet}^{loc,inj}$ preserves fibrations.   Additionally, since $UY$ is fibrant in $Pre_{sSet}^{loc,inj}$, it satisfies descent for all hypercovers.  Since it is a presheaf taking values in $sAb$, it is sectionwise fibrant.  Therefore, $UY$ is also fibrant in $Pre_{sSet}^{loc,proj}$.
\end{proof}

\begin{prop} Let $P^\prime$ denote the full subcategory of $Pre_{sAb}$ spanned by objects satisfying descent for hypercovers.  Localizing at local weak equivalences, we can form the homotopy categegory $Ho(P^\prime)$, and $Ho(P^\prime)$ is equivalent to $D^+(\ca b)$, the derived category of chain complexes of sheaves of abelian groups in non-negative degrees.  
\end{prop}
\begin{proof} First observe that the inclusion $\tilde Pre_{sAb} \leftrightarrows Pre_{sAb}$ and the levelwise sheafification functors descend to equivalences of homotopy categories since weak equivalences in Jardine's model structure on $Pre_{sAb}$ are the local weak equivalences.  Because $Ch^+(\ca b)$ is equivalent to $\tilde Pre_{sAb}$,  we need only show that $Ho(Pre_{sAb})$ is equivalent to $Ho(P^\prime)$ to complete the proof.  Since the forgetful functor $U: Pre_{sAb} \lra Pre_{sSet} ^{proj, loc}$ preserves fibrant objects, as was noted in the proof of Lemma \ref{brisu}, the full subcategory $P_{cf}$ of cofibrant fibrant objects is contained in $P^\prime$.  It is easy to check that since $P_{cf} \subset P^\prime \subset Pre_{sAb}$, $Ho(P^\prime)$ exists and is equivalent to $Ho(Pre_{sAb})$.
\end{proof}

\begin{rem} Consider the case of a simplicial presheaf on a topological space $X$.  In general it is a stronger requirement on a simplicial presheaf on $X$ to satisfy descent for all hypercovers than it is to satisfy \v Cech descent.  Lurie explains in \cite{lur} that if $X$ has finite covering dimension, then the two conditions are the same.   However, we are interested in presheaves on manifolds, all of which have finite covering dimension.  If we consider sheaves on the site of all differentiable manifolds, then the result is unchanged since we are only considering the hypercovers of \cite{dhi} rather than the most general hypercovers.
\end{rem}

\section{Omega Descent}\label{ommdescentsection}
The standard definition of descent, used in \cite{hol, jar, dhi}, is given in Definitions \ref{cechdescentdef} and \ref{descentdef}.  In this section, however,  we describe a glueing condition for presheaves of $\om$-categories and show that in the case when the presheaf takes values in $Pic _\om$, it coincides with the homotopy limit descent condition.  This is an attempt to expand on the work of Hollander \cite{hol}, who showed that descent for 1-stacks can be described in a homotopy theoretic way which is consistent with descent for simplicial presheaves.  Our definition of the glueing condition is motivated by Breen's description of descent for 2-stacks \cite{bre}.  The idea is that a sheaf $\ca$ of $\om$-categories on $\mathcal S$ satisfies the glueing condition if one can glue 0-objects, 1-objects, and k-objects for any $k \geq 0$. \\

 Informally, glueing of 0-objects has the following meaning.  Given an open cover $\UU = \{ U _i\} _{i\in I}$ of $X \in \mathcal S$ the data for glueing of 0-objects consists of 0-objects $a_i \in \ca (U_i) _0$ which are identified on intersections via 1-morphisms  $a_{ij} \in \ca (U_{ij}) _1$, $a_{ij} : (a_i)_{|ij} \lra (a_j)_{|ij}$, the 1-morphisms $a_{ij}$ are identified on triple intersections via 2-morphisms $a_{ijk} : a_{jk} *_0 a_{ij} \Longrightarrow a_{ik}$ and so on.  The glueing condition for 0-objects states that for any such system $(\{a_i\}_{i \in I} , \{a_{ij}\}_{i,j \in I} , \{a_{ijk}\} ,...)$, there exists an 0-object $x \in \ca (X)$--unique up to isomorphism--with isomorphisms $x_{|U_i} \lra a_i$ which fit together in a consistent way.  In other words, the system can be glued to a global 0-object in an essentially unique way. \\ 

For $\ca$ to satisfy the glueing condition, it must satisfy the glueing condition for 0-objects, and for each pair of sections $x,y \in \ca (X)_k$, the presheaf of $\om$-categories $Hom_\ca  ^k(x,y)$ satisfies the glueing condition for 0-objects.  \\
    
To make this description formal,  Let $\YY : \Delta \lra sSet$ be the Yoneda embedding.  A system $(a_i \in \ca (U_i)_0 , a_{ij} \in \ca (U_{ij})_1,...)$ as above can be thought of as a morphism $a \in Hom_{sSet ^\Delta} (\YY , \check {N\ca})$.   Here, $N$ is the nerve functor so that $N\ca$ is a simplicial presheaf, and $\check{N\ca}$ is the diagram from Definition \ref{cechdescentdef}.  There is a restriction map of the constant diagram $ \rho : N \ca (X) _{const} \lra \check{N\ca}$.  In order to compare $N\ca (X) _0$ with, $Hom_{sSet^{\De}}(\YY , \cna)$, we identify $N\ca (X)_0$ with $\{F \in Hom_{sSet^\De}(\YY , N\ca (X)_{const}) \st F(\De ^n ) = F (\De ^0) \; \textrm{for all} \; n\}$, so $N\ca (X)_0$ is a subset of $Hom_{sSet ^\De} (\YY , N\ca (X)_{const})$ and $\rho$ maps $ N\ca (X)_0$ to $ Hom_{sSet ^\De}(\YY , \cna)$. 

\begin{rem} In Street's definition of the descent \cite{str4}, he defines a descent object $Desc(\check N\ca) \in \om Cat$ object of $\ca \in Pre_\om$ with respect to the \v Cech complex for $\UU$.  Using the adjunction, $F_\om: sSet \leftrightarrows \om Cat : N$, $Hom_{sSet ^\De}(\YY , \cna)$ is identified with the 0-objects of $Desc(\check N\ca)$.   
\end{rem}

\begin{rem}\label{simpenrichment} The category of cosimplicial objects in $sSet$ is a simplicial model category \cite{goj}.  For $X,Y \in sSet ^\De$, the enrichment over simplicial sets is given by $hom(X,Y) _n = Hom_{sSet ^\De}(X \times \De ^n, Y)$, where $hom(X,Y)\in sSet$. For $K \in sSet$, $X \in sSet ^\De$, $X\times K \in sSet ^\De$ is $(X\times K)_n = X_n \times K$, and for any $Y \in sSet ^\De$, $Hom_{sSet^\De} (X\times K, Y) \simeq Hom_{sSet}(K, hom(X,Y))$.  The homotopy of Definition \ref{0glue} is really a homotopy $H : \De ^1 \lra hom(\YY , \cna)$ between 0-simplices $F$ and $\rho G$ in $hom(\YY , \cna)$.  The simplicial set $hom (\YY , \cna)$ is commonly called the total space $Tot(\cna)$ of $\cna$.    
\end{rem}

\begin{Def}\label{0glue} Let $\ca $ be a presheaf of $\om$-categories on $\mathcal S$, and let $\UU = \{ U _i\} _{i\in I}$ be an open cover of $X \in \mathcal S$. We say that $\ca$ satisfies \emph{0-glueing} with respect to $\UU$ if for all $F \in Hom_{sSet ^\De}(\YY, \cna)$, there exists a homotopy $H : Hom_{sSet ^\De}(\YY \times \De ^1 , \cna)$ from $F$ to $\rho G$ for some $G \in N\ca (X)_0$. We say that $\ca$ also satisfies \emph{unique 0-glueing} with respect to $\UU$ if two 0-objects $a,b \in \ca (X)_0 = N \ca (X) _0$ are isomorphic in $\ca (X)$ whenever $\rho a$ and $\rho b$ are isomorphic in $\h$.
\end{Def}

The intuitive meaning of the uniqueness of glueing is that if $a,b \in \ca (X)_0$ are locally isomorphic in a consistent way, then $a,b$ are isomorphic. Hence if $G$ and $G^\prime \in \ca (X)_0$ glue $F \in \h$ in the notation of Definition \ref{0glue}, then $G$ and $G^\prime $ are isomorphic in $\ca (X)$.  \\

To describe $k$-glueing for $k>0$, first observe that there is a shift functor $[1]: \om Cat \lra \om Cat$, where if $A = (Ob(A),s_k , t_k, *_k)_{k \in \N}$, $A[1]$ is the $\om$-category $(Ob(A), s[1]_k = s_{k+1} , t[1]_k = t_{k+1}, *[1]_k  = *_{k+1})_{k \in \N}$.  For $x,y \in A_0$, we are especially interested in sub-$\om$-categories $ A[1]^{x,y} = Hom_A (x,y):= \{a \in A[1] \st s_0a = x \; , \; t_0a = y \}$.  More generally, for $k\geq 1$, $x,y \in A_{k-1}$, let $A[k]^{x,y} =   \{a \in A[k] \st s_{k-1}a = x \; , \; t_{k-1}a = y \}$.  Observe that $(A[k])[1] = A[k+1]$.  

\begin{rem} $A[k]^{x,y} = Hom_A ^k (x,y) $ from Definition \ref{homkdef}.
\end{rem}   

\begin{Def} Let $\ca$ be a presheaf of $\om$-categories, $X$ an object in $\mathcal S$ and $\UU$ an open cover of $X$.
\ben Suppose $k >0$.  Then $\ca$ satisfies \emph{the (unique) k-glueing condition with respect to $\UU$} if for $x,y \in \ca (X)_{k-1}$, $\ca [k] ^{x,y}$ satisfies (unique) 0-glueing whenever $s_{k-2}x = s_{k-2}y$ and $t_{k-2}x = t_{k-2}y$. We use the convention that $s_{-1}x = t_{-1}x = 0$ for all $x$.  
\i We say that $\ca$ satisfies \emph{$\om$-descent with respect to $\UU$} if for all $k\geq 0$, $\ca $ satisfies the unique k-glueing condition.  
\i We say that $\ca$ satisfies \emph{$\om$-descent for loops} if for all $x \in \ca (X)_0$, each $\ca [k] ^{x,x} $ satisfies the unique 0-glueing condition.
\een
\end{Def}

\begin{Def} Let $\ca$ be a presheaf of Picard $\om$-categories.  We say that $\ca$ satisfies \emph{$\om$-descent, k-glueing, etc.} if it satisfies $\om$-descent with respect to $\UU $, k-glueing with respect to $\UU$,  et cetera, respectively for all objects $X$ of $\mathcal S$ and open covers $\UU $ of $X$.  
\end{Def}
 
Our goal for the remainder of this section is prove the following theorem, which relates two notions of descent.   

\begin{thm}\label{equivdescent} Let $\ca$ be a presheaf site $\mathcal S$ with values in $Pic _\om$.  Then $\ca$ satisfies $\om$-descent if and only if it satisfies \v Cech descent.  
\end{thm} 

It is important to note that since $N: Pic _\om \lra sAb$ is an equivalence of model categories, a presheaf $\ca \in Pre _\om$ satisfies (\v Cech) descent if and only if its nerve $N\ca \in Pre_{sSet}$ satisfies (\v Cech) descent.  \\

To prove Theorem \ref{equivdescent}, we will show that for each open cover $\UU$, a presheaf $\ca$ of Picard $\om$-categories on $\mathcal S$ satisfies the unique glueing condition with respect to $\UU$ on $X$ if and only if it satisfies \v Cech descent with respect to $\UU$.  For the rest of the section, fix an object $X \in Ob(\mathcal S)$ and open cover $\UU = \{ U_i \}_{i \in I}$ of $X$.

\subsection{Loop and Path Functors}

We define a pair of adjoint functors $[-1] : Ch^+(Ab) \leftrightarrows Ch^+(Ab) : \Om$, which form a Quillen pair for the model category $Ch ^+ (Ab)$. The functor $[-1]$ is defined as follows.  Let $B \in Ch^+ (Ab)$.  Then $([-1]B)_i = B_{i-1}$ for $i>0$, and $([-1]B)_0 = 0$.  Thus, $[-1]B = ...\lra B_1 \lra B_0 \lra 0$, where $B_0 $ is in degree 1.  On the other hand, $\Om$ is defined by letting $(\Om A)_i = A_{i+1} $ for $i>0$, and $(\Om A)_0 = Ker(A_1 \stackrel{d}\lra A_0)$.   Thus, $([1]^{0,0} A) = ...\lra A_3 \lra A_2 \lra Ker(d)$, with $Ker(d)\subset A_1$ in degree $0$.  Clearly, $Hom_{ch^+(Ab)} ([-1]B,A) \simeq Hom_{ch^+(Ab)}(B, \Om A)$, so $[-1] $ is left adjoint to $\Om $. Since $[-1]$ preserves weak equivalences and cofibrations in $Ch^+ (Ab)$, $([-1], \Om )$ is a Quillen pair (Prop 8.5.3 in \cite{hir}).  

\begin{lem}\label{looplem} Given the equivalences $Pic _\om \simeq Ch ^+ (Ab) \simeq sAb$, the following pairs of endofunctors correspond to each other:
\ben $([-1], \Om) $ on $Ch^+(Ab)$ defined above,
\i $([-1], [1]^{0,0})$ on $Pic _\om$, where $[1]^{0,0}$ is defined after Definition \ref{0glue} and $[-1]$ is defined in Proposition \ref{shiftfunctor}, and
\i $(M,L )$, where $L$ is given by $L(X)_n = Ker(d_0 ^X : X _{n+1} \lra X_n )\cap Ker(d_1 ^n)$ and $d_n ^{L(X)} = -d_{n+1} ^X$, $s_n^{L(X)} = -s_{n+1} ^X$.  One can describe $L(X)_n$ as the $(n+1)$-simplices in $X$ such that the 0-th face and 0-th vertex of $x$ is $0$.  On the other hand, $M(X)_n = X_{n-1}$ for $n>0$ and $M_0 = 0$.  The structure maps are $d^{M(X)} _i = d^X_{i-1} $ for $i>0$ and $d_0 = 0$.   
\een
\end{lem}
\begin{proof} Using the equivalences $P: Ch^+(Ab) \lra Pic _\om$ and $K: sAb \lra Ch^+(sAb)$, the result follows easily.  
\end{proof}

Lemma \ref{looplem} suggests that we should think of $\Om (A)$ as loops in $A$ based at $0$.  We can also consider path functors. 

\begin{lem}\label{pathlem2} Given the equivalences, $Pic _\om \simeq Ch^+(Ab) \simeq sAb$, the following functors correspond to each other.  
\ben In $Pic_ \om$ the path functor $[1]: Pic _\om \lra Pic _\om$ is the restriction of $[1]: \om Cat \lra \om Cat$ defined immediately after Definition \ref{0glue}.
\i $\Pi : Ch^+ (Ab) \lra Ch^+ (Ab)$  defined by $\Pi (A) = ...\lra A_3 \lra A_2 \lra A_1 \oplus A_0$, with $A_1\oplus A_0$ in degree $0$ and differential $d\oplus 0 : A_2 \lra A_1 \oplus A_0$.  
\i $Path: sAb \lra sAb $ defined by $Path(X) = \hat S (X) \oplus X_0$, where $X_0$ is a discrete simplicial set with $X_0$ in degree $0$ and $\hat S : sAb \lra sAb$ is given by $\hat S(X)_n :=   Ker(d_0 ^X : X _{n+1} \lra X_n )$ and $d_n ^{\hat S(X)} = -d_{n+1} ^X$, $s_n^{\hat S(X)} = -s_{n+1} ^X$.
\een
\end{lem}

\begin{proof} To make the identification of $\Pi$ with $[1]$, let ${P} $ denote the equivalence ${P} : Ch^+(Ab) \lra Pic_\om$.  Identifying $((a_{n+1}^+,a_{n+1}^-),...,(a_2^+ , a_2 ^-), (a_1^+ + a_0^+ , a_1 ^- + a_0 ^-)) \in {P} (\Pi (A))_n$ with $((a_{n+1}^+,a_{n+1}^-),...,(a_2^+ , a_2 ^-), (a_1^+  , a_1 ^- ), (a_0^+ , a_0 ^+ - da_1 ^+)) \in ({P} A)_{n+1} = (PA)[1]_n$ establishes that $[1]\circ {P} = {P} \circ \Pi$. \\

To make the identification of $Path$ with $\Pi$, let $S: Ch^+ (Ab) \lra Ch^+ (Ab)$ denote the functor $A \mapsto (...\lra A_3 \lra A_2 \lra A_1)$.  It is easily verified using the Dold-Kan correspondence that $\hat S$ corresponds with $S$.  Observe that $\Pi (A) = S(A) \oplus (...0 \lra A_0)$ so that $K\Pi (A) = K(S(A)) \oplus K(A_0)$, where $K(A_0)$ is the discrete simplicial abelian group with $A_0$ in degree $0$.  
\end{proof}

For $A \in Pic_ \om$ and 0-objects $a,b \in A_0$, the sub-$\om$-category $A[1]^{a,b}$ of $A[1]$ is not a Picard $\om$-category.  However, we can still describe its nerve as a sub-object $  NA[1]^{a,b}$ of $ NA[1]$ in $sSet$.  Denote the path functor $sAb \lra sSet $ corresponding to $[1]^{a,b} : Pic _\om \lra \om Cat$ by $Path^{a,b}$.  First we argue that the n-simplices of $NA[1]$ can be viewed as the $(n+1)$-simplices of $NA$ for which the 0th face is $s_0 ^{n+1}b$ for some vertex $b\in NA_0$.  By Lemma \ref{pathlem2}, $NA[1]_n = \{y \in (NA)_{n+1} \st d_0 y = 0\} \oplus A_0$, and there is an inclusion $NA[1]_n \hookrightarrow NA_{n+1}$ given by $(y,b)\mapsto y + s_0 ^{n+1} b$.

\begin{lem} Let $A$ be a Picard $\om$-category.  Then $NA[1]^{a,b} _n$ consists of simplices $x \in NA_{n+1}$ for which the 0-th vertex $v_0 x  = (d_1 ) ^{n+1} x  = a$ and $d_0  x= b$.  
\end{lem} 

\begin{proof} We prove by induction that for an n-simplex in $x \in NA[1]_n \subset NA_{n+1}$ such that $v_0 x = a$ and $d_0 x = b$, $x \in NA[1]^{a,b}$.  For $n = 0$, let $x$ be a 0-simplex in $NA[1]$ such that $v_0 x =a $ and $d_0 x = b$.  Since $(NA)_1 = A_1$, we can think of the 1-simplex $x$ in $NA$ as a 1-morphism in $A$, where the source of $ x$ is $s_0  x = d_1 x = v_0 x = a$ and the target of $ x$ is $t_0  x = d_0 x = b$.  Hence, $x \in N[1]^{a,b}_0$. \\

Now let $x$ be an n-simplex in $NA[1]_n \subset NA_{n+1}$ such that $v_0 x = a$ and $d_0 x = b$. Observe that $v_0 x = v_0 (d_i x)=a$ and $d_0x = d_0 d_i x = b$ for every $0 < i \leq n+1$, and in fact, $v_0 w = a$ and $d_0 w = b$ for every r-dimensional face $w$ of $x$.   We make use of the notation defined in \S \ref{paritysection} for the remainder of the proof.  By construction, an n-simplex in $NA[1]$ is a morphism $x \in Hom_{\om Cat} (\OO (\tilde \De ^n) , A[1])$, where $\OO(\tilde \De ^n)$ is Street's n-th oriental, defined in \S \ref{paritysection}.  For an $\om$-category $B = (Ob(B) , *_i , s_i , t_i){i \in \N}$, and $m\geq 0$, let $|B|_m = (B_m , *_i , s_i , t_i)_{0 \leq i \leq m-1}$ denote the $m$-category formed by taking all $k$-morphisms for $k\leq m$.  In \cite{str}, Street shows that a choice of morphism   $x \in Hom_{\om Cat} (\OO (\tilde \De ^n) , A[1])$ is equivalent to a morphism  $g \in Hom_{\om Cat} (|\OO (\tilde \De ^n)|_{n-1} , A[1])$ and $\al \in A[1]_n$ such that $s[1]_{n-1} \al = g(s_{n-1} \langle (01...n) \rangle)$ and $t[1]_{n-1} \al = g(t_{n-1} \langle (01...n) \rangle)$, where $\langle (01...n) \rangle$ is the unique non-trivial n-morphism in $\OO (\tilde \De ^n)$.  Hence, the 0-source and 0-target of $\al$ in $A$ are determined by $g$ because $s_0\al = s_0 s[1]_{n-1}\al = s_0 g(s_{n-1} \langle (0...n) \rangle ) = g(s_0s_{n-1}\langle (0...n)\rangle ) = g(s_0 \langle (0...n)\rangle) $, and similarly, $t_0\al = g(t_0 \langle (0...n)\rangle)$.  Street shows that $g(s_{n-1} \langle (0...n \rangle )$ and $ g(t_{n-1}\langle (0...n\rangle )$ are compositions of of $g$ applied to $\langle \be \rangle$ for some $\be : \De ^r \hookrightarrow \De ^n$ with $r <n$.  Using properties (1d) and (2b) of Definition \ref{16nov1}, taking $s_0$ or $t_0$ of a composition simply applies $s_0$ or $t_0$ to the last morphism in the chain of compositions.  Thus, $s_0 \al = s_0 g(\langle \be \rangle) $ and $t_0 \al =  t_0 g(\langle \be \rangle)$ for some $\be : \De ^r \hookrightarrow \De ^n$ with $r <n$.  But $\OO (\tilde \De ^r) \stackrel{\beta _*}\lra \OO (\tilde \De ^n) \stackrel{x}\lra A[1]$ is simply one of the r-faces of the n-simplex $x$.  Since $x \be _*$ is an r-face of $x$,  $a = v_0 x = v_0 (x\be _*)$ and $b = d_0 (x \be _*)$. By induction hypothesis, $s_0 (g \langle \be \rangle) = a $ and $t_0 (g \langle \be \rangle ) = b$, so  $s_0\al = a $ and $t_0 \al =  b$.  Also, for $r <n$, every r-dimensional face $x \gamma _*$ for $\gamma :\De ^r \hookrightarrow \De ^n$, $v_0(x \gamma _*) = v_0 x = a $ and $d_0 (x \gamma _* ) = d_0 x = b$, whence $s_0 (g\langle \gamma \rangle ) = a$ and $t_0 (g \langle \gamma \rangle ) = b$ by induction hypothesis. Thus, for $r \leq n$, every r-face $w$ of $x$, the r-morphism $g(\langle \gamma _w \rangle ) $in $A[1]$ representing $w$ has $a$ as its 0-source and $b$ as its 0-target.  Therefore, $x \in NA[1]^{a,b}$.      
\end{proof}

\begin{rem} For $X \in sAb$, $Path^{a,b} X$ (the paths in $Path (X)$ from $a$ to $b$) is isomorphic to the left mapping space $Hom_{X} ^L(a,b) $ of Lurie \cite{lur}.  
\end{rem}

\subsection{Equivalence of Descent Conditions}



\begin{lem}\label{0glueequivlem} Let $\ca$ be a presheaf of $\om$-categories.  The unique 0-glueing condition is equivalent to the condition that $\pi _0 (N\ca (X)) \simeq \pi _0 (holim \cna) $.
\end{lem}
\begin{proof} In \cite{bok} ch.10-11, it was proved that for $G \in sAb ^\De$, $hom(\YY , G) \lra holim G$ is a weak equivalence.  Therefore, $\pi _0 (\h ) \simeq \pi _0 holim \cna$.  We conclude the proof by arguing that $\pi _0 (N \ca (X)) \simeq \pi _0 (\h)$ if and only if $\ca$ satisfies the unique 0-glueing condition.\\

The 0-glueing condition states that for all $f \in Hom_{sSet ^\De } (\YY , \cna)$, there exists $H \in Hom_{sSet ^\De} (\YY \times \De ^1, \cna)$ such that $H_{|\YY \times \{0\}} = f$ and $H _{|\YY \times \{1\}} = \rho g$ for some $g\in \ca (X) _0$.  Since $  Hom_{sSet ^\De} (\YY \times \De ^1, \cna) = \h _1$, this is equivalent to asking that for every vertex $f \in \h _0$, there exists $H \in \h_1$ such that $d _1 H = f$ and $d _0 H = \rho g$ for some $g \in N\ca (X) _0 $.  In other words, The 0-glueing condition states that $\rho _* : \pi _0 (N\ca (X)) \lra \pi _0 (\h)$ is a surjection.  The uniqueness part of the unique 0-glueing condition states that $\rho _*$ is also injective.
\end{proof}

\begin{lem}\label{loopmovelem} Let $\ca$ be a presheaf of Picard $\om$-categories.  Then $\ca$ satisfies the unique glueing condition for all loops if and only if it satisfies the unique glueing condition for loops at $0$.  Furthermore, for any $a \in \ca (X)_0$, $\ca [1]^{0,0} \simeq \ca [1]^{a,a}$ in $Pre_{\om }$.  
\end{lem}
\begin{proof} First we show that for $A \in Pic_\om$ and  $a \in A_0$, addition by $a$, $p_a  : \ca \lra \ca $ (mapping $x \mapsto x+a$) defines an isomorphism of $\om$-categories, though not of Picard $\om$-categories.  To see that $p_a$ is a morphism of $\om$-categories, let $\sigma = s$ or $t$ and $n\geq 0 $.  Then $ \sigma _n p_a (x) = \sigma _n (a + x ) = \sigma _n (a) + \sigma _n (x)  = a + \sigma _n (x) = p_a (\sigma _n x)$ since $a$ is a 0-object. Composition is also preserved; $p_a (x*_n y) = a + x*_n y = a+ x + y - s_n x = a + x + y +a -(s_n x +  a) = a + x + y +a  -s_n(x+a) = (x+a)*_n(y+a) = p_ax *_n p_a y$.  Therefore, $p_a$ is a morphism of $\om$-categories.  Since $p_a $ has inverse $p_{-a}$, it is an isomorphism.  If $\ca$ is a presheaf of $\om$-categories and $a \in \ca (X) _0$, then $p_a : \ca  \lra \ca $ is an isomorphism of presheaves of $\om$-categories sending basepoint $0$ to basepoint $a$.   
\end{proof}

\begin{lem}\label{pinshift} Let $G \in sAb$ and $L : sAb \lra sAb$ be the functor described above, corresponding to $[1]^{0,0} : Pic_\om \lra Pic _\om$.  Then $\pi _{n+1} L^k(G) \simeq \pi _n L^{k+1}G$ for all $k,n \geq 0$.  
\end{lem} 
\begin{proof} It follows from the definition of $L$ that  $ Hom_{sSet}(\De ^n , LG) \simeq \{ f \in Hom_{sSet} (\De ^{n+1} , G) \st (0)\mapsto 0 \; and \; fd_0(012...n) = 0\}$.  Therefore, $ \{ f \in Hom_{sSet}(\De ^n , LG) \st \partial \De ^n \lra 0 \} \simeq \{ f \in Hom_{sSet }(\De ^{n+1} , G) \st \partial \De ^{n+1} \lra 0\} $. Not only are they isomorophic as sets, but homotopies in  $ \{ f \in Hom_{sSet}(\De ^n , LG) \st \partial \De ^n \lra 0 \}$ coincide with those in $\{ f \in Hom_{sSet }(\De ^{n+1} , G) \st \partial \De ^{n+1} \lra 0\} $.  To see this, we will show that $f$ and $g$ are homotopic in  $ \{ f \in Hom_{sSet}(\De ^n , LG) \st \partial \De ^n \lra 0 \}$ if and only if the corresponding simplices in $\{ f \in Hom_{sSet }(\De ^{n+1} , G) \st \partial \De ^{n+1} \lra 0\} $ are homotopic.  However, since $G$ and $LG$ are simplicial abelian groups, it suffices to show that $f$ being homotopic to $0$ is the same in both sets.  It is a standard fact, found in \cite{hov2}, that for $f \in G_n$ with $d_i f = 0 $ for all $i$, $f$ is homotopic to $0$ if and only if and only if there is some $h \in G_{n+1}$ such that $d_{n+1} h = f $ and $d_i h =0$ for all $i \leq n$.  Clearly then, $f : \De ^n \lra LG$ is homotopic to $0$ if and only if the corresponding map $\De ^{n+1} \lra G$ is homotopic to $0$.   
\end{proof}

\begin{prop}\label{loopdescentthm} Let $\ca$ be a presheaf of Picard $\om$-categories.  Then $\ca$ is a satisfies \v Cech descent if and only if $\ca$ satisfies the unique glueing condition for loops.  
\end{prop} 

\begin{proof}  Let $\GG = N\ca$ be the corresponding presheaf of simplicial abelian groups. Then $\GG$ satisfies descent for all hypercovers if and only if $\GG(X) \lra holim \check \GG$ is a weak equivalence, i.e. $\pi _n \GG(X) \lra \pi _n holim \check \GG$ is an isomorphism for each $n \geq 0$.  By lemma \ref{pinshift}, $\pi _n \GG(X) \lra \pi _n holim \check \GG$ is an isomorphism for each $n \geq 0$ if and only if $\pi _0 (L^n\GG(X)) \simeq \pi _0 (L^n(holim \GG))$ for all $n\geq 0$.   Since $L$ is right Quillen, it preserves homotopy limits.  Therefore, $ \pi _0 (L^n(holim \GG)) \simeq \pi _0 (holim (L^n \GG))$.  In summary, $\GG$ satisfies \v Cech descent if and only if $\pi _0 (L^n \GG (X)) \simeq \pi _0 (holim (L^n  \check \GG))$ for all $n\geq 0 $.  By Lemma \ref{0glueequivlem}, we conclude that $\GG$ satisfies \v Cech descent if and only if $L^n \GG$ satisfies the unique 0-glueing condition for each n.  To say that $L^n \GG$ satisfies the unique 0-glueing condition is just to say that $\GG$ satisfies the unique n-glueing condition.  Therefore, $\GG$ satisfies \v Cech descent if and only if $\GG$ satisfies the unique glueing condition for loops based at $0$.  However, Lemma \ref{loopmovelem} implies that this is equivalent to the unique glueing condition for all loops.
\end{proof}

\begin{cor} Let be a $\ca$ is a presheaf of Picard $\om$-categories.  If $\ca$ satisfies $\om$-descent, then it satisfies \v Cech descent.
\end{cor}

To complete the proof of Theorem \ref{equivdescent}, we now show that if a presheaf of simplicial abelian groups satisfies \v Cech descent for all hypercovers, it satisfies the unique glueing condition, not just for loops.  

\begin{lem}\label{shiftab} Let $\ca$ be a presheaf of Picard $\om$-categories, $X \in \mathcal S$, and $a,b \in \ca (X)_0$.  If there exists $f \in \ca (X)_1$ such that $s_0f = a$, $t_0 f = b$, then $\ca [1] ^{a,b} \simeq \ca [1]^{0,0} $ in $Pre _\om$, whence $\ca [1]^{a,b}$ is a presheaf of Picard $\om$-categories.  
\end{lem}
\begin{proof} There is an isomorphism $\ca [1]^{a,b} \lra \ca [1]^{a,a}$ sending a section $y$ to $x\inv *_0 y = y+ (x\inv -b)$.  First, let us see that for $y \in \ca [1]^{a,b} $, $x\inv *_0 y \in \ca [1]^{a,a}$.  We easily see that $s_0 (y+ x\inv -b) = s_0y + s_0 x\inv -s_0 b= a +b-b = a$ and $t_o (y + x\inv -b) = t_0 y + t_0 x\inv -t_0 b = b +a-b = a$ so that $x\inv *_0 y \in \ca [1]^{a,a}$.  Since $x\inv -b \in \ca [1](X)_0$, addition by $x \inv -b$ is an isomorphism (The proof is identical to that of Lemma \ref{loopmovelem}).  Recall from Lemma \ref{loopmovelem} that $\ca [1]^{a,a} \simeq \ca [1]^{0,0}$.  
\end{proof}

\begin{lem} Let $\ca$ be a sheaf of $\om$-categories and $a,b \in  \ca (X)_0$.  Then $\ca[1]^{a,b} (X) \lra holim \check \ca [1] ^{a,b}$ is a weak equivalence if and only if $\ca [1]^{a,b} (X) \lra hom_{sSet ^\De }(\YY , \check \ca [1]^{a,b})$ is a weak equivalence. 
\end{lem}
\begin{proof} Let $\BB = N\ca [1]^{a,b} \in Pre _{sAb}$.  If some $\BB (U_{i_0...i_n}) = \emptyset$, then $\check \BB = \emptyset$, and $\BB (X) = \emptyset$, so both  $\ca[1]^{a,b} (X) \lra holim \check \ca [1] ^{a,b}$ and $\ca [1]^{a,b} (X) \lra hom_{sSet ^\De }(\YY , \ca [1]^{a,b})$ are weak equivalences.  Now suppose that each  $\BB (U_{i_0...i_n}) \neq \emptyset$.  By the two out of three axiom for weak equivalences, it suffices to show that 
$hom_{sSet ^\De }(\YY , \ca [1]^{a,b}) \lra holim \check \BB$ is a weak equivalence.  By Ch. 11, \S 4 in \cite{bok}, the result will follow if we show that $\check \BB$ is a fibrant object in $sSet ^\De$ with Bousfield and Kan's model structure.  \\

For $X \in sSet ^\De$, let the n-th matching space $M^n X = \{ (x_0,...x_n) \in X^n \times X ^n \times ...\times X^n \st s^i x_j = s^{j-1}x_i \; if \; 0 \leq i < j \leq n \}$, where $s^i$ denotes the i-th coface map $X(\sigma _i)$.  There is a map $X^{n+1} \lra M^n X$ in $sSet$ given by $x \mapsto (s^0x ,...,s^nx)$.  By definition (Ch. 10 \S 4 in \cite{bok}), $X$ is fibrant if and only if each $X^{n+1} \lra M^n X$, $n\geq 0$ and $X^0 \lra *$ are fibrations in $sSet$.  We now proceed to show that $\check \BB$ is fibrant.  \\

First we show that we can endow each each $\check \BB ^n$ with the structure of a simplicial abelian group such that all $s^i : \check \BB ^{n+1} \lra \check \BB ^n$ is a morphism in $sAb$.   For the open cover $\UU = \{ U_i \}_{i \in I}$, we are assuming that for each $n \geq 0$ and each $\al \in I^{[n]}$, $\BB (U_\al) \neq \emptyset$, so by Lemma \ref{shiftab}, we can choose a group structure on $\check \BB ^n$ by choosing a 1-simplex $f = \{f_\al \}_{\al \in I ^{[n]}}$ with $d_1 f = a$, $d_0f = b$.  The goal is to choose a group structure on each $\check \BB$ so that the coface maps are all morphisms of simplicial abelian groups.  First declare an equivalence relationship on $I^{[n]}$ by setting $\al \sim \be$ if $U_\al = U_\be$, and let $\overline I^{[n]}$ denote the set of equivalence classes.  Observe that the coface maps $s^m : \check \BB ^{n+1} \lra \check \BB ^n$ are given by $(s^m (\{ x_\al\} _{\al \in I^{[n+1]}}))_\be = x_{\sigma _m ^* \be}$, where $\sigma _m : [n+1] \lra [n]$ is the monotonic map which repeats only $m$. \\

To choose the group stuctures on $\check \BB ^n$, we start with $n = 0$.  Choosing any group structure  $f = \{f_\al\}_{\al \in I ^{[0]}}$ such that $f_\al = f_\be $ for $\al \sim \be$.  Now, having chosen group structures for all $\check \BB ^k$ for $k\leq n$ such that all coface maps are morphisms of simplicial abelian groups, choose any group structure   $f = \{f_\al\}_{\al \in I ^{[n+1]}}$ such that $f_\al = f_\be$ if $\al \sim \be$ and if $\al = \sigma _m ^* \be$ for some $\be \in I ^{[n]}$, $f_\al = f_{\be }$.   To see that such a choice exists, we simply observe that if $\al \sim \ga $ such that $\al = \sigma _m ^* \be  $ and $\ga =  \sigma ^* _l \de$, then $U_\be = U_{\sigma _m ^* \be} = U_\al = U_\ga = U_{\sigma _l ^* \de} = U_\de$, so $\ga \sim \de$ and $f_\ga = f_\de$.  The coface maps $(s^m (\{ x_\al\} _{\al \in I^{[n+1]}})_\be = x_{\sigma _m ^* \be}$ are morphisms of simplicial abelian groups because for $\be \in I^{[n]}$, $\pi _\be s^m$ is the composition of the identity map $\BB (U_{\sigma _m ^* \be}) \lra \BB (U_\be)$ with the projection $\pi _{\sigma _m ^* \be} : \check \BB ^{n+1} \lra \BB (U_{\sigma _m ^* \be})$, both of which are group maps since the group stuctures $f_{\sigma ^* _m \be}$ and $f_\be$ were chosen to coincide.  \\

We now demonstrate that each $\check \BB ^{n+1} \lra M^n \check \BB$ is surjective.  It will follow that these maps are levelwise epimorphisms of abelian groups, hence fibrations in $sSet$.  Choose any $n \geq 0$.  The proof that $\check \BB ^{n+1} \lra M^n \check \BB$ is surjective is very much the same as the proof in the previous paragraph.  For $x_0$...$x_0 \in \check \BB ^n$ ($x_i = \{ x_i ^\al \} _{\al \in I^{[n]}}$), $(x_0 ,...,x_m)\in M^n \check \BB$ if and only if for all $0 \leq l <  m \leq n$ and $\al \in I ^{[n-1]}$, $x_m ^{\sigma _l ^* \al} = x_l ^{\sigma _{m-1}^* \al}$.  Choose any $y = \{ y ^\al \} _{\al \in I^{[n+1]}}$ such that for all $0 \leq m \leq n$, $\be \in I ^{[n]}$, $y^{\sigma ^* \be} = x_m ^\be$.  Hence the map $\check \BB ^{n+1} \lra M^n \check \BB$ sends $y$ to $(x_0,...x_n)$.  However, we need to check that such a choice exists.  We need to show that if $\sigma _m ^* \be = \sigma _l ^* \ga$, then $x_m ^\be = x_l ^\ga$.  Suppose that $\sigma _m ^* \be = \sigma _l ^* \ga$.  Clearly $\al = \sigma _m ^* \sigma _l ^* \de$ for some $\de$, so $\al = (\sigma _l \sigma _m )^* \de = (\sigma _{m-1} \sigma _l) ^* \de = \sigma _l ^* \sigma _{m-1}^* \gamma$.  In general, if $\sigma ^* _m \tau = \sigma _m ^* \mu $, then $\tau = \mu$.  Hence, $\be  = \sigma ^* _l \de$ and $\ga = \sigma _{m-1}^* \ga $, so $x_m ^\be = x_m ^{\sigma _l ^* \de } = x_l ^{\sigma _{m-1} ^* \de} = x_l ^\ga$.  Therefore, $y^\al$ is well-defined.   
\end{proof}

\begin{lem}\label{aatoab} Let $\ca $ be a presheaf of Picard $\om$-categories that satisfies \v Cech descent, and let $a,b \in \ca (X)_0$.  If $Hom_{sSet ^\De}(\YY , \check N \ca [1]^{a,b})$ is non-empty, then there exists a path $x \in \ca (X)_1$ from $a$ to $b$, and $\ca [ 1] ^{a,b} \simeq \ca [1]^{a,a} \simeq \ca [1]^{0,0}$ as sheaves of $\om$-categories.  It follows that $\ca [1]^{a,b} \in Pre_{\om Ab}$.  
\end{lem}
\begin{proof} Choose any $F\in Hom_{sSet ^\De}(\YY , \cna[1]^{a,b})$.  Let $\YY ^{(1)}$ denote the object in $sSet ^\De $ for which $\YY ^{(1)} ([n]) = \De ^{n+1}$, $\YY ^{(1)} ([n] \stackrel{\partial _i}\lra [n+1]) = \De ^{n+1} \stackrel{\partial_{i+1}}\lra \De ^{n+2}$, and  $\YY ^{(1)} ([n] \stackrel{\sigma _i}\lra [n-1]) = \De ^{n+1} \stackrel{\sigma_{i+1}}\lra \De ^{n}$.  First we observe that $Hom_{sSet ^\De}(\YY , \cna[1]^{a,b})$ is isomorphic to $Hom_{sSet ^\De} (\YY ^{(1)}, \cna )^{a,b} := \{ f \in Hom_{sSet ^\De} (\YY ^{(1)}, \cna ) \st f_n :\De ^{n+1} \lra \cna _n \; \textrm{satisfies} \; f((0)) = a,\; d_0f(01...n+1) = b\}$. \\

We will define a morphism $p: \De ^n \times \De ^1$ in $sSet$ which sends $\De ^n \times \{ 0\} $ to $a$ and $\De ^n \times \{1\} $ to $b$.  First observe that $\De ^n \times \De ^1$ has non-degenerate $(n+1)$-simplices $z_k : = (012..,k-1,k,k,k+1,...n, s_0^k (01) \in \De ^n _{n+1}\times \De ^1 _{n+1}$, $0\leq k \leq n$, which satisfy relations $d_kz_k = d_k z_{k-1}$, for $k\geq 1$.  For any $X \in sSet$, to give a morphism $f  \in Hom_{sSet}(\De ^n \times \De ^1, X)$, it is necessary and sufficient to give $(n+1)$-simplices $y_k = f(z_k) \in X_{n+1}$ such that $d_k y_k = d_k y_{k-1}$.  Let $p: \De ^n \times \De ^1 \lra \De ^{n+1}$ be the morphism given by $y_k = s_0 ^k d_1 ^k (012...n+1)$.  One may verify that $\De ^n \times \{0\} \lra s_0 ^n (0)$ and $\De ^n \times \{1\} \lra d_0 (01...n+1)$.  An easy but tedious calculation shows that $p$ extends to a morphism $p: \YY \times \De ^1 \lra \YY ^{(1)}$ in $sSet ^\De$.\\

Now, $F\in Hom_{sSet ^\De}(\YY , \cna[1]^{a,b})$ corresponds to some $f \in Hom_{sSet ^\De} (\YY ^{(1)}, \cna )^{a,b}$.  We can form the composition $fp\in Hom_{sSet ^\De }(\YY \times \De ^1 , \cna [1]) = hom(\YY , \cna)_1$, which sends $\YY \times \{0\}$ to $a$ and $\YY \times \{1\}$ to $b$.  Thus, $fp \in hom(\YY , \cna)$ is a 1-simplex in $\h$ from $a$ to $b$.  Since $\ca$ satisfies \v Cech descent, $\rho  :N\ca (X) \lra hom (\YY , \cna )$ is a weak equivalence of fibrant simplicial sets.  Therefore, there exists a morphism $G: hom (\YY , \cna ) \lra N\ca (X)$ such that $G \rho  $ is homotopic to the identity.  Since $N\ca (X)$ is fibrant, one can also find a 1-simplex in $N \ca (X)$ from $a$ to $b$, i.e. a path in $\ca (X)_1$ from $a$ to $b$. The result now follows from Lemma \ref{shiftab}.   
\end{proof}
\begin{rem} For $a,b \in \ca (X) _0$ as in Lemma \ref{aatoab}, the (presheaf of) abelian group structure of $\ca [1]^{a,b}$ is not the one endowed from being a sub-$\om$-category.
\end{rem}

\begin{lem}\label{biglemma} Let $\ca$ be a presheaf of Picard $\om$-categories which satisfies \v Cech descent.  Suppose $a,b \in \ca (X)_k$ are such that $s_{k-1} a = s_{k-1}b$ and  $t_{k-1}a = t_{k-1}b$.  Then $\ca [ k+1]^{a,b} $ satisfies \v Cech descent.  Additionally, if there exists a $(k+1)$-morphism $x$ from $a$ to $b$, then $\ca [k+1]^{a,b}$ is a presheaf of Picard $\om$-categories.  
\end{lem}

\begin{proof} We prove the statement by induction on $k$.  First let $k = 0$.  If $Hom_{sSet ^\De}(\YY , \check N \ca [1]^{a,b})$ is non-empty, then there exists a path $x \in \ca (X)_1$ from $a$ to $b$ and $\ca [ 1] ^{a,b} \simeq \ca [1]^{a,a} \simeq \ca [1]^{0,0} \in Pre_{\om Ab}$.  Since $\ca$ satisfies \v Cech descent, it satisfies the unique k-glueing condition for loops at $0$, and since $(\ca [1] ^{0,0})[k]^{0,0} = \ca [k+1]^{0,0}$, $\ca [ 1] ^{0,0}$ satisfies the unique k-glueing condition for loops at $0$.  Therefore, $\ca [1]^{0,0}$ is a presheaf of Picard $\om$-categories which satisfies \v Cech descent.  It follows that since $\ca [1]^{0,0} \simeq \ca [1]^{a,b}$, $\ca [1]^{a,b} $ is a presheaf of Picard $\om$-categories satisfying \v Cech descent.  \\

If, on the other hand, $Hom_{sSet ^\De}(\YY , \check N \ca [1]^{a,b}) = \emptyset$, then the simplicial set $hom_{sSet ^\De}(\YY , \check N \ca [1]^{a,b}) = \emptyset$.  But  $Hom_{sSet ^\De}(\YY , \check N \ca [1]^{a,b}) = \emptyset$ also implies that $N\ca [1]^{a,b}(X)_0 = \emptyset$, whence  $N\ca [1]^{a,b}(X) = \emptyset$, so $\ca [1]^{a,b}$ trivially satisfies \v Cech descent.  This proves the base case ($k=0$).  \\

Now suppose that  $a,b \in \ca (X)_k$ are such that $s_{k-1} a = s_{k-1}b$ and  $t_{k-1}a = t_{k-1}b$.  Then for any open $U \subset X$, as a set,
\begin{eqnarray*} \ca[k+1] ^{a,b} (U) & = & \{x \in Ob(\ca (U)) \st s_k x = a , \; t_k x  = b \}\\
                                  & = & \{ x \in Ob(\ca (U)) \st s_k x = a , \; t_k x  = b, \; s_{k-1} x = s_{k-1} a, \; t_{k-1} x = t_{k-1} b \}\\
                                  & = & \{ x \in Ob(\ca [k] ^{s_{k-1}a , t_{k-1}b} (U)) \st s_k x = a , \; t_k x  = b \} \\
                                  & = & \{ x \in Ob(\ca [k] ^{s_{k-1}a , t_{k-1}b} (U)) \st s[k]_0 x = a , \; t[k]_0 x  = b \} \\
                                  & = & (\ca [k] ^{s_{k-1}a , t_{k-1}b})[1]^{a,b} (U).
\end{eqnarray*}
Hence,$ \ca[k+1] ^{a,b} = (\ca [k] ^{s_{k-1}a , t_{k-1}b})[1]^{a,b} $.  But since $a$ is a $k$-morphism from $s_{k-1}a $ to $ t_{k-1}b$, $\ca [k] ^{s_{k-1}a , t_{k-1}b}$ is a prehseaf of Picard $\om$-categories satisfying \v Cech descent, by induction hypothesis.  By the base case, $ \ca[k+1] ^{a,b} = (\ca [k] ^{s_{k-1}a , t_{k-1}b})[1]^{a,b}$ satisfies \v Cech descent, and if there exists a $(k+1)$-morphism $x \in \ca (X)_{k+1}$ from $a$ to $b$, then $ \ca[k+1] ^{a,b} \in Pre_{\om Ab}$.
\end{proof}

Theorem \ref{equivdescent} now follows directly from Lemma \ref{biglemma}.  
\begin{proof}  Suppose that $\ca \in Pre _{\om Ab}$ and satisfies \v Cech descent.  Then by Lemma \ref{0glueequivlem}, $\ca$ satisfies the unique 0-glueing condition.  Lemma \ref{biglemma} ensures that each $\ca [k]^{a,b}$ satisfies \v Cech descent hence the unique 0-glueing property by Lemma \ref{0glueequivlem}.  Therefore $\ca$ satisfies the unique $k$-glueing property for each $k$, i.e. satisfies $\om$-descent. 
\end{proof}

\section{$\infty$-torsors}

We have added this section for completeness, as using $\om$-descent as a way to make $\infty$-torsors accessible was a primary motivation for establishing the equivalence of the two descent conditions.  The central ideas in this section (Definition \ref{zuho} and Propositions \ref{jjlis} and \ref{iuts}) are due to Fiorenza, Sati, Schreiber, and Stasheff \cite{fss,sss,sch2}.  We simply formulate them in a slightly different way, prefering to define objects up to homotopy. \\

\begin{Def}\label{zuho} Let $G$ be a presheaf of simplicial abelian groups on a site $\CC$.  Let $B G$ denote any delooping object of $G$ in the homotopy category of $Pre_{sSet} ^{proj, loc} (\mathcal C)$.  This means that $B G$ is an object with a point $* \lra B G$, and $G$ is the homotopy pullback of the diagram 

\[
\begin{CD} 
   @. *\\
@.  @VVV\\
* @>>> B G
\end{CD}
\]
We define $Tors _G = hom(-,\widetilde {B G})$, where $hom$ denotes simplicial enrichment in $Pre_{sSet} (\CC)$, and $\widetilde {B G}$ denotes a sheafification (i.e. fibrant replacement) of $B G$ \cite{fss}. 
\end{Def}

\begin{rem} Note that $Tors_G$ is well-definied up to weak equivalence due to the uniqueness up to homotopy of looping and delooping functors \cite{goj,hov}.  Furthermore, Lemma \ref{welem} shows that for two different choices $T_1$, $T_2$ for $Tors_G$, $T_1(X)$ is weakly equivalent to $T_2(X)$ for any $X \in \CC$.  
\end{rem}

\begin{lem}\label{welem} Let $F,G \in Pre_{sSet}^{proj,loc}$ be fibrant objects which are weakly equivalent. Then For any $X \in \CC$, $F(X)$ and $G(X)$ are weakly equivalent.
\end{lem}

\begin{proof} Let $V \lra X$ be a hypercover of $X$, and let $V^\prime$ be a cofibrant replacement of $V$.  Then By \cite{dhi} Lemma 4.4, $F(X)\simeq hom(X, F) \lra hom(V^\prime , F)$ and $G(X) \lra hom(V^\prime , G) $ are weak equivalences.  If there is a weak equivalence $F \lra G$, then it is a general fact \cite{hir} Corollary 9.3.3 that in a simplicial model category with cofibrant $V^\prime$ and a weak equivalence of fibrant objects $F \lra G$ then $hom(V^\prime , F) \lra hom(V^\prime , G) $ is a weak equivalence of simplicial sets.  Therefore, $F(X)$ is weakly equivalent to $G(X)$.  If there is a zigzag of weak equivalences from $F$ to $G$, then the result is the same: $F(X)$ and $G(X)$ are weakly equivalent. 
\end{proof}

The central observations of this section are Propositions \ref{jjlis} and \ref{iuts}.  The proof Proposition \ref{jjlis} is a modification of the proof of  Proposition 3.2.17 in \cite{fss}, which is for complexes concentrated in one degree only.  First we make use of the following fact about homotopy limits for presheaves. \\

A homotopy pullback is simply the homotopy limit in the model category $Pre_{sSet}^{proj, loc}$.  However, since every sectionwise weak equivalence is a local weak equivalence, the identity map $i: Pre_{sSet}^{proj} \lra Pre_{sSet}^{proj, loc}$ preserves weak equivalence and is adjoint to itself.  It is a general fact that if a functor $U $ between model categories is a right adjoint and preserves weak equivalences, then $U$ preserves homotopy limits.  Hence, to compute the homotopy limit in the local model structure, it is enough to compute it in the global projective model structure. \\

Given a complex $A \in Ch^+ (\ca b)$ of presheaves of abelian groups concentrated in non-negative degrees, recall that $A[1]$ is $A$ shifted up one degree so that $A[1]^n = A^{n-1}$. For a presheaf of simplicial groups $G$, let $G[1]$ be the presheaf of simplicial groups corresponding to the shift functor in $Ch^+ (\ca b)$ by the Dold-Kan correspondence.  

\begin{prop}\label{jjlis} Let $G$ be a presheaf of simplicial groups.  Then $G[1]$ is a delooping object of $G$.
\end{prop}
\begin{proof}
We use the Dold-Kan correspondence between $Pre_{sAb}(\CC)$ and $Ch^+(\ca b)$.  Given a complex $A \in Ch^+ (\ca b)$ of presheaves of abelian groups concentrated in non-negative degrees, recall that $A[1]$ is $A$ shifted up one degree so that $A[1]^n = A^{n-1}$. Define $B = B(A)$ as follows.  Let $B^n = A^n \times A^{n-1}$, and let $d: B^n \lra B^{n-1}$ be $d (x,y) = (da+ (-1)^nb, db)$.  Clearly, $d^2 = 0$ so that $B \in Ch^+(\ca b)$.  We define $f: B \lra A[1]$ as the obvious map: $f^n : A^n \lra A^{n-1} \lra A^{n-1}$ is just $\pi _2$.  It is obvious that this is a chain map.  Furthermore, $B$ is acyclic in the sense that each homology class $H_nB = 0$.  Using the preceding facts about homotopy pullbacks, since $B \lra A[1]$ is a fibration, the pullback of the diagram 
\[
\begin{CD} 
   @. B\\
@.  @VVV\\
0 @>>> A[1]
\end{CD}
\]
is in fact the homotopy pullback.  The pullback $P$ has the property that any $g: C \lra B$ such that $fg = 0$ factors through $P $.  First we see that there is a map $h: A \lra B$ given by in degree $n$ by $x \mapsto (x,0)$.  and that $fh = 0$.  It is easy to see that a map $g$ such that $fg = 0$ is precisely a map $C \lra A \lra B$.  Hence, $A$ is the pullback.  Since there is a weak equivalence from the diagram 
\[
\begin{CD} 
   @. B\\
@.  @VVV\\
0 @>>> A[1]
\end{CD}
\]
to 
\[
\begin{CD} 
   @. 0\\
@.  @VVV\\
0 @>>> A[1],
\end{CD}
\]
$A$ is the homotopy pullback of the latter diagram, whence $A[1]$ is a delooping of $A$.  

\end{proof}

Since $G[1]$ is a delooping of $G$, it follows that $Tors _G(X) = \hom (X , \widetilde {G[1]}) \simeq \widetilde {G[1]} (X)$, whence $Tors_G = \widetilde {G[1]}$.  Having defined $Tors_G$, we define $Tors^n _G$, which is well defined up to local weak equivalence.  Let $Tors ^n _G := \widetilde{ G[n]}$, where the fibrant replacement $\widetilde {G[n]}$ is chosen to be a sheaf of simplicial abelian groups.  It was shown in Lemma \ref{brisu} that it is possible to make such a choice.  

\begin{prop}\label{iuts} Let $G $ be a presheaf of simplicial abelian groups.  Then $Tors ^n _G =  \widetilde {G[n]} = Tors_{Tors^{n-1}_G}$.
\end{prop}
\begin{proof}
This follows by induction.  By definition, $Tors _{Tors^{n-1}_G} = Tors _{\widetilde {G[n-1]}} = ((G[n-1])^\sim[1])^\sim$.  Since the Dold-Kan correspondence commutes with taking stalks, which is to say that the diagram 
\[
\begin{CD} 
 Ch^+ (p \ca b)  @>>> Pre_{sAb}\\
@VVV  @VVV\\
Ch^+(Ab) @>>> sAb
\end{CD}
\]
commutes if the vertical arrows represent taking stalks at a point $x$.  Therefore, for a local weak equivalence $A \lra B$ in $Pre_{sAb}$, the induced map $A[1] \lra B[1]$ is a local weak equivalence.  Therefore, given $A \in Pre_{sAb}$ with sheafification $\tilde A$, the local weak equivalence $A \lra \tilde A$ induces a local weak equivalence $A[1] \lra \tilde A [1]$.  Hence, $\widetilde {A[1]} $ is weakly equivalent to $\tilde A [1]$. From here it is easy to see that $\widetilde {G[n]} $ is weakly equivalent to $((G[n-1])^\sim[1])^\sim$.  The result follows.   
\end{proof}

\subsection{Comparison with other Approaches}

Jardine and Luo have taken a different approach to principal bundles \cite{jar2,jal}, and we now compare their formulations with those of \cite{fss}.  Here we consider only the case where $\CC$ is the site of open sets on a space $X$ so that $X$ is the terminal object in $\CC$. Throughout this section, fix a space $X$ and a sheaf $G$ of simplicial groups on $X$.  Let $G-sPre(\CC)$ denote the simplicial presheaves on $X$ with $G$-action.

\begin{lem} There is a cofibrantly generated closed model structure on the category $G-sPre(\CC)$ of simplicial G-presheaves, where a map is a fibration (resp. weak equivalence) if the underlying map of simplicial presheaves is a global fibration (resp. local weak equivalence).
\end{lem}

\begin{Def} 
\ben Let $G-Tors$ be the category of cofibrant fibrant simplicial G-presheaves $P$ such that $P/G \lra *$ is a hypercover (i.e. local trivial fibration).  We call the objects in $G-Tors$ \emph{G-principal bundles}.  A G-principal bundle $P$ is called a principal bundle over $X = P/G$.  
\i Choose a factorization $\emptyset \lra EG \lra X$ in $G-sPre(\CC)$ where the first map is a cofibration and the second is a trivial fibration.  Let $\bbb G= EG/G$.
\een  
\end{Def}

Note that $EG$ and $\bbb G$ are defined up to weak homotopy equivalence in  $G-sPre(\CC)$.

\begin{lem}\label{delooplem} Any $\bbb G$ is a delooping object of $G$.  
\end{lem}

\begin{proof}
We would like to see that $G$ is a homotopy pullback of 
\[
\begin{CD} 
   @. *\\
@.  @VVV\\
* @>>> \bbb G
\end{CD}
\]
in (some/any) model structure where the weak equivalences are the local weak equivalences. 

In Remark 4 of \cite{jal}, the following observation is made.  Let $WG $ and $\overline WG$ be defined sectionwise.  That is to say, $(WG)(U) = W(G(U))$ and $(\overline WG)(U) = \overline W(G(U))$, where $W$ and $\overline W$ are the standard constructions \cite{goj,cur,kan}.  By taking cofibrant replacements and factorizing maps, we can find objects $EG$ and $\tilde WG$ in  $G-sPre(\CC)$ together with maps $WG \stackrel{p}\lla \tilde WG \stackrel{j}\lra EG$ such that $p$ is a trivial fibration in  $G-sPre(\CC)$, $j$ is a trivial cofibration in $G-sPre(\CC)$, and $\tilde WG$ is cofibrant.  Since $p$ is a weak equivalence of cofibrant simplicial G-spaces, it induces a weak equivalence $\tilde WG / G \lra WG/ G = \overline WG$.  Similarly, $j$ induces a weak equivalence $\tilde WG/ G \lra EG/G = \bbb G$.   We therefore have a sequence of weak equivalences in the diagram category:
from  
\[
\begin{CD} 
   @. \tilde W G\\
@.  @VVV\\
* @>>> \tilde WG / G
\end{CD}
\]
to 
\[
\begin{CD} 
   @. WG\\
@.  @VVV\\
* @>>> \overline WG
\end{CD}
\]
and also to
\[
\begin{CD} 
   @. EG\\
@.  @VVV\\
* @>>> \bbb G,
\end{CD}
\]
which maps to 
\[
\begin{CD} 
   @. *\\
@.  @VVV\\
* @>>> \bbb G.
\end{CD}
\]
by a weak equivalence.  Replacing an object in the diagram category by a weakly equivalent one does not change the homotopy pullback.  Hence, we need only show that $G$ is the homotopy pullback of 
\[
\begin{CD} 
   @. WG\\
@.  @VVV\\
* @>>> \overline WG.
\end{CD}
\]

First we show that the pullback of $* \lra \overline WG \longleftarrow WG$ is in fact a homotopy pullback.  The homotopy pullback is simply the homotopy limit in the model category $Pre_{sSet}^{proj, loc}$.  From the paragraph preceding Proposition \label{jjlis}, we know that to compute the homotopy limit in the local model structure, it is enough to compute it in the global projective model structure.  We know that $Pre_{sSet}^{proj}$ is proper \cite{dhi}.  It is well known that in a right proper model category, the pullback of a diagram $X\lra Z \lla Y$ is the homotopy pullback provided that one of the morphisms $X \lra Z$ or $Y \lra Z$ is a fibration.  In the global projective model structure, $WG \lra \overline WG$ is a fibration, so the homotopy pullback of the original diagram is simply the pullback of $* \lra WG \lla \overline WG$.  \\



The pullback can be taken sectionwise, and the reader can consult \cite{goj} for an exposition of the fact that sectionwise, $G$ is isomorphic to the pullback of 
\[
\begin{CD} 
   @. WG\\
@.  @VVV\\
* @>>> \overline WG.
\end{CD}
\]
\end{proof}

\begin{lem} Any delooping object $B G$ is weakly equivalent to a classifying space.   
\end{lem}
\begin{proof} Take a classifying space $\bbb G$.  By Lemma \ref{delooplem}, $\bbb G$ is a delooping of $G$.  However, delooping is well-defined up to weak equivalence, so since $B G$ and $\bbb G$ are both deloopings, they are isomorphic in the homotopy category.
\end{proof} 

Alternately, a straightforward calculation shows that $G[1]$ is weakly equivalent to $\overline WG$ and hence $\bbb G$.

\section{Appendix}




\subsection{Deligne's Theorem}\label{delignedetails}
\noi Let $Pic_\om ^1 = \{ {A} \in Pic_\om \st {A}  = {A} _1 \}$, and let $Pic$ denote the category of Picard categories in the sense of Deligne \cite{del}.  That is, a Picard category $\CC$ is a quadruple $(\CC , + , \sigma, \tau)$, where $+: \CC \times \CC \lra \CC$ is a functor such that for all objects $x \in \CC$, $x+: \CC \lra \CC$ is an equivalence, and addition is commutative and associative up to isomorphisms $\tau$ and $\sigma$. 
In this section we explain in detail the relationship between $Pic _\om ^1$ and $Pic$.  

\begin{lem}Let $\CC \in Pic$. For any objects $x$,$y \in \CC$, $id_x +id_y = id_{x +y}$ whenever $x$ is isomorphic to $y$.
\end{lem}
\begin{proof}
Take any $g \in Hom_\CC (x,y)$.  Then 
\[ (g \circ id_x ) + (g\inv \circ id _y) = (g + g\inv ) \circ (id_x + id_y) 
\]
\[
g+ g \inv = (g + g \inv)(id_x + id_y) 
\]
\[
id_{x+y} = id_x + id_y
\]
\end{proof}

\begin{lem}\label{13nov1} 
\begin{enumerate} Assume $\CC \in Pic$ such that $+$ is strictly commutative and associative. 

\item For any $f \in Hom_\CC(x,y)$ in $\CC$, $f + f\inv = id_{x+y}$.

\item For any $f\in H_\CC(a,b)$, $g \in Hom_\CC(b,c)$, $id_b + g\circ f = g +f$.

\item Assume that for every $x \in Ob (\CC)$, $x + : \CC \lra \CC$ is actually an isomorphism.  
Then $Ob(\CC)$ is an abelian group. 
 For all $f \in Hom _\CC (x,y)$,  $f + id_0 = f $,
 and there exists a unique $h \in Hom_\CC (-x,-y)$
 such that $f +h = id_0$. Hence, 
$Hom_\CC = \cup _{x,y \in ob(\CC)} (x,y)$ is also an abelian group.  
\end{enumerate}
\end{lem} 

\begin{proof}
\begin{enumerate}
\item $id_{x+y} = id _x + id_y = id_x + f \circ f \inv = (id _x \circ id _x )+(g \circ g\inv) = (id_x +g)\circ (id_x + g\inv)  = (id_x +g)\circ (g\inv + id_x)  = (id_x \circ g \inv ) + (g \circ id_x) = g \inv + g$

\item $g \circ f + id _b = (g \circ f ) + (g \inv \circ g) = (g + g\inv ) \circ (f + g) = id_{x+y} \circ (f+g) = f+g$

\item The first claim is obvious.  Since $0 + 0 = 0$, $id_0 + id _0 = id_0$.  Therefore, for $f \in Hom _\CC (x,y)$, $id_0 + (id_0 + f) = id_0 + f$.  However, since $0+ : \CC \lra \CC$ is an equivalence of categories, it gives a bijection $ Hom_\CC (x,y) \lra Hom_\CC (x,y)$ sending $f \mapsto id_0 + f$.  Since this is injective, $id_0 + f = f$.  We showed that $f + f\inv = id_{x+y}$, so $f + ( f\inv + id_{-x-y} ) = id_{x+y} + id_{-x-y} = id_0$.  We have showed the existence of an additive inverse $h = f\inv + id_{-x-y}$ to $f$.  To show uniqueness.  If $f + g = id_0 = f+h$, $g + id_0 = g + (f +h) = (g + f) + h = id_0 + h $.  Again, since $0+$ is an equivalence, $h = g$. 
\end{enumerate}
\end{proof}






\begin{prop} $Pic ^1 _{strict} $ consists of all small Picard categories in $Pic$ such that $+$ is strictly associative and commutative (i.e. $\tau$ and $\si$ are identities) and for each $x \in ob(\CC)$, $x+ : \CC \lra \CC$ is an isomorphism, not just an equivalence.
\end{prop}
\begin{proof}
Clearly any 1-category in $ Pic ^1 _{strict} \subset Pic$ is a small Picard category satisfying these properties.  On the other hand, if a small Picard category $\CC \in Pic$ satisfies these properties, lemma~\ref{13nov1} shows that $\CC _1 = Hom _\CC$ is actually an abelian group, and sending the inclusion $ob(\CC) \hookrightarrow \CC _1$ ($x \mapsto id_x $) is an inclusion of abelian groups.  Furthermore, lemma~\ref{13nov1} demonstrates that the second property $ f\circ g = f + g - s_0 f$ is satisfied.     
\end{proof}


\subsection{Nerve and Path Functors for $\om$-categories}

We wish to see that $\om$-descent and \v Cech descent are equivalent for $\om$-groupoids generally, not just Picard $\om$-categories.  Integral to our proof for Picard $\om$-categories was a description of the nerve of $A[1]^{a,b}$.  As a step towards extending the proof to $\om$-groupoids, we give a characterization of the nerve of $A[1]^{a,b} $ for any $A \in \om Cat$ and $a,b \in A_0$.  

\begin{prop}\label{naoneprop} Let $A$ be an $\om$-category.  Then 
\[ Hom_{\om Cat} (\OO (\tilde \De ^n)), A[1]^{a,b} ) \simeq \{ f \in Hom_{\om Cat}(\OO (\tilde \De ^n \times \tilde \De ^1 ), A) \st f( \langle \underline u , 0 \rangle) = a , \; f(\langle \underline u , 1 \rangle ) = b \; \textrm{for all} \; \underline u \in \tilde \De ^n \times \tilde \De ^1 \}.
\]  
\end{prop}

\begin{proof} 
Let $C = \tilde \De ^n \times \tilde \De ^1$.  For $(M,P) \in \NN (C)$, let $\Theta ((M,P)) = (M \cap \tilde \De ^n \times \tilde \De ^1 _1 , P \cap \tilde \De ^n \times \tilde \De ^1 _1 ) \in \NN (C)$.  Notice that $\tilde \De ^n \times \tilde \De ^1 _1 = \tilde \De ^n \times \{ (01)\}$.  Recall from \S \ref{paritysection} that $\OO (C)$ is generated by atoms $\langle c \rangle  $ for $c \in C$.  This implies that every element in $\al \in \OO (C)$ is a gotten by a composition of atoms $\al = \langle c_1 \rangle *_{k_1}c_2 *_{k_2}...*_{k_{t-1}}\langle c_t \rangle $ for $c_1,...c_t \in C$.  We omit parentheses in such compositions for generality and ease of exposition.  Viewing $\Theta $ as a map of sets $\Theta : \OO (C) \lra \NN (C)$, for $\al \in \OO (C)$, let $\overline \al$ denote the set $\Theta \inv (\Theta \al)$.  We will prove the following statements:
\ben $\Theta (s_n (M,P)) = s_n (\Theta (M,P))$, and  $\Theta (t_n (M,P)) = t_n (\Theta (M,P))$.
\i If $N$, $Q \subset \tilde \De ^n \times \tilde \De ^1 _0$, then $\Th$ neglects composition with $(N,Q)$, i.e. $\Th ((M,P)*_k (N,Q)) = \Th (M,P)$ and $\Th ((N,Q)*_k (M,P)) = \Th (M,P)$ whenever the compositions are defined.  More generally, for any $(N,Q)$, $(M,P)$, $\Th ((N,Q)*_k (M,P)) = \Th (N,Q) *_k\Th (M,P)$   
\i For $\uu \in \tilde \De ^n$, $\Th \langle \uu , (01) \rangle = \langle \uu \rangle \times \{ (01)\}$, and $s_n(\langle \underline u \rangle \times \{ (01)\}) = s_{n-1} \langle \underline u \rangle \times \{(01)\}$ (similarly for $t_n$).    
\i If $\langle\underline u , (01) \rangle  \in \overline {(\langle \underline x , (01) \rangle )}$, $\langle \underline v , (01) \rangle  \in \overline {(\langle\underline y , (01) \rangle )}$ and we are able to compose $\langle\underline x , (01) \rangle *_k \langle\underline y , (01) \rangle $, then one can form the composition $\langle\underline u \rangle  *_{k-1}\langle\underline v \rangle $ in $\OO (\tilde \De ^n )$.  In this case, $\Th (\langle\underline x , (01) \rangle  *_k \langle\underline y , (01) \rangle  ) = \langle\underline u \rangle *_{k-1} \langle\underline v \rangle  \times \{ (01)\}$.
\een
The proofs of (1)-(3), are essentially based on the observation that the operations involved in taking the source and target and composition respect the decomposition of a subset $S_n \subset C_n = \bigsqcup _{p+q = n} \tilde \De ^n _ p \times \tilde \De ^1 _q$ into $S_n = \bigsqcup _{p+q = n} S_n \cap (\tilde \De ^n _ p \times \tilde \De ^1 _q)$.  To be more precise, the operations consist of replacing a set $M$ by $M_n$, $|M|_n$, or $(M\setminus M_n)$ as well as taking unions.  Let $C_{*,1} = \tilde \De ^n \times \tilde \De ^1 _1$.  Clearly, $(M \cap C_{*,1})_n = M_n \cap C_{*,1}$.  Also, $(P\cup M) \cap C_{*,1} = (P \cap C_{*,1})\cup (M\cap C_{*,1})$.  That $|M \cap C_{*,1}|_n = |M|_n \cap C_{*,1}$ follows from the previous two properties together with the fact that $|M|_n = \bigcup _{k=0} ^n M_k$.  Finally, the first property shows that intersecting with $C_{*,1}$ respects grading so that $(M \setminus M_n)\cap C_{*,1} = (M \cap C_{*,1})\setminus (M_n \cap C_{*,1}) =   (M \cap C_{*,1})\setminus (M \cap C_{*,1})_n$.  Statements (1) - (3) now follow easily. 
\ben $\Th (s_n (M,P)) = \Th (|M|_n , M_n \cup |P|_{n-1})= (|M|_n \cap  \tilde \De ^n \times \tilde \De ^1 _1 ,( M_n \cup |P|_{n-1} ) \cap \tilde \De ^n \times \tilde \De ^1 _1) = (|M \cap \tilde \De ^n \times \tilde \De ^1 _1|_n , (M_n \cap \tilde \De ^n \times \tilde \De ^1 _1) \cup |P \cap \tilde \De ^n \times \tilde \De ^1 _1|_{n-1}) = s_n (\Th (M,P))$.  The proof for $t_n$ is similar.  
\i First suppose that we can form the composition $(N,Q)*_k (M,P) = (M \cup (N \setminus N_n) , (P \setminus P_n)\cup Q)$.  It follows from the observations in the previous paragraph together with the fact that $\Th (N, Q) = (\emptyset, \emptyset)$ that $\Th ((N,Q)*_k (M,P)) = (M \cap C_{*,1} , (P \cap C_{*,1} )\setminus (P_n \cap C_{*,1})$.  In order for $(N,Q)$ and $(M,P)$ to be composable, we require that $s_n (N,Q) = t_n(M,P)$, which implies that $P_n = N_n$, whence $P_n \cap C_{*,1} = \emptyset$.  Therefore, $ \Th ((N,Q)*_k (M,P) = (M \cap C_{*,1} , P \cap C_{*,1} ) = \Th (M,P)$.  Showing that $\Th (M, P)*_n (N,Q) = \Th (M,P)$ follows similarly. More generally, for any $(M,P)$, $(N,Q)$, 
\begin{eqnarray*}
\Th ((N,Q)*_k (M,P))& = &(M \cap C_{*,1} , (P \cap C_{*,1} )\setminus (P_n \cap C_{*,1})\\
                    & = &((M\cup (N \setminus N_n))\cap C_{*,1} , (Q \cup (P \setminus P_n))\cap C_{*,1} \\
                    & = &((M\cap C_{*,1} )\cup (N\setminus N_n )\cap C_{*,1} , (Q \cap C_{*,1})\cup (P \setminus P_n)\cap C_{*,1})\\
                    & = &((M\cap C_{*,1} )\cup (N\cap C_{*,1} \setminus N_n \cap C_{*,1}) , (Q \cap C_{*,1})\cup (P \cap C_{*,1} \setminus P_n\cap C_{*,1}))\\
                    & = & (N\cap C_{*,1} , Q\cap C_{*,1} )*_k(M\cap C_{*,1} , P\cap C_{*,1} )\\
                    & = & \Th (N,Q) *_k \Th (M,P)
\end{eqnarray*}
\i First we verify the claim that for $\underline u \in \tilde \De ^n _r$ and $(M,P) = \Th (\langle  \underline u , (01) \rangle ) \in \NN (C)$, $M_{n-k} = \mu (u) _{n-k-1} \times \{(01)\}$ and $P_{n-k} = \pi (u)_{n-k-1} \times \{ (01) \}$ for $0\leq k \leq n$, where $\mu (z)$, and $\pi (z)$ are as in \S \ref{paritysection}. Let us prove the statement by induction on $k$.  Let $z = (\underline u , (01))$ so that $<z> = (\mu (z) , \pi (z))$.  For $k = 0$,  $\mu (z)_n = \{ (\underline u , (01) \}$, so $\Th (\mu (z))_n = \{ \underline u , (01) \}$.  Now suppose that the claim is true up for all $i\leq k$.  Then $\mu (z)_{n-k} = \mu (u) _{n-k-1} \times \{ (01) \} \cup S_k \times \{ (0) \} \cup T_k \times \{ (1) \}$ for some subsets $S_k$, $T_k \subset \tilde \De ^n _{n-k}$.  Hence,  
\begin{eqnarray*} 
\mu (z) _{n - (k+1)} &=& \mu (z)_{n-k} ^- \setminus \mu (z) _{n-k}^+\\
                    &=& (\mu (\underline u ) ^- _{n-k} \times \{ (01) \} \cup \mu (\underline u ) _{n-k} \times \{ (01) \}^\pm  \cup    S_k ^- \times \{ (0) \} \cup T_k ^- \times \{ (1) \} ) \\
                    & &\setminus  (\mu (\underline u ) ^+ _{n-k} \times \{ (01) \} \cup \mu (\underline u ) _{n-k} \times \{ (01) \}^\pm  \cup    S_k ^+ \times \{ (0) \} \cup T_k ^+ \times \{ (1) \} )\\
                    &=& \mu (\underline u ) _{n-k-1} \times \{ (01) \} \cup S_{k+1} \times \{(0) \} \cup T_{k+1} \times \{ (1) \}
\end{eqnarray*} for some sets $T_{k+1} $, $S_{k+1}$ and where $\pm$ means, that it can be $+$ or $-$ depending on the parity of $n$ and $k$.  Therefore, $\Th (\mu (z))_{n-(k+1)} = \mu (\underline u ) _{n-k-1} \times \{ (01) \}$.  It follows by induction that $\Th (\mu (z) )_{n-k} = \mu ( \underline u )_{n-k-1} \times \{ (01) \} $ for all $0 \leq k \leq n$.  A similar shows that  $\Th (\pi (z) )_{n-k} = \pi ( \underline u )_{n-k-1} \times \{ (01) \} $ for all $0 \leq k \leq n$.  We conclude that $\Th (\mu (z) , \pi (z)) = \langle \uu \rangle \times \{ (01)\}$.  Since $(\langle \underline u \rangle \times \{ (01)\})_n = \langle \underline u \rangle _{n-1} \times  \{ (01)\}$, an easy calculation shows that $s_n (\langle \underline u \rangle \times \{ (01)\}) = (s_{n-1} \langle \underline u \rangle) \times  \{ (01)\}$ and $t_n (\langle \underline u \rangle \times \{ (01)\}) = (t_{n-1} \langle \underline u \rangle) \times  \{ (01)\}$

\i Observe that $\Th (\langle \underline u , (01) \rangle ) = \Th (\langle x , (01) \rangle )$ and $\Th (\langle \underline v , (01) \rangle ) = \Th (\langle y , (01) \rangle )$.  Since $\langle \underline x , (01) \rangle$ and $ \langle y , (01) \rangle$ are composable, $s_k \langle \underline x , (01) \rangle = t_k  \langle y , (01) \rangle$, so by part (1), $s_k \Th (\langle \underline u , (01) \rangle) = t_k \Th (\langle \underline v , (01) \rangle)$ and we can form the composition $\Th (\langle \underline u , (01) \rangle) *_k \Th (\langle \underline v , (01) \rangle)$.  But we just showed that $\Th (\langle \underline u , (01) \rangle) = ( \mu (\underline u ) \times \{ (01) \} , \pi (\underline u ) \times \{ (01) \}) = \langle \underline u \rangle \times \{ (01) \}$ and $\Th (\langle \underline v , (01) \rangle) = ( \mu (\underline v ) \times \{ (01) \} , \pi (\underline v ) \times \{ (01) \}) = \langle \underline v \rangle \times \{ (01) \}$.   It follows that $\langle \underline u \rangle$ and $\langle \underline u \rangle$ are composable since $s_{k-1} \langle \underline u \rangle \times \{(01)\} = s_k(\langle \underline u \rangle \times \{ (01) \}) = s_k \Theta \langle \underline u  , (01) \rangle = t_k \Theta \langle \underline v  , (01) \rangle  = t_k (\langle \underline v \rangle \times \{ (01) \}) = t_{k-1} \langle \underline v \rangle \times \{(01)\}$.  Now, 
\begin{eqnarray*} 
\Th (\langle \underline x , (01) \rangle *_k  \langle \underline y , (01) \rangle) & = & \Th (\langle \underline x , (01) \rangle) *_k \Theta ( \langle \underline y , (01) \rangle) \\
 & = & \Th (\langle \underline u , (01) \rangle) *_k \Th (\langle \underline v , (01) \rangle)\\
  & = & ( \mu (\underline v ) \times \{ (01) \} \cup (\mu (\underline u ) \setminus \mu (\underline u ) _{k-1})\times \{ (01) \}, \pi (\underline u ) \times \{ (01) \} \cup ( \pi (\underline v) \setminus \pi (\underline v ) _{k-1})\times \{(01)\}) \\
   & = &  ( (\mu (\underline v ) \cup (\mu (\underline u ) \setminus \mu (\underline u ) _{k-1}))\times \{ (01) \}, (\pi (\underline u )  \cup ( \pi (\underline v) \setminus \pi (\underline v ) _{k-1}))\times \{(01)\}) \\
&= & (\langle \underline u \rangle *_{k-1} \langle \underline v \rangle ) \times \{ (01) \}.
\end{eqnarray*}  
\een
We are now able to prove the main proposition.  We will use induction to give a bijection 
\[Hom_{\om Cat} (|\OO (\tilde \De ^n )|_k , A[1]^{a,b}) \tilde \lra \HH ^n _k := \{ f \in Hom(|\OO (\tilde \De ^n \times \tilde \De ^1)|_{k+1} , A) \st f(\langle \underline u , (0) \rangle ) = a \; and \; f(\langle \underline u , (1) \rangle ) = a \; \textrm{for all} \;  \underline u \in \tilde \De ^n \}. 
\]
This bijection will send $g$ to the $f$ in $\HH ^n _k$ such that $f(\langle \underline u , (0) \rangle ) = a$,  $f(\langle \underline u , (0) \rangle ) = a$ and $f(\langle \underline u , (01)\rangle = g(\langle \underline u \rangle )$ for all $\underline u \in \tilde \De ^n$, and $f \in \HH ^n _k$ corresponds to $g$ such that $g( \langle \uu \rangle ) = f (\langle \uu , (01) \rangle )$.   We now proceed by induction.  \\

For $k = 0$, a functor $g \in Hom_{\om Cat} (|\OO (\tilde \De ^n )|_0 , A[1]^{a,b})$ consists of a choice of $n+1$ objects $g(0)$,...,$g(n) \in A[1]^{a,b} _0$.  Equivalently, this is a choice of $n+1$ 1-morphisms in $A$ from $a$ to $b$. Since $\OO (\tilde \De  ^n \times \tilde \De ^1)$ is freely generated by its atoms $\langle c \rangle$ for $c \in \tilde \De  ^n \times \tilde \De ^1$, an $f$ in $\HH ^n _1 \subset Hom_{\om Cat} (|\OO (\tilde \De  ^n \times \tilde \De ^1)|_1 , A)$ is freely determined by $f (\langle c \rangle) $ for $c \in  (\tilde \De  ^n \times \tilde \De ^1)_i$, $i = 0,1$ as long as $f$ is compatible with source and target maps $s_0$, $t_0$.  Since we require that $f(\langle \underline u , (0) \rangle )  = a$ and $f (\langle \underline u , (1)\rangle ) = b$ for all $u$, we can freely choose $f(\langle \uu , (01) \rangle ) \in A_1$ as long as $s_0 f(\langle \uu , (01) \rangle ) = a$ and $t_0 f(\langle \uu , (01) \rangle ) = b$ for $\uu \in \tilde \De ^n _0$.  Therefore, a choice of $f \in \HH ^n _1$ is a choice of $(n+1)$ 1-morphisms in $A$ from $a$ to $b$.  It is evident that under this identification, for $\uu \in \tilde \De ^n _0$, $g(\langle \uu \rangle ) = f(\langle \uu , (01) \rangle )$. \\

Now assume that $  Hom_{\om Cat} (|\OO (\tilde \De ^n )|_i , A[1]^{a,b})$ is identified with $\HH ^n _{i}$ as above for all $i \leq k$.  Since $\OO (\tilde \De ^n \times \tilde \De ^1)$ is freely generated by its atoms, a functor $\hat f \in \HH ^n _{k+1} $ is equivalent to a functor $f  \in \HH ^n _{k}$ together with any choice of $(k+2)$-morphisms $\hat f(\langle \uu , (01) \rangle ) $ for $\uu \in \tilde \De ^n _{k+1}$ such that $s_{k+1} f(\langle \uu , (01) \rangle ) = f(s_{k+1} (\langle \uu , (01) \rangle ))$ and $t_{k+1} f(\langle \uu , (01) \rangle ) = f(t_{k+1} (\langle \uu , (01) \rangle ))$.  This is because all $f(\langle \uu , (i) \rangle )$ for $\uu \in \tilde \De ^n _{k+2}$ are forced to be $a$ or $b$.  Also, $\OO (\tilde \De ^n)$ is freely generated by its atoms, so a choice of  $\hat g \in Hom_{\om Cat} (|\OO (\tilde \De ^n) |_{k+1} , A[1]^{a,b})$ is equivalent to a choice of  $ g \in Hom_{\om Cat} (|\OO (\tilde \De ^n )|_{k} , A[1]^{a,b})$ together with a choice of $\hat g \langle \uu \rangle $ for $\uu \in \tilde \De ^n _{k+1}$ with the appropriate $k$-source and target.  \\

Let us then begin with $f \in \HH ^n_{k}$, which corresponds to $g \in Hom_{\om Cat} (|\OO (\tilde \De ^n |_k , A[1]^{a,b})$.  First we show that if $x \in \OO (\tilde \De ^n \times \tilde \De ^1)$ is not in the sub-$\om$-category generated by elements of the form $\langle \uu , (i) \rangle$, then $f(x) =g p_1 \Th x$, where $p_1$ denotes projection onto the first factor. Choose such an $x \in \OO (\tilde \De ^n \times \tilde \De ^1)$. Since $\OO (\tilde \De ^n \times \tilde \De ^1)$ is freely generated by its atoms, it is also generated by its atoms, so is a composition of atoms: $x = \langle x_1 \rangle *_{l_1}\langle x_2 \rangle *_{l_2}...*_{l_{e-1}}\langle x_e \rangle  $ for some $x_1 , ...,x_e \in \tilde \De ^n \times \tilde \De ^1$ (omitting parentheses).  Since $f (\langle \uu , (i) \rangle )$ is a 0-object in $A$ for $i = 0,1$, the value of $f$ is not affected by composition with elements of the form $\langle \uu , (i) \rangle $.  Similarly, $\Th$ also neglects composition with these elements. Let $x_{i_1}$,...$x_{i_r}$ be the ones of the form $x_{i_k} = \langle \underline u_{i_k} , (01)\rangle $, so 
\begin{eqnarray*} f(x) &=& f\langle x_1 \rangle *_{l_1}f\langle x_2 \rangle *_{l_2}...*_{l_{e-1}}\langle fx_e \rangle\\
                       &=& f\langle x_{i_1} \rangle *_{l_{i1}}f\langle x_{i_2} \rangle *_{l_{i2}}...*_{l_{ir}}f\langle x_{i_r} \rangle\\
                       &=& f\langle u_{i_1}, (01) \rangle *_{l_{i1}}f\langle u_{i_2} , (01) \rangle *_{l_{i2}}...*_{l_{ir}}f\langle u_{i_r}, (01) \rangle, whereas
\end{eqnarray*} 
\begin{eqnarray*} 
\Theta(x) &=& \Theta \langle x_1 \rangle *_{l_1}\theta \langle x_2 \rangle *_{l_2}...*_{l_{e-1}}\langle \Theta x_e \rangle\\
                       &=& \Theta \langle x_{i_1} \rangle *_{l_{i1}} \Theta \langle x_{i_2} \rangle *_{l_{i2}}...*_{l_{ir}} \Theta \langle x_{i_r} \rangle\\
                       &=& \Theta \langle u_{i_1}, (01) \rangle *_{l_{i1}} \Theta \langle u_{i_2} , (01) \rangle *_{l_{i2}}...*_{l_{ir}} \Theta \langle u_{i_r}, (01) \rangle\\
                       &=&  (\langle u_{i_1} \rangle *_{l_{i1}-1}  \langle u_{i_2 }  \rangle *_{l_{i2}-1}...*_{l_{ir}-1} \langle u_{i_r} \rangle)\times \{(01)\}.
\end{eqnarray*}
Therefore, $gp_1 \Theta x = g \langle u_{i_1} \rangle *_{l_{i1}-1}  \langle g u_{i_2 }  \rangle *_{l_{i2}-1}...*_{l_{ir}-1} \langle g u_{i_r} \rangle)$, and 
$f(x) = f\langle u_{i_1}, (01) \rangle *_{l_{i1}}f\langle u_{i_2} , (01) \rangle *_{l_{i2}}...*_{l_{ir}}f\langle u_{i_r}, (01) \rangle$.  By induction hypothesis, $f\langle u_{i_j}, (01) \rangle = g \langle u_{i_j} \rangle$ for each $j$, and since the compositions appearing in $gp_1 \Theta x$ are in $A[1]^{a,b}$, $*[1]_n = *_{n+1}$ for any $n$ so that the compositions coincide too.  \\

Now, for  $\uu \in \tilde \De ^n _{k+1}$, $f(t_{k+1} (\langle \uu , (01) \rangle )) = gp_1 \theta t_{k+1} \langle \uu , (01) \rangle =  gp_1 t_{k+1} \theta \langle \uu , (01) \rangle = gp_1 t_{k+1} (\langle \uu \rangle \times \{ (01)\} = gp_1 t_k \langle \uu \rangle \times \{(01)\} = g(t_k \langle \uu \rangle ) = t[1]_k g(\langle \uu \rangle) = t_{k+1} g(\langle \uu \rangle) $.  Similarly, $f(s_{k+1} (\langle \uu , (01) \rangle ))  = s_{k+1} g(\langle \uu \rangle)$.  In summary, an extension $\hat f \in \HH ^n _{k+1}$ of $f$ is equivalent to a choice of $(k+2)$-morphisms $\hat f(\langle \uu , (01) \rangle ) $ for $\uu \in \tilde \De ^n _{k+1}$ such that $s_{k+1} f\langle \uu , (01) \rangle = s_{k+1} g(\langle \uu \rangle)$ and $t_{k+1} f\langle \uu , (01) \rangle = s_{k+1} g(\langle \uu \rangle)$.  \\

In comparison, an extension $\hat g$ of $g$ consists of any choice of $(k+1)$-morphisms $g\langle \uu \rangle  \in A[1]^{a,b}$ for $\uu \in \tilde \De ^n _{k+1}$ such that $s[1]_{k}\hat g \langle \uu \rangle = g s_k \langle \uu \rangle $ and $t[1]_{k}\hat g \langle \uu \rangle = g t_k \langle \uu \rangle$.  Since $g(s_k \langle \uu \rangle )
 = s[1]_k g \langle \uu \rangle  = s_{k+1} g \langle \uu \rangle$ and $g(t_k \langle \uu \rangle ) 
= t_{k+1} g \langle \uu \rangle$, an extension $\hat g$ of $g$ is the same as a choice of $(k+2)$-morphisms $\hat g \langle \uu \rangle  \in A_{k+2}$ such that  $s_{k+1}\hat g \langle \uu \rangle = 
s_{k+1}g \langle \uu \rangle  $ and  $t{k+1}\hat g \langle \uu \rangle = t_{k+1} g \langle \uu \rangle $. Note that this implies that $s_0 \hat g \langle \uu \rangle = a $ and $t_0 \hat g \langle \uu \rangle = b$.  It is now evident that the choice for an extension $\hat f$ of $f$ is equivalent to a choice of an extension $\hat g$ of $g$.  It follows that $\HH ^n_{k+2}$ is in bijection with $Hom_{\om Cat} (|\OO (\tilde \De ^n))|_{k+1}, A[1]^{a,b} )$, thus completing our proof by induction.  
\end{proof}

\begin{cor} For $A \in \om Cat$, 
\ben $NA[1]^{a,b}_n \simeq \{ f \in  Hom_{Cs} (\De ^n \otimes \De ^1 , \fN A) \st f_{|\De ^n \times 0 } =a , \; f_{|\De ^n \times 1 } =b\}$.  
\i If $A$ is an $\om$-groupoid, there is a weak homotopy equivalence $N(A[1]^{a,b}) \lra Hom^L_{NA} (a,b)$. 
\een
\end{cor}
\begin{proof} 
\ben Let $\De ^n \otimes \De ^1$ denote the complicial set with underlying simplicial set $\De ^n \times \De ^1$ and for which the thin r-simplices are $(x,y) \in \De ^n _r \times \De ^1 _r $ such that for some $i \leq j$ in $[r]$, $x$ is degenerate at $i$ and $y$ is degenerate at $j$.  In \cite{ver}, Verity proves that there is an isomorphism of $\om$-categories $\overline c ^{n,1} : \FF _\om (\De ^n \otimes \De ^1) \lra \OO (\tilde \De ^n \times \tilde \De ^1)$.  Since $\FF _\om$ is left adjoint to the nerve functor $\fN : \om Cat \lra Cs$, $Hom_{\om Cat}(\OO (\tilde \De ^n \times \tilde \De ^1), A) \simeq Hom_{\om Cat } ( \FF _\om (\De ^n \otimes \De ^1) , A) \simeq Hom_{Cs}(\De ^n \otimes \De ^1 , \fN A) $.\\

\noi By Proposition \ref{naoneprop}, the only thing left to prove is that  $\{ f \in Hom_{\om Cat}(\OO (\tilde \De ^n \times \De ^1) , A) \st f( \langle \underline u , 0 \rangle) = a , \; f(\langle \underline u , 1 \rangle ) = b \; \textrm{for all} \; \underline u \in \tilde \De ^n \times \tilde \De ^1 \} \subset Hom_{\om Cat}(\OO (\tilde \De ^n \times \De ^1) , A)$ corresponds to $Hom_{Cs} (\De ^n \otimes \De ^1 , NA) \st f_{|\De ^n \times \{ 0\}} = a , \; f_{|\De ^n \times \{ 0\}} = b \}$ under the isomorphism $Hom_{\om Cat}(\OO (\tilde \De ^n \times \tilde \De ^1), A) \simeq Hom_{Cs}(\De ^n \otimes \De ^1 , \fN A)$.  To describe this isomorphism, we will first need the following two facts, which can be found in \cite{ver}.  One can take products of maps of parity complexes, so given morphisms $[r] \stackrel{\phi}\lra[n]$, $[s] \stackrel{\psi}\lra [m]$, there is a morphism $\tilde \De (\phi ) \times \tilde \De (\psi ): \tilde \De ^r \times \tilde \De ^l \lra \tilde \De ^n \times \tilde \De ^m$.  In fact, $\OO (\tilde \De (\phi)\times \tilde \De (\psi)) ( \langle \uu , \vv \rangle ) = \langle \phi \uu , \psi \vv \rangle $.  Secondly, there is a morphism $\nabla _r : \tilde \De ^r \lra \tilde \De ^r$ of parity complexes, which sends $\vv = (v_0v_1..v_r) $ to $\{ (v_0...v_s, v_s...v_r) \in \tilde \De ^r _s \times \tilde \De ^r _{r-s} \st s = 0,1...,r \}$. Now, given $f \in Hom_{\om Cat}(\OO (\tilde \De ^n \times \tilde \De ^1), A) $, we describe the corresponding $F \in  Hom_{Cs}(\De ^n \otimes \De ^1 , \fN A)$.  Given an r-simplex $(\al , \be) \in \De ^n _r \times \De ^1 _r$, which we think of as a pair $([r] \stackrel{\al}\lra [n] , [r]\stackrel{\be}\lra [1])$, $F(\al, \be )$ is the composition $\OO (\tilde \De ^r ) \stackrel { \OO (\nabla _r )}\lra \OO (\tilde \De ^r \times \tilde \De ^r) \stackrel {\OO (\tilde \De (\al) \times \tilde \De (\be))}\lra \OO (\tilde \De ^n \times \tilde \De ^1) \stackrel {f}\lra A$ (an r-simplex in $\fN A$).  By the universal property of $F _\om$, $f\circ \OO (\tilde \De (\al) \times \tilde \De (\be)) \circ \OO (\nabla _r ) = f \circ c^{n,1} \circ \io _{(\al , \be)}$.  Additionally, $\FF _\om (\De ^n \otimes \De ^1) $ is, by definition \cite{ver}, a quotient $q: F_\om (\De ^n \times \De ^1) \lra \FF _\om (\De ^n \otimes \De ^1)$ of  $F_\om (\De ^n \times \De ^1)$, and $\overline c^{n,1} \circ q = c^{n,1}$.\\

Now we will show that $f$ satisfies $f(\langle \uu , (0)\rangle) = a$ and $f(\langle \uu , (1) \rangle ) = b$ for all $\uu \in \tilde \De ^n$ if and only if $F(x,(0)) = a$ and $F(x,(1)) = b$ for all $r$ and all r-simplices $x \in \De ^n$.  For $i \in \{0 , 1\}$, we write $(i)$ as shorthand for the r-simplex $(ii...i)$ with $i$ listed r times.  Suppose $f$ satisfies this condition.  To show that $F$ has the desired property, it is enough to verify the statement for non-degenerate simplices $(x,(i))$ since if it is true for non-degenerate simplices, then $F(\sigma _k(x,(i))= \sigma _k F(x,(i)) = \sigma _k a = a$, where $\sigma _k$ is the k-th degeneracy map.  Since it is enough to verify the statement for non-degenerate simplices, we may also assume that $x$ is non-degenerate, since $x$ is degenerate if and only if $(x,(i))$ is degenerate. \\

We know that $F(\al , \be ) = f \circ c^{n,1}\circ \io _{(\al , \be)}$. Let $\al = (u_0u_1...u_r)$ non-degenerate and $\be = (i)$.  By definition, $\io _{(\al , \be)} (\langle 01...r\rangle ) = [\![ \al , \be ]\!]$, so $F(\al, \be) \langle 01...r \rangle = f \circ c^{n,1}([\![ \al , \be ]\!])$, which equals $f(\langle \underline u , (i) \rangle)$ by Theorem 255 of \cite{ver}.  Hence, $F(\al , \be)\langle 01...r \rangle = a$ if $i=0$ or equals $b$ if $i=1$.  For the general case, we show that $F(\al , \be ) \langle \vv \rangle = a$ or $b$.  A k-dimensional $\vv \in \td ^r _k$ corresponds to a strictly increasing morphism $[k]\stackrel{\phi}\lra [r]$.  We know that $\io _{(\al , \be)} \circ \OO (\td (\phi)) = \io _{\phi ^* (\al , \be)}$, whence $\io _{(\al , \be)} \langle \vv \rangle = \io _{(\al , \be)} \circ \OO (\td (\phi)) \langle 01...k \rangle = \io _{\phi ^* (\al , \be)} \langle 01...k \rangle = [\![ \phi ^* a , \phi ^* \be ]\!] =  [\![ \phi ^* \al , (i) ]\!]$.  Then, $F(\al , \be) \langle \vv \rangle = f\circ c^{n,1} \io _{(\al , \be)} \langle \vv \rangle = f \circ c^{n,1} ( [\![ \phi ^* \al , (i) ]\!])= f (\langle \uu , (i) \rangle )$ (again by Theorem 255 of \cite{ver}), which is equal to $a$ if $i = 0$ or $b$ if $i=1$.  Thus, $F$ is as desired.  This argument can be run backwards to show that if $F(\al , (0)) = a$, $F(\al , (1)) = b$ for all $\al \in \De ^n$, then $f(\langle \uu , (0) \rangle ) = a$ and $f(\langle \uu , (1) \rangle ) = b$ for all $\uu \in \td ^n$.  

\i $\De ^{n+1} $ is a retract of $\De ^n \times \De ^1$, with the inclusion $\De ^{n+1} \hookrightarrow \De^{n} \times \De ^1$ given by $(0,1,...,n+1)\mapsto (012...nn, 0....01)$ and $p: \De ^n \times \De ^1 \lra \De ^{n+1}$ chosen to send thin simplices to degenerate ones.  
\een
\end{proof}

\subsection{Homotopies}\label{homotopy}

In this section we make some basic observations about homotopies between morphisms of chain complexes from the perspective of $\om$-categories. 

\subsubsection{Homotopies}
Using the equivalence of $\cab$ with $Pic _\om$, we see that homotopies of maps of complexes correspond to the following notion of homotopy between maps $F$, $G : A \lra {B}$ in $Pic _\om$.  A homotopy consists of maps $H^n : A _n \lra {B} _{n+1}$ such that $H^n (A _{n-1} ) \subset {B} _n  $ and $(t_n - s_n )H^n + H^{n-1}(t_{n-1} - s_{n-1})$ agrees with $G - F$ when we pass to the quotient $A _n / A _{n-1} \lra {B} _n / {B} _{n-1}$.

\begin{lem}\label{hhomotoplem} For complexes $A$, $B \in \cab$, there is a $\CC \in \spics $ such that $\CC _0 = \{$ maps of complexes $A \lra B$ $\}$ and $Hom ^1 (F,G) = \{$ homotopies from $F$ to $G$ $\}$.
\end{lem}
\begin{proof} 
For $i>0$, let $H^i$ consist of all maps from $A \lra B[-i]$.  By this we mean $h \in H^i$ consists of a sequence of maps $h_k : A_k \lra B_{k-i}$ but not necessarily commuting with $d$.  Define $H^0 = Hom_\cab (A,B)$.
  Now define a differential $\partial : H^i \lra H^{i-1} $ by $\partial f = df + (-1)^{i-1} fd$. 
 This makes $H^*$ into a complex of abelian groups.  $\tilde H$ is the desired $\om$-category.
\end{proof}

Lemma \ref{hhomotoplem} gives the following corollary, an observation also made by Street \cite{str4}. 
 
\begin{cor} There is an $\om$ category, $A$ such that $A _0 = \cab$, $A _1 = \{$ maps of complexes $\}$, and $A _1 = \{$ homotopies of maps of complexes $\}$. 
\end{cor}

For any $\om$-category in abelian groups $A$, there is a group homomorphism $D : A \lra A$ given by $D = \sum _{n \geq 0} t_n - s_n$.  Since we assume that $A$ is a union of the $A _n$, this sum makes sense because it is a finite sum on $A _n$.  Now we can efficiently define a homtopy between two functors $F,G \in Hom_{Pic _\om} (A , {B})$ by reformulating the description at the beginning of this section.  

\begin{Def} Let $A , B \in \spics$ and  $F,G \in Hom_{Pic _\om} (A , B)$.  We say that a group homomorphism $H : A \lra B$ is a \emph{homotopy} if
\begin{enumerate}
\item  $H (A _n ) \subset B _{n+1}$, 
\item  $DH + HD  = G - F$, and 
\item $s_n H (s_n - s_{n-1}) = 0 $ for all $n \geq 0$. 
\end{enumerate}
\end{Def}

\begin{lem}\label{30mar2} Let $f,g : A \lra B$ be maps in $\cab$.  A homotopy $h : f \lra g$ is equivalent to a homotopy $H $ from $Pf$ to $P g$ in $Pic _\om$.  
\end{lem} 

\begin{proof} Since $P A = \oplus K_i \simeq A^i $ $K_i = \si (A^i)$ 
as in \S \ref{17nov2} and $P B = \oplus L_i \simeq B^i$, a homomorphism $H : P A \lra P B$ is determined by $H_i : K_i  \lra P B$.
  And $s_n H (s_n - s_{n-1})$ 
for all $n$ if and only if 
$H_i (K_i ) \subset L_{i+1}$.
  Then for
 $x = ((0,0),...(0,dx_n),(x_n,x_n),(0,0),...) \in K_n$, $Hx =  ((0,0),...(0,dy_{n+1}),(y_{n+1},y_{n+1},(0,0),...) \in L_{n+1}$.
  Defining $h : A^n \lra B^{n+1} $
 by $h = \pi _{n+1} H \si $
 gives a bijection between maps 
$h : A \lra B[1]$ (not necessarily commuting with the differential $d$) and homomorphisms 
$H : PA \lra P B$ such that $s_n H (s_n - s_{n-1})$ and $H(P A _n ) \subset 
P B _{n+1}$.
  A direct computation shows that with 
$x = ((0,0),...(0,dx_n),(x_n,x_n),(0,0),...) \in K_n$, $(DH + HD)x =  ((0,0),...(0,(hd + dh)dx_n),((dh + hd)x_n,(dh + hd)x_n),(0,0),...) \in L_n$. 
 Hence $HD + DH = P g - P f $ 
if and only if $hd + dh = g-f$. 
\end{proof}

 \begin{rem} ${P} A$ is a projective abelian group if and only if each $A^i $ is projective.
 \end{rem}


\begin{prop} Let $\PP$ denote the full subcategory of $\spics $ consisting of objects which are projective abelian groups, and let $\overline \PP$ denote the category the objects of which are $ob(\PP)$ and the morphisms of which are morphisms in $\PP$ modulo homotopy.  There is an equivalence of categories $D^{\leq 0} (A b ) \lra \overline \PP$.  
\end{prop} 

\begin{proof} Let $K^{\leq 0} (Proj)$ denote the homotopy category of complexes of projective abelian groups in degrees $\leq 0$.  Lemma 
 \ref{30mar2} shows that $\HH : K^{\leq 0} (Proj) \lra \overline \PP$ is an equivalence of categories.  The theorem now follows since $K^{\leq 0} (Proj) \lra D^{\leq 0} (A b)$ is an equivalence. 
\end{proof}

\subsubsection{Homotopies from the Perspective of $\om$-Categories}
The homotopies described above are algebraic in nature, but we can still define homotopies for ordinary $\om$-categories.  The following definition extends the concept of homotopy for from $Pic_\om$ to $\om$-cat. 

\begin{Def}\label{30mar3} Let $f,g: A \lra B$ be two functors of $\om$-categories.  A \emph{homotopy} $H : F \lra G$ is a map $H : A \lra B$ with $H(A _n ) \subset B _{n+1}$, and for $x \in \ca _n$, $Hx$ is an $n+1$ isomorphism
\[ Ht_{n-1}x *_{n-1}(Ht_{n-2}x *_{n-2}(...(Ht_1x *_1 (Ht_0 x *_0 fx))..) \stackrel {Hx}{\lra _{n+1}} (...((gx*_0 Hs_0x)*_1 Hs_1x )*_2...)*_{n-1} Hs_{n-1}x \]
\end{Def}

\noi Intuitively, a homotopy $H : f \lra g$ looks like a ``natural transformation'' from $f$ to $g$, except any diagram of $n$-morphisms which should commute only commutes up to an $(n+1)$-isomorphism specified by $H$.   

\begin{prop} Let $f, g : A \lra B$ be maps of complexes in $\cab$.  A homotopy $h : f \lra g$ defines a homotopy $H$ in the sense of definition \ref{30mar3}.  
\end{prop}
\begin{proof} Having defined $H$ on $\tilde A _i$ for $i < n$, we define $H$ on $\tilde A _n$.  
Let $x = ((x_0 ^- , x_0 ^+ ),...,(x_{n-1} ^- , x_{n-1} ^+ ), (x_n , x_n ),(0,0),...) \in \tilde A _n$.  The chain homotopy formula $d(hx_n) = (gx_n + hx_{n-1}^- ) - ( fx_n + hx_{n-1} ^+)$ implies that $hx_n$ is an $(n+1)$-morphism
\[ Ht_{n-1}x *_{n-1}(Ht_{n-2}x *_{n-2}(...(Ht_1x *_1 (Ht_0 x *_0 fx))..) \stackrel {Hx}{\lra _{n+1}} (...((gx*_0 Hs_0x)*_1 Hs_1x )*_2...)*_{n-1} Hs_{n-1}x \]
as required, so $(Hx)_{n+1} ^\pm = hx_n$, and $(Hx)_i ^\al $ is determined by $H _{|\tilde A _{n-1}}$ for $i <n$. 
\end{proof}

If we prefer in Definition \ref{30mar3}, we can require that a homotopy $H$ satisfies some compatibility with compositions, eg. $(gy *_0 hx)*_1(hy*_fx) = h(y*_0x)$, etc.



\end{document}